\documentclass{amsart}
\usepackage{amssymb, graphics, color, enumitem, mathrsfs}
\usepackage[all]{xy}
\usepackage{appendix}

 \usepackage[applemac]{inputenc}
\usepackage[cyr]{aeguill}
\usepackage[francais]{babel}

\usepackage[margin = 1in]{geometry}

\newtheorem{lemme}{Lemme}[section]
\newtheorem{proposition}[lemme]{Proposition}
\newtheorem{corollaire}[lemme]{Corollaire}
\newtheorem{theoreme}[lemme]{Théorème}

\newtheorem{exemple}[lemme]{Exemple}
\newtheorem{definition}[lemme]{Définition}
\newtheorem{remarque}[lemme]{Remarque}

\newtheorem*{remerciements}{Remerciements}

\newtheorem{hypothese}[lemme]{Hypothèse}

\newcommand\cf{cf\@. }

\newcommand\eg{e\@.g\@. }

\newcommand\pa{ \partial}

\newcommand\bbC{\mathbb C}

\newcommand\bbH{\mathbb H}

\newcommand\bbP{\mathbb P}

\newcommand\bbR{\mathbb R}
\newcommand\bbS{\mathbb S}

\newcommand\bbZ{\mathbb Z}

\renewcommand\Im{\operatorname{Im}}

\usepackage{color}

\newcommand{\lrp}[1]{\left( {#1} \right)}

\newcommand\bX{\overline{X}}

\newcommand\hB{\widehat{B}}
\newcommand\hX{\widehat{X}}
\newcommand\hW{\widehat{W}}
\newcommand\hH{\widehat{H}}
\newcommand\hS{\widehat{S}}
\newcommand\hE{\widehat{E}}
\newcommand\hF{\widehat{F}}

\newcommand\CI{\mathcal{C}^{\infty}}

\newcommand\Rm{\operatorname{Rm}}

\newcommand\cC{\mathcal{C}}
\newcommand\cA{\mathcal{A}}

\newcommand\cF{\mathcal{F}}
\newcommand\cL{\mathcal{L}}

\newcommand\tH{\widetilde{H}}
\newcommand\cV{\mathcal{V}}
\newcommand\cU{\mathcal{U}}
\newcommand\cI{\mathcal{I}}
\newcommand\cJ{\mathcal{J}}

\newcommand\hxi{\widehat{\xi}}

\newcommand\txi{\widetilde{\xi}}

\newcommand\pr{\operatorname{pr}}

\newcommand\phg{\operatorname{phg}}

\newcommand\Id{\operatorname{Id}}

\newcommand\cH{\mathcal{H}}

\newcommand\hM{\widehat{M}}

\newcommand\hphi{\hat{\phi}}

\newcommand\cM{\mathcal{M}}

\newcommand\AC{\operatorname{AC}}

\newcommand\SU{\operatorname{SU}}

\newcommand\Sp{\operatorname{Sp}}

\newcommand\QAC{\operatorname{QAC}}

\newcommand\Qb{\operatorname{Qb}}
\newcommand\nQb{\operatorname{\gp Qb}}
\newcommand\QFB{\operatorname{QFB}}
\newcommand\nQFB{\operatorname{\gp QFB}}
\newcommand\nQAC{\operatorname{\gp QAC}}

\newcommand\ALE{\operatorname{ALE}}

\newcommand\vol{\operatorname{vol}}

\newcommand\tW{\widetilde{W}}

\newcommand\cK{\mathcal{K}}

\newcommand\GL{\operatorname{GL}}

\newcommand\cS{\mathcal{S}}

\newcommand\cB{\mathcal{B}}

\newcommand\TN{\operatorname{TN}}

\newcommand\hV{\widehat{V}}
\newcommand\cW{\mathcal{W}}
\newcommand\cT{\mathcal{T}}

\newcommand\bbT{\mathbb{T}}

\newcommand\diag{\operatorname{diag}}

\newcommand\cO{\mathcal{O}}
\newcommand\gn{\mathfrak{n}}
\newcommand\gp{\mathfrak{p}}
\newcommand\gt{\mathfrak{t}}

\begin{document}
\title[G\'eom\'etrie \`a l'infini des vari\'et\'es hyperk\"ahl\'eriennes toriques]
{G\'eom\'etrie \`a l'infini des vari\'et\'es hyperk\"ahl\'eriennes toriques}

\author{Fr\'ed\'eric Rochon}

\maketitle

\begin{abstract}
On montre que les métriques hyperkählériennes toriques simplement connexes de type topologique fini ayant une croissance du volume maximale sont génériquement quasi-asymptotiquement coniques, ce qui permet de calculer explicitement leurs groupes de cohomologie $L^2$ réduite.  Dans le cas asymptotiquement conique, on donne aussi une description fine de la géométrie à l'infini de leurs déformations de Taub-NUT d'ordre 1 en termes d'une compactification par une variété à coins, ce qui permet d'établir que ces déformations sont à géométrie bornée, d'obtenir des estimations sur leur courbure et la croissance de leur volume et d'identifier de façon unique leur cône tangent à l'infini.  Dans plusieurs exemples, la dimension de ce cône tangent est strictement plus petite que l'ordre de la croissance du volume.  Enfin, nos méthodes montrent que les déformations de Taub-NUT d'ordre maximal de la métrique euclidienne sont quasi-fibrées au bord, ce qui permet entre autres d'identifier de façon unique leur cône tangent à l'infini, de calculer leurs groupes de cohomologie $L^2$ réduite et  de démontrer  la conjecture de $S$-dualité de Sen pour les monopôles centrés de type $(1,1,\ldots,1)$.
\end{abstract}

\tableofcontents

\numberwithin{equation}{section}

\section{Introduction}

Une variété hyperkählérienne de dimension $4n$ est une variété riemannienne ayant pour groupe d'holonomie un sous-groupe du groupe symplectique (compact) $\Sp(n)=\operatorname{O}(4n)\cap \GL(n,\bbH)$.  Une telle variété admet des structures complexes parallèles $I_1,I_2$ et $I_3$ satisfaisant à la relation quaternionique $I_1 I_2=I_3$.  Elle est en particulier kählérienne par rapport à chacune de ces structures complexes.  Depuis l'introduction de la terminologie par Calabi \cite{Calabi}, les variétés hyperkählériennes ont été l'objet d'études intenses.  Plusieurs espaces de modules sont en fait canoniquement des variétés hyperkählériennes, notamment le revêtement universel de l'espace de modules réduit des monopôles $\SU(2)$ de charge magnétique $k$ sur $\bbR^3$, l'espace de modules des instantons sur $\bbR^4$ ou encore l'espace de modules des fibrés de Higgs sur une surface de Riemann.

En dimension $4$, $\Sp(1)=\SU(2)$ et la notion de variété hyperkählérienne coïncide avec celle de variété de Calabi-Yau.  Par le théorème de Yau \cite{Yau1978}, les exemples compacts correspondent donc dans ce cas aux surfaces kählériennes compactes ayant une première classe de Chern nulle.  Toujours en dimension $4$, il y a aussi beaucoup d'exemples non compacts appelés instantons gravitationnels lorsque le tenseur de courbure est de carré intégrable.  Leur géométrie à l'infini est maintenant relativement bien comprise.  Lorsque la croissance du volume est maximale, ce sont des instantons gravitationnels asymptotiquement localement euclidiens (ALE), c'est-à-dire modelés asymptotiquement à l'infini par $\bbC^2/\Gamma$ pour un sous-groupe discret $\Gamma$ de $\SU(2)$, en quel cas une classification complète a été obtenue par Kronheimer \cite{Kronheimer1989}.  Plus généralement, suite entre autres à une conjecture de Cherkis et Kaputsin \cite{Etesi-Jardim} et des travaux de Minerbe \cite{Minerbe2010,Minerbe2011} et Hein \cite{Hein2012}, une classification complète a été obtenue par Chen-Chen \cite{ChenChen1,ChenChen2,ChenChen3} lorsque le tenseur de courbure décroît plus rapidement que quadratiquement, donnant lieu à quatre types de géométrie à l'infini en fonction de la croissance du volume: ALE lorsque la croissance est maximale, ALF (pour \emph{asymptotically locally flat} en anglais) lorsqu'elle est cubique, ALG si elle est quadratique et ALH si elle est linéaire.  

En dimensions plus grandes, la situation est assez différente.  D'abord, être une variété hyperkählérienne est une condition beaucoup plus stricte que d'être une variété de Calabi-Yau.  Le théorème de Yau \cite{Yau1978} peut malgré tout être utilisé pour obtenir des exemples de variétés hyperkählériennes compactes, voir entre autres les résultats de Beauville \cite{Beauville}.  Dans le cadre non compact, beaucoup d'exemples peuvent être obtenus grâce à la méthode du quotient hyperkählérien introduite dans \cite{HKLR1987}.  Cette méthode est au cœur de la classification des instantons gravitationnels ALE de Kronheimer \cite{Kronheimer1989}.  Elle est aussi centrale dans la notion de variété de carquois de Nakajima \cite{Nakajima1994}, en quelque sorte une généralisation des instantons gravitationnels ALE en dimensions plus grandes, ainsi que dans la construction de variétés hyperkählériennes toriques \cite{PP1988,Gotto1992,GRG1997,BD,Bielawski}.  Les variétés d'arc (\emph{bow varieties} en anglais) de Cherkis \cite{Cherkis2011}, une généralisation des variétés de carquois, sont aussi définies en termes d'un quotient hyperkählérien.  

Comme on commence à peine à comprendre la géométrie à l'infini de ces métriques, une classification éventuelle semble encore bien lointaine.  Motivée par une conjecture de Vafa et Witten \cite{Vafa-Witten} sur la cohomologie $L^2$ réduite des variétés de carquois, une série de travaux \cite{Carron2011, Melrose_London, DR} a réussi à identifier précisément la géométrie à l'infini des variétés de carquois en montrant qu'elles sont (génériquement) quasi-asymptotiquement coniques ($\QAC$) au sens de Degeratu-Mazzeo \cite{DM2018}.  De même, motivés par la conjecture de Sen \cite{Sen}, les travaux \cite{KS,FKS} indiquent qu'on est en voie de montrer que la métrique hyperkählérienne sur le revêtement universel de l'espace de modules réduit des monopôles $\SU(2)$ de charge $k$ sur $\bbR^3$ est quasi-fibrée au bord ($\QFB$) au sens de \cite{CDR}.

Continuant sur cette lancée, le but du présent article est de déterminer la géométrie à l'infini de certaines variétés hyperkählériennes toriques complètes simplement connexes de type topologique fini.  Lorsqu'une telle variété a une croissance du volume maximale, on sait par la classification de Bielawski \cite[Theorem~1]{Bielawski} qu'elle correspond à un quotient hyperkählérien
\begin{equation}
   M_{\zeta}= \mu_N^{-1}(-\zeta)/N
\label{int.1}\end{equation}
de $\bbH^d$ par un sous-tore $N$ du tore diagonal $T^d$ dans $\Sp(d)$, où $\zeta\in \bbR^3\otimes \gn^*$ avec $\gn^*$ le dual de l'algèbre de Lie $\gn$ de $N$ et où $\mu_N: \bbH^d\to \bbR^3\otimes \gn^*$ est une application moment hyperhamiltonienne pour l'action de $N$ sur $\bbH^d$.  En adoptant l'approche de \cite{Melrose_London,DR} à ce cadre, on obtient le résultat suivant (voir le Corollaire~\ref{qac.19} ci-bas pour une description plus détaillée).
\begin{theoreme}
Pour $\zeta\in \bbR^3\otimes \gn^*$ générique, la métrique hyperkählérienne du quotient \eqref{int.1} est $\QAC$.
\label{int.2}\end{theoreme}

En particulier, comme par \cite{CDR} les métriques $\QAC$ sont associées à une structure de Lie à l'infini au sens de \cite{ALN04}, on sait par \cite{ALN04,Bui} qu'elles sont automatiquement complètes à géométrie bornée.  Une description plus fine de la structure $\QAC$ montre aussi (Corollaire~\ref{cone.2} ci-bas) que le cône tangent à l'infini du quotient \eqref{int.2} est le quotient singulier $M_0$.  Le Théorème~\ref{int.2} nous permet aussi de calculer la cohomologie $L^2$ réduite comme suit (voir aussi le Corollaire~\ref{qac.20} ci-bas).
\begin{corollaire}
Pour $\zeta\in \bbR^3\otimes \gn^*$ générique, la cohomologie $L^2$ réduite du quotient hyperkählérien \eqref{int.1} est donnée par
$$
   \Im( H^*_c(M_{\zeta})\to H^*(M_\zeta)),
$$
où $H^*(M_{\zeta})$ et $H^*_c(M_{\zeta})$ dénotent les groupes de cohomologie de de Rham pour les formes lisses et les formes lisses à support compact respectivement.
\label{int.3}\end{corollaire}

Comme décrit dans \cite{Bielawski}, il est aussi possible de considérer des déformations de Taub-NUT de ces métriques hyperkählériennes toriques, ce qui a pour effet notamment de diminuer l'ordre de  croissance du volume.  Pour décrire la géométrie à l'infini de certaines de ces déformations, on introduit la classe des métriques tordues quasi-feuilletées au bord (Définition~\ref{tqfb.15} ci-bas), une généralisation des métriques tordues $\QAC$ de \cite{CR,CR2023} et des métriques feuilletées au bord de \cite{Rochon2012} qui, en tant que métriques associées à une structure de Lie à l'infini, sont automatiquement complètes à géométrie bornée.  Pour voir que les déformations de Taub-NUT correspondent à ce genre de géométrie, il faut toutefois se restreindre au cas où la métrique hyperkählérienne du Théorème~\ref{int.2} est en fait asymptotiquement conique (AC), ce qui revient à demander à ce que sa compactification $\QAC$, qui en général est une variété à coins, soit en fait une variété à bord (voir le Corollaire~\ref{cone.3} ci-bas pour des exemples).  On doit aussi supposer que la déformation de Taub-NUT est d'ordre $1$, c'est-à-dire qu'elle est spécifiée par une injection linéaire $\sigma:\bbR\hookrightarrow \gt^d$, où $\gt^d$ dénote l'algèbre de Lie du tore diagonal $T^d$.  Dans ce cas, on peut formuler notre description de la géométrie à l'infini comme suit, voir aussi le Théorème~\ref{dtn.25} ci-bas pour un énoncé plus détaillé.  
\begin{theoreme}
Lorsque la métrique hyperkählérienne du Théorème~\ref{int.2} est AC, une déformation de Taub-NUT d'ordre $1$ de cette métrique est une métrique tordue quasi-feuilletée au bord.  En particulier, elle a une croissance du volume (à l'infini) d'ordre $\dim M_\zeta-1$.  
\label{int.4}\end{theoreme} 

De ce résultat découle une décroissance à l'infini de la courbure sectionnelle (en général seulement dans certaines directions, voir le Corollaire~\ref{dtn.26} ci-bas pour l'énoncé précis).  Il permet aussi d'identifier un unique cône tangent à l'infini (Corollaire~\ref{dtn.29}).  En prenant $\sigma$ de sorte que $T^d= \overline{\exp\circ \sigma (\bbR)}$ et en considérant la déformation de Taub-NUT correspondante de $\bbH^d$ muni de sa métrique canonique $g_{\bbH^d}$, on obtient en particulier le résultat suivant (voir l'Exemple~\ref{dtn.36} et le Lemme~\ref{dtn.36b} ci-bas pour plus de détails).
\begin{corollaire}
Pour $\sigma$ tel que $T^d= \overline{\exp\circ \sigma (\bbR)}$, la déformation de Taub-NUT de $(\bbH^d,g_{\bbH^d})$ spécifiée par $\sigma$ a pour unique cône tangent à l'infini un espace stratifié de dimension $3d$.  Si $d>1$, cette dimension est en particulier strictement inférieure à l'ordre $4d-1$ de la croissance du volume de la métrique.  De plus, la courbure sectionnelle décroît linéairement à l'infini et en faisant varier $\sigma$, on obtient une infinité non dénombrable de telles métriques sur $\bbC^{2d}$ qui sont mutuellement non isométriques, même à une homothétie près.  
\label{int.5}\end{corollaire}
\begin{remarque}
Des métriques de Calabi-Yau sur $\bbC^n$ ayant un cône tangent à l'infini de dimension strictement inférieure à l'ordre de croissance du volume ont aussi été récemment identifiées par Min \cite{Min2}.
\label{int.6}\end{remarque}

Lorsque $\sigma$ est plutôt de la forme $\sigma(t)=(a_1t,\ldots,a_d t)$ avec $a_1,\ldots,a_d$ des entiers non nuls, le Théorème~\ref{int.4} permet aussi, en les raffinant un peu, de retrouver les résultats récents de Min \cite{Min2025} concernant la métrique de Taubian-Calabi sur $\bbC^{2d}$ et ses variantes, voir l'Exemple~\ref{dtn.35} ci-bas pour plus de détails.  Remarquons que les déformations de Taub-NUT d'ordre $1$ de la métrique euclidienne $g_{\bbH^{d}}$ apparaissent naturellement en physique mathématique.  En effet, par \cite{GRG1997} et \cite{LWY1996b}, elles peuvent être interprétées comme étant des espaces de modules de monopôles sur $\bbR^3$ de groupe de jauge $\SU(d+2)$ spontanément brisé à $\SU(d)\times U(1)\times U(1)$, ce qui correspond physiquement à deux monopôles massifs et $d-1$ monopôles de masse nulle.

Le Théorème~\ref{int.4} s'applique aussi aux déformations de Taub-NUT d'ordre 1 de la métrique de Calabi \cite[Théorème~5.3]{Calabi} sur $T^*\bbC\bbP^n$, un autre exemple de métrique hyperkählérienne torique AC, donnant lieu en particulier au résultat suivant (voir les Exemples~\ref{dtn.40}, \ref{dtn.37} et \ref{dtn.37b} ci-bas pour plus détails).
\begin{corollaire}
Les déformations de Taub-NUT d'ordre $1$ de la métrique de Calabi sur $T^*\bbC\bbP^n$  sont des métriques tordues quasi-feuilletées au bord dont la croissance du volume est d'ordre $4n-1$.  De plus, génériquement, la courbure sectionnelle décroît linéairement à l'infini et le cône tangent à l'infini est de dimension $3n$.  Toutefois, pour certaines déformations, la courbure sectionnelle ne décroît pas à l'infini dans certaines directions.  De plus, pour chaque $k\in\{3n,3n+1,\ldots, 4n-1\}$, il existe une déformation de Taub-NUT d'ordre $1$ de la métrique de Calabi sur $T^*\bbC\bbP^n$ ayant un cône tangent à l'infini de dimension $k$.
\label{int.7}\end{corollaire}

Pour démontrer le Théorème~\ref{int.4}, l'idée est d'effectuer une suite d'éclatements pour introduire une compactification par une variété à coins en suivant une approche combinant celles de \cite{CR,CR2023} et \cite{Melrose_London,DR}.  Dans le cas de la métrique euclidienne sur $\bbC^{2d}$, nos méthodes peuvent aussi être adaptées au cas des déformations de Taub-NUT d'ordre maximal.  En référant au Théorème~\ref{QFB.15} et à ses corollaires pour plus de détails, voici le résultat qu'on peut obtenir.
\begin{theoreme}
Soit $g_{\TN_\sigma}$ une déformation de Taub-NUT d'ordre $d$ de la métrique euclidienne $g_{\bbH^{d}}$ sur $\bbH^{d}$ spécifiée par une bijection linéaire $\sigma:\bbR^d\to\mathfrak{t}^d$.  Alors $g_{\TN_\sigma}$ est une métrique $\QFB$ polyhomogène ayant pour unique cône tangent à l'infini $\bbR^{3d}$ muni de la métrique euclidienne.  En particulier, $g_{\TN_\sigma}$ est à géométrie bornée et a une croissance du volume (à l'infini) d'ordre $3d$.  
\label{int.8}\end{theoreme}

En faisant varier $\sigma$, le Corollaire~\ref{QFB.17} ci-bas montre qu'on obtient une famille de dimension réelle $\frac{d(d+1)}2$ de métriques isométriquement distinctes.  Cette famille d'exemples contient notamment le produit cartésien de $d$ copies de la métrique de Taub-NUT $g_{\TN}$ sur $\bbC^2$ lorsque $\sigma$ correspond à l'isomorphisme canonique $\bbR^d=\mathfrak{t}^d$, mais aussi la métrique $L^2$ de l'espace de modules des monopôles centrés de type $(1,\ldots,1)$ par \cite{LWY1996,Murray1997} et \cite[Appendix]{Kraan}.  Comme les métriques du Théorème~\ref{int.8} sont $\QFB$ par rapport à la même structure de Lie à l'infini, on sait toutefois par \cite{ALN04} qu'elles sont toutes quasi-isométriques.  Suivant la stratégie ébauchée dans \cite[\S~7]{Gibbons1997}, cette observation permet de se ramener au cas $\sigma$ est l'isomorphisme canonique pour calculer assez facilement leur cohomologie $L^2$ réduite, donnant lieu au résultat suivant (Corollaire~\ref{QFB.18} ci-bas).
\begin{corollaire}
La cohomologie $L^2$ réduite des métriques du Théorème~\ref{int.8} est de dimension $1$ en degré $2d$ et est triviale pour les autres degrés. 
\label{int.9}\end{corollaire}
Nos méthodes  ne nous permettent pas toutefois comme dans \cite{HHM2004} ou \cite[\S~6]{KR2} de donner une interprétation en termes de cohomologie d'intersection.  Le cas particulier où le Corollaire~\ref{int.9} est appliqué à la  métrique $L^2$ de l'espace de modules des monopôles centrés de type $(1,1,\ldots,1)$ est important, car il donne lieu au résultat suivant.
\begin{corollaire}
La conjecture de $S$-dualité de Sen \cite{Sen, Gibbons1996} est valide pour les monopôles centrés de type $(1,1,\ldots,1)$.
\label{int.10}\end{corollaire}

Pour les déformations de Taub-NUT d'ordre intermédiaire, il n'est pas exclu que l'approche par compactification et éclatements développée dans cet article puisse à nouveau fonctionner, peut-être par exemple en élargissant la classe des métriques tordues quasi-feuilletées au bord.  Cependant, l'auteur n'entrevoit pas pour l'instant comment généraliser cette approche à des déformations de Taub-NUT pour des métriques hyperkählériennes toriques qui sont $\QAC$ sans être AC.  À la lumière de \cite{Min2}, on peut aussi se demander si l'approche développée ici pourrait servir à mieux comprendre la géométrie à l'infini des métriques de Calabi-Yau de \cite{Apostolov-Cifarelli}.  

L'article est organisé comme suit.  La section~\ref{tqfb.0} fait un rappel sur les métriques tordues $\QAC$ de \cite{CR,CR2023} et introduit les métriques tordues quasi-feuilletées au bord tout en dérivant certaines de leurs propriétés.  La section~\ref{qac.0} construit la compactification nécessaire pour établir le Théorème~\ref{int.2} et ses corollaires.  Dans la section~\ref{dtn.0}, on construit la compactification qui est ensuite utilisée pour démontrer le Théorème~\ref{int.4} et ses corollaires.  Enfin, dans la section~\ref{QFB.0}, on adapte cette approche aux déformations de Taub-NUT d'ordre maximal pour démontrer le Théorème~\ref{QFB.15} et ses corollaires, notamment la conjecture de Sen pour les monopôles centrés de type $(1,1,\ldots,1)$.    

\begin{remerciements}
L'auteur remercie Charles Cifarelli, Lorenzo Foscolo, Yuji Odaka et Andy Royston pour des discussions éclairantes en lien avec cet article.  Certaines de ces discussions se sont déroulées dans le cadre d'ateliers ayant eu lieu au \emph{Simons Laufer Mathematical Sciences Intstitute} et au \emph{Simons Center for Geometry and Physics}.  L'auteur tient à remercier ces deux instituts de recherche pour leur hospitalité.   Ce travail de recherche a été soutenu financièrement par une subvention à la découverte du CRSNG et une subvention de projet de recherche en équipe du FRQNT.
\end{remerciements}

\section{Métriques tordues quasi-feuilletées au bord}\label{tqfb.0}

Dans cette section, nous allons introduire la notion de métriques tordues quasi-feuilletées au bord intervenant dans notre résultat principal.  Pour y parvenir, il faudra dans un premier temps faire un récapitulatif des notions de métriques quasi-fibrées au bord \cite{CDR} et de métriques tordues quasi-asymptotiquement coniques     \cite{CR, CR2023}.  Ces notions sont formulées en termes d'une compactification par une variété à coins.  Pour une introduction et plus de détails sur la géométrie des variétés à coins, on réfère à \cite{Melrose1992, MelroseMWC, Grieser, Ammar}.

Soit donc $M$ une variété à coins.  Sauf mention contraire, on supposera qu'elle est compacte.  Dénotons par $\cM_1(M)$ l'ensemble des hypersurfaces bordantes, c'est-à-dire les coins de codimension $1$.  On supposera, comme dans \cite{Melrose1992}, que chacune des hypersurfaces bordantes de $M$ est plongée dans $M$.  Supposons de plus que chaque hypersurface bordante $H\in\cM_1(M)$ est munie d'un fibré $\phi_H: H\to S_H$ dont les fibres et la base $S_H$ sont des variétés à coins.  Dénotons par $\phi$ la collection de ces fibrés $\phi_H$.  

\begin{definition}[\cite{AM2011,ALMP2012,DLR}] On dit que $(M,\phi)$ est une \textbf{variété à coins fibrés}, ou de manière équivalente, que $\phi$ est une \textbf{structure de fibrés itérés} pour $M$, si $\cM_1(M)$ possède un ordre partiel tel que:
\begin{enumerate}
\item Tout sous-ensemble $\cI\subset \cM_1(M)$ tel que $\displaystyle \bigcap_{H\in \cI}H \ne \emptyset$ est totalement ordonné;
\item Si $H<G$, alors $H\cap G\ne \emptyset$ et l'application $\displaystyle \phi_H|_{H\cap G}: H\cap G\to S_H$ est une submersion surjective, $S_{GH}:=\phi_G(H\cap G)$ est l'une des hypersurfaces bordantes de $S_G$ et il existe une submersion surjective $\phi_{GH}: S_{GH}\to S_H$ telle que $\phi_{GH}\circ \phi_G=\phi_H$ sur $H\cap G$;
\item Pour $G\in \cM_1(M)$, les hypersurfaces bordantes de $S_G$ sont données par $S_{GH}$ pour $H<G$.
\end{enumerate}
\label{tqfb.1}\end{definition}

De cette définition, il découle automatiquement que la base $S_H$ et les fibres du fibré $\phi_H:H\to S_H$ sont naturellement elles aussi des variétés à coins fibrés.
Les variétés à coins fibrés sont intimement reliées à la notion d'espace stratifié.  Puisque ce lien jouera un rôle important dans cet article, prenons le temps de bien expliquer de quoi il s'agit.  
\begin{definition}
  Un \textbf{espace stratifié} de dimension $n$ est un espace métrisable localement séparable $X$ muni d'une \textbf{stratification}, c'est-à-dire d'une partition localement finie $\cS=\{s_i\}$ en sous-ensembles localement fermés de $X$ appelés \textbf{strates} qui sont des variétés lisses de dimension $\dim s_i\le n$ telles qu'au moins l'une des strates soit de dimension $n$ et 
  $$
      s_i\cap \overline{s_j}\ne \emptyset \quad \Longrightarrow \quad s_i\subset \overline{s}_j.
  $$ 
Lorsque $s_i\subset \overline{s}_j$, on écrit $s_i\le s_j$ et aussi $s_i<s_j$ si $s_i\ne s_j$, ce qui induit un ordre partiel sur l'ensemble des strates.  Une stratification induit une filtration
$$
   \emptyset \subset X_1\subset \cdots \subset X_n=X,
$$  
où $X_j$ est l'union des strates de dimension au plus $j$.  Les strates incluses dans $X\setminus X_{n-1}$ sont dites \textbf{régulières}, alors que celles incluses dans $X_{n-1}$ sont dites \textbf{singulières}.
\label{es.1}\end{definition}

Remarquons que la fermeture $\overline{s}_i$ d'une strate $s_i$ est elle-même naturellement un espace stratifié.  La \textbf{profondeur} d'un espace stratifié est le plus grand entier $k$ tel qu'on puisse trouver $k+1$ strates distinctes $s_1,\ldots,s_{k+1}$ telles que
$$
          s_1<\cdots <s_{k+1}.
$$
Comme indiqué dans \cite{AM2011,ALMP2012,DLR}, une variété à coins fibrés $(M,\phi)$ est automatiquement accompagnée d'un espace stratifié ${}^SM:=M/\sim$, où $\sim$ est la relation d'équivalence 
$$
   p\sim q\quad \Longleftrightarrow \quad p=q \; \mbox{ou} \;\exists \ H\in \cM_1(M) \; \mbox{telle que} \; p,q\in H \;\mbox{et} \; \phi_H(p)=\phi_H(q). 
$$
En termes de l'application $\beta: M\to {}^{S}M$ induite par passage au quotient, la variété à coins fibrés $M$ peut être vue comme une \textbf{résolution} de l'espace stratifié ${}^{S}M$ avec $\beta(M\setminus \pa M)$ correspondant à la strate régulière.  En fait, l'application $\beta$ induit une bijection entre les hypersurfaces bordantes de $M$ et les strates singulières de ${}^{S}M$. Plus précisément, à une hypersurface bordante $H\in\cM_1(M)$ correspond une strate $s_H$ de ${}^{S}M$ ayant pour fermeture $\overline{s}_H= \beta(H)$.  Dans cette correspondance, la base $S_H$ peut être vue, en tant que variété à coins fibrés, comme la résolution de l'espace stratifié $\overline{s}_H$.  On dira qu'un espace stratifié est de \textbf{Thom-Mather}  (aussi \emph{smoothly stratified} en anglais) s'il admet une résolution par une variété à coins fibrés $(M,\phi)$, c'est-à-dire s'il correspond à l'espace stratifié ${}^{S}M$ associé à $(M,\phi)$.  Ce ne sont pas tous les espaces stratifiés qui sont de Thom-Mather, mais tel que discuté dans \cite{ALMP2012,DLR}, la propriété d'être de Thom-Mather peut-être décrite intrinsèquement sur l'espace stratifié en termes de certaines données de contrôle.  

Dans cet article, tous les espaces stratifiés seront de Thom-Mather.  Ils interviendront dans la description d'objets singuliers tels que le cône tangent à l'infini de certaines variétés riemanniennes complètes, mais aussi dans la description d'une action lisse d'un groupe de Lie compact sur une variété à coins et son quotient.  En effet, supposons qu'une telle action d'un groupe $G$ sur une variété à coins $M$ soit \textbf{sans intersection au bord} au sens de \cite[Definition~1.4]{AM2011}, c'est-à-dire que pour tout $g\in G$ et pour toute hypersurface bordante $H\in\cM_1(M)$, 
\begin{equation}
g\cdot H\ne H \quad \Longrightarrow \quad (g\cdot H)\cap H=\emptyset.
\label{es.2}\end{equation}
Si de plus cette action est libre en au moins un point, on sait alors par \cite[Theorem~7.5]{AM2011}  que $M$ admet une stratification en termes des classes de conjugaison des stabilisateurs de cette action et que l'éclatement de ses strates (dans un ordre compatible avec l'ordre partiel de la stratification) définit une variété à coins $\widetilde{M}$ et une application de contraction (\emph{blow-down map} en anglais) $\beta: \widetilde{M}\to M$ telles que l'action de $G$ sur $M$ se relève en une action libre sur $\widetilde{M}$.  De plus, pour une telle résolution, le quotient $\widetilde{M}/G$ est naturellement une variété à coins et peut être vu comme une résolution du quotient $M/G$ en tant qu'espace stratifié de Thom-Mather.  Dans cet article, les actions d'un groupe $G$ sur une variété à coins seront toujours telles que $g\cdot H=H$ pour tout $g\in G$ et pour toute hypersurface bordante $H\in\cM_1(M)$, de sorte que l'action sera trivialement sans intersection au bord.  Si $\phi$ est une structure de fibrés itérés sur $M$, on dira qu'une telle action de $G$ sur $M$ est \textbf{compatible} avec la structure de fibrés itérés si pour toute hypersurface bordante $H\in\cM_1(M)$, il existe une action de $G$ sur $S_H$ à laquelle est associée une seule classe de conjugaison de stabilisateurs telle que le fibré $\phi_H: H\to S_H$ est $G$-équivariant.  Si de plus l'action de $G$ est libre, alors on peut vérifier dans ce cas que le quotient $M/G$ est naturellement une variété à coins fibrés.

 Pour introduire les classes de métriques qui interviendront dans cet article, nous aurons besoin de la notion suivante.     
\begin{definition}
Une \textbf{fonction bordante} de $H\in \cM_1(M)$, aussi dite \textbf{fonction de définition} de $H$, est une fonction $x_H\in\CI(M)$ telle que $x_H\ge 0$, $H=x_H^{-1}(0)$ et sa différentielle $dx_H$ est non nulle partout sur $H$.  Une telle fonction $x_H$ est dite \textbf{compatible} avec une structure de fibrés itérés $\phi$ si pour toute hypersurface bordante $G>H$, la restriction de $x_H$ à $G$ est constante le long des fibres du fibré $\phi_G: G\to S_G$.  Une \textbf{fonction bordante totale} est une fonction $v\in\CI(M)$ de la forme
$$
\displaystyle v=\prod_{H\in\cM_1(M)} x_H
$$ 
avec $x_H$ un choix de fonction bordante pour $H$.  
\label{tqfb.2}\end{definition}

Par \cite[Lemma~1.4]{DLR}, des fonctions bordantes compatibles existent toujours.  Si $p\in \pa M$ est situé à l'intérieur d'un coins $H_1\cap\cdots \cap H_k$ de codimension $k$ pour des hypersurfaces bordantes $H_1,\ldots, H_k\in \cM_1(M)$ étiquetées de sorte que $H_1<\cdots<H_k$,  alors pour $x_{i}$ un choix de fonction bordante compatible pour $H_i$, on sait par \cite[Lemma~1.10]{CDR} qu'il existe un voisinage de $p$ où chaque fibré $\phi_{H_i}: H_i\to S_{H_i}$ est trivial, ainsi que des fonctions $y_i=(y^1_i,\ldots,y^{\ell_i})$ pour $i\in\{1,\ldots,k\}$ et $z=(z_1,\ldots, z_k)$ telles que 
\begin{equation}
(x_1,y_1,\ldots,x_k,y_k,z)
\label{tqfb.4}\end{equation} 
constitue un système de coordonnées dans ce voisinage ayant la propriété que sur $H_i$, $(x_1,y_1,\ldots,x_{i-1},y_{i-1},y_i)$ induit un système de coordonnées sur $S_{H_i}$ par rapport auquel $\phi_{H_i}$ correspond à l'application 
$$
(x_1,y_1,\ldots, \widehat{x}_i,y_i,\ldots, x_k,y_k,z)\mapsto (x_1,y_1,\ldots,x_{i-1},y_{i-1},y_i),
$$
où \og$\widehat{}$ \fg \; au-dessus d'une variable dénote son omission.
\begin{definition}\emph{(\cite[Definition~2.3]{CR2023})} Une \textbf{fonction de pondération} pour une variété à coins fibrés $(M,\phi)$ est une application 
$$
        \begin{array}{llcl}
        \gp: & \cM_1(M) & \to & [0,1) \\
              & H & \mapsto & \nu_H
        \end{array}
$$
telle que pour toutes hypersurfaces bordantes $G,H\in \cM_1(M)$,
$$
 H<G \; \Longrightarrow \; \nu_H\le \nu_G.
$$
Si pour chaque hypersurface bordante $H\in \cM_1(M)$, $x_H\in \CI(M)$ est un choix de fonction bordante compatible, alors la \textbf{distance $\gp$-tordue} correspondante est la fonction
$$
    \rho:= \prod_{H\in\cM_1(M)} x_H^{-\frac{1}{1-\nu_H}}.
$$
De manière équivalente, on dira que $\rho^{-1}$ est une \textbf{fonction bordante totale $\gp$-pondérée}.
\label{tqfb.5}\end{definition}

En termes de l'algèbre de Lie  des \textbf{champs vectoriels bordants}  (ou \textbf{$b$-champs vectoriels}) sur $M$  donnée par
\begin{equation}
 \cV_b(M):= \{\xi\in \CI(M;TM)\; | \; \xi \; \mbox{est tangent à} \;H  \;\forall  H\in \cM_1(M)\},
\label{tqfb.6}\end{equation}
rappelons qu'un choix de fonction bordante totale $\gp$-pondérée induit une sous-algèbre de Lie de champs vectoriels bordants comme suit.
\begin{definition}\emph{(\cite[Definition~2.4]{CR2023})} Soit $\rho$ une distance $\gp$-tordue pour $(M,\phi)$.  Pour un tel choix, un \textbf{champ vectoriel quasi-fibré au bord $\gp$-pondéré} est un champ vectoriel bordant $\xi\in \cV_b(X)$ tel que:
\begin{enumerate}
\item $\xi|_H$ est tangent aux fibres de $\phi_H: H\to S_H$ pour chaque $H\in \cM_1(M)$;
\item $\xi \rho^{-1}\in v\rho^{-1}\cA_{\phg}(M)\cap L^{\infty}(M)$, où $v=\prod_{H\in\cM_1(M)}x_H$ est une fonction bordante totale et $\cA_{\phg}(M)$ dénote l'espace des fonctions polyhomogènes sur $M$ au sens de \cite[\S~4]{Melrose1992}.
\end{enumerate}
On dénote par $\cV_{\nQFB}(M)$ l'espace des champs vectoriels quasi-fibrés au bord $\gp$-pondérés.
\label{tqfb.7}\end{definition}

Lorsque $\gp$ est la fonction de pondération triviale donnée par $\gp(H)=0$ pour toute hypersurface bordante $H\in\cM_1(M)$, on a que $\cV_{\nQFB}(M)$ correspond à l'algèbre de Lie des champs vectoriels quasi-fibrés au bord ($\QFB$) introduite dans \cite{CDR}.  En fait, on obtient aussi cette algèbre de Lie de champs vectoriels lorsque $\gp$ est une fonction de pondération constante.    Tel qu'indiqué dans \cite{CR2023}, les conditions (1) et (2) de la Définition~\ref{tqfb.7} sont fermées par rapport au crochet de Lie, ce qui assure que $\cV_{\nQFB}(M)$ est bien une sous-algèbre de Lie de $\cV_b(M)$.  C'est aussi une sous-algèbre de Lie de l'algèbre de Lie $\cV_e(M)$ des champs vectoriels edges \cite{MazzeoEdge, AGR}, à savoir l'algèbre de Lie constituée des champs vectoriels bordants satisfaisant à la condition (1) de la Définition~\ref{tqfb.7}.

L'algèbre de Lie $\cV_{\nQFB}(M)$ dépend en général du choix de fonction bordante totale $\gp$-pondérée.  On dira que deux telles fonctions sont $\nQFB$-équivalentes si elles engendrent la même algèbre de Lie $\cV_{\nQFB}(M)$ de champs vectoriels quasi-fibrés au bord $\gp$-pondérés.  Par \cite[Lemma~2.6]{CR2023}, deux telles fonctions $\rho$ et $\rho'$ sont $\nQFB$-équivalentes lorsque la fonction 
$$
    f:= \log\lrp{\frac{\rho}{\rho'}}\in \cA_{\phg}(M)\cap L^{\infty}(M)
$$
est telle que près de chaque hypersurface bordante $H\in \cM_1(M)$, 
$$
    f= \phi_H^* f_H + \cO(x_H)
$$
pour une certaine fonction $f_H\in \cA_{\phg}(S_H)\cap L^{\infty}(S_H)$.  En termes des coordonnées \eqref{tqfb.4}, on sait par \cite[(2.4)]{CR2023} que $\cV_{\nQFB}(M)$ est localement engendrée par
\begin{multline}
v_1x_1\frac{\pa}{\pa x_1}, v_1\frac{\pa}{\pa y_1}, v_2\left((1-\nu_{H_2})x_2\frac{\pa}{\pa x_2}-(1-\nu_{H_1})x_1\frac{\pa}{\pa x_1}\right), v_2\frac{\pa}{\pa y_2}, \ldots, \\
v_k\left((1-\nu_{H_k})x_k\frac{\pa}{\pa x_k}-(1-\nu_{H_{k-1}})x_{k-1}\frac{\pa}{\pa x_{k-1}}\right), v_k\frac{\pa}{\pa y_k}, \frac{\pa}{\pa z}
\label{tqfb.8}\end{multline}
en tant que $\CI(M)$-module, où $v_i:= \prod_{j=i}^k x_j$ et où $\frac{\pa}{\pa y_i}$ et $\frac{\pa}{\pa z}$ dénotent respectivement $\left( \frac{\pa}{\pa y_i^1},\ldots, \frac{\pa}{\pa y_i^{\ell_i}}  \right)$ et $\left(\frac{\pa}{\pa z^1},\ldots,\frac{\pa}{\pa z^q}  \right)$.  Cette description locale montre que $\cV_{\nQFB}(M)$ est un faisceau localement libre de rang $m=\dim M$, donc par le théorème de Serre-Swan, il existe un fibré vectoriel lisse, le \textbf{fibré tangent quasi-fibré au bord $\gp$-pondéré}, dénoté ${}^{\gp}TM\to M$, et une application naturelle $\iota_{\gp}: {}^{\gp}TM\to TM$ se restreignant à un isomorphisme de fibrés vectoriels sur $M\setminus \pa M$ telle que 
\begin{equation}
 \cV_{\nQFB}(M)= (\iota_{\gp})_*\CI(M;{}^{\gp}TM).
\label{tqfb.9}\end{equation}
En $p\in M$, la fibre de ${}^{\gp}TM$ au-dessus de $p$ est donnée par 
$$
    {}^{\gp}T_pM= \cV_{\nQFB}(M)/\cI_p\cdot \cV_{\nQFB}(M),
$$
où $\cI_p$ est l'idéal des fonctions lisses s'annulant en $p$.  Le \textbf{fibré cotangent quasi-fibré au bord $\gp$-pondéré} est le fibré vectoriel ${}^{\gp}T^*M$ dual de ${}^{\gp}TM$.  Dans les coordonnées \eqref{tqfb.4}, une base locale de sections de ${}^{\gp}T^*M$ est donnée par
\begin{equation}
 \rho_1^{-\gp}d\rho_1, \frac{dy_1^1}{v_1}, \ldots, \frac{dy_1^{\ell_1}}{v_1}, \ldots, \rho_k^{-\gp}d\rho_k, \frac{dy_k^1}{v_k},\ldots, \frac{dy_k^{\ell_k}}{v_k}, dz^1,\ldots, dz^q,
\label{tqfb.9b}\end{equation}  
où 
$$
\rho_i= \prod_{j=i}^k x_j^{-\frac{1}{1-\nu_{H_j}}}\quad  \mbox{et} \quad \rho_i^{-\gp}=\prod_{j=i}^k x_j^{\frac{\nu_{H_j}}{1-\nu_{H_j}}}.
$$

L'identification \eqref{tqfb.9} induit une structure d'algèbre de Lie sur $\CI(M;{}^{\gp}TM)$ conférant à ${}^{\gp}TM\to M$ une structure d'algébroïde de Lie avec application d'ancrage $\iota_{\gp}$.  De plus, l'algèbre de Lie de champs vectoriels $\cV_{\nQFB}(M)$ induit une structure de Lie à l'infini sur $M\setminus \pa M$ au sens de \cite{ALN04}.  À cette structure de Lie à l'infini est associée une classe naturelle de métriques riemanniennes complètes sur $M\setminus \pa M$ comme suit.
\begin{definition}
Une \textbf{métrique $\gp$-pondérée quasi-fibrée au bord ($\nQFB$)}  est une métrique riemannienne sur $M\setminus \pa M$ de la forme 
$$
     g=\iota_{\gp}^* g_{\nQFB}
$$
avec $g_{\nQFB}\in\CI(M;S^2({}^{\gp}T^*M))$ un choix de métrique euclidienne pour le fibré ${}^{\gp}TM\to M$.  Lorsque cela ne porte pas à confusion, on dénotera aussi par $g_{\nQFB}$ la métrique riemannienne $g$ induite par $g_{\nQFB}$.
\label{tqfb.10}\end{definition}
\begin{remarque}
Dans la définition précédente, $g_{\nQFB}$ est une section de $S^{2}({}^{\gp}T^*M)$ lisse jusqu'au bord.  Toutefois, comme dans \cite[Remark~1.18]{CDR}, tout en restant dans la classe de quasi-isométrie de $g$, on peut plus généralement considérer des sections $g_{\nQFB}$ qui sont moins régulières au bord,  par exemple en ne demandant qu'un développement polyhomogène ou un contrôle uniforme des dérivées près du bord. 
\label{tqfb.11}\end{remarque}

Par les résultats de \cite{ALN04, Bui}, une métrique $\gp$-pondérée quasi-fibrée au bord est automatiquement complète à géométrie bornée.  Lorsque la variété à coins fibrés est telle que pour chaque hypersurface bordante maximale $H\in\cM_1(M)$, $S_H=H$ et $\phi_H:H\to S_H$ est l'application identité, on dit alors que $(M,\phi)$ est une variété à à coins fibrés \textbf{de type quasi-asymptotiquement conique} ($\QAC$), en quel cas on peut être plus précis en remplaçant \og quasi-fibré au bord \fg  \ par  \og quasi-asymptotiquement conique \fg \ dans les terminologies introduites jusqu'à présent, \eg une \textbf{métrique $\gp$-pondérée quasi-asymptotiquement conique} correspond à une métrique $\gp$-pondérée quasi-fibrée au bord lorsque que la variété à coins fibrés est de type $\QAC$ et on peut utiliser la notation $\cV_{\nQAC}(M)$ pour dénoter l'algèbre de Lie des champs vectoriels $\gp$-pondérés quasi-fibrés au bord $\cV_{\nQFB}(M)$.  Dans \cite{CR2023}, on s'était restreint à des variétés à coins fibrés de type $\QAC$ pour introduire la notion de métrique tordue $\QAC$, mais la définition a un sens aussi dans le cadre quasi-fibré.
\begin{remarque}
Si $M$ est de type $\QAC$, soit $x_{\max}\in\CI(M)$ le produit des fonctions bordantes associées aux hypersurfaces bordantes maximales.  En remplaçant la condition $(2)$ par $\xi \rho^{-1}\in \frac{v\rho^{-1}}{x_{\max}}\cA_{\phg}(M)\cap L^{\infty}(M)$ dans la Définition~\ref{tqfb.7}, on obtient alors l'algèbre de Lie $\cV_{\nQb}(M)$ des champs vectoriels $\gp$-pondérés $\Qb$ de \cite{CR2023}.  Cela définit une structure de Lie à l'infini et une classe de métriques associée, à savoir la classe des métriques $\gp$-pondérées $\Qb$ de \cite[Definition~2.13]{CR2023}.  Une telle métrique $g_{\nQb}$ est de la forme 
$$
   g_{\nQb}=x_{\max}^2g_{\nQAC}
$$    
pour une métrique $\gp$-pondérée $\QAC$ $g_{\nQAC}$.
\label{tqfb.12a}\end{remarque}
\begin{remarque}
De même, l'algèbre de Lie $\cV_e(M)$ des champs vectoriels edges induit une structure de Lie à l'infini sur $M\setminus \pa M$ à laquelle est associée la classe des métriques edges de Mazzeo \cite{MazzeoEdge}.  Par \cite{ALN04,Bui}, les métriques edges sont automatiquement complètes à géométrie bornée.  La classe des métriques edges permet aussi de définir celle des métriques wedges.  Par définition, une \textbf{métrique wedge} est une métrique riemannienne $g_w$ sur $M\setminus \pa M$ de la forme
\begin{equation}
       g_w= v^2 g_e
\label{edge.2}\end{equation} 
pour $g_e$ une métrique edge.  Contrairement aux métriques edges, les métriques wedges ne sont typiquement pas géodésiquement complètes et leur courbure n'est bien souvent pas bornée.  On réfère à \cite{ALMP2012,AGR,KR1} pour plus de détails sur les métriques wedges.  
\label{edge.1}\end{remarque}

Comme dans \cite[(1.19)]{KR1} pour les métriques $\QFB$,  il est utile dans les coordonnées \eqref{tqfb.4} de remplacer la $i$ ième coordonnée $x_i$ par $u_i:=\rho^{-1}=\prod_{j=1}^k x_j^{\frac{1}{1-\nu_{H_j}}}$, en quel cas la base locale de champs vectoriels \eqref{tqfb.8} prend la forme
\begin{equation}
\begin{gathered}
v_1\lrp{x_1\frac{\pa}{\pa x_1}+\frac{u_i}{1-\nu_{H_1}}\frac{\pa}{\pa u_i}}, v_1\frac{\pa}{\pa y_1}, v_2\left((1-\nu_{H_2})x_2\frac{\pa}{\pa x_2}-(1-\nu_{H_1})x_1\frac{\pa}{\pa x_1}\right), v_2\frac{\pa}{\pa y_2}, \ldots, \\v_{i-1}\left((1-\nu_{H_{i-1}})x_{i-1}\frac{\pa}{\pa x_{i-1}}-(1-\nu_{H_{i-2}})x_{i-2}\frac{\pa}{\pa x_{i-2}}\right), v_{i-1}\frac{\pa}{\pa y_{i-1}},\\
-(1-\nu_{H_{i-1}})v_{i-1}\frac{\pa}{\pa x_{i-1}}, v_i\frac{\pa}{\pa y_i}, (1-\nu_{H_{i+1}})v_{i+1}x_{i+1}\frac{\pa}{\pa x_{i+1}}, v_{i+1}\frac{\pa}{\pa y_{i+1}}, \\
v_{i+2}\left((1-\nu_{H_{i+2}})x_{i+2}\frac{\pa}{\pa x_{i+2}}-(1-\nu_{H_{i+1}})x_{i+1}\frac{\pa}{\pa x_{i+1}}\right), v_{i+2}\frac{\pa}{\pa y_{i+2}}, \ldots, \\
v_k\left((1-\nu_{H_k})x_k\frac{\pa}{\pa x_k}-(1-\nu_{H_{k-1}})x_{k-1}\frac{\pa}{\pa x_{k-1}}\right), v_k\frac{\pa}{\pa y_k}, \frac{\pa}{\pa z}.
\end{gathered}
\label{tqfb.11b}\end{equation}
En faisant un changement de base par rapport à $\CI(M)$, ces champs vectoriels sont équivalents à
\begin{equation}
\begin{gathered}
v_1u_i\frac{\pa}{\pa u_i}, v_1\frac{\pa}{\pa x_1}, v_1\frac{\pa}{\pa y_1}, \ldots, v_{i-1}\frac{\pa}{\pa x_{i-1}}, v_{i-1}\frac{\pa}{\pa y_{i-1}}, v_{i}\frac{\pa}{\pa y_{i}},v_{i+1}x_{i+1}\frac{\pa}{\pa x_{i+1}}, v_{i+1}\frac{\pa}{\pa y_{i+1}},\\
v_{i+2}\left((1-\nu_{H_{i+2}})x_{i+2}\frac{\pa}{\pa x_{i+2}}-(1-\nu_{H_{i+1}})x_{i+1}\frac{\pa}{\pa x_{i+1}}\right), v_{i+2}\frac{\pa}{\pa y_{i+2}}, \ldots, \\
v_k\left((1-\nu_{H_k})x_k\frac{\pa}{\pa x_k}-(1-\nu_{H_{k-1}})x_{k-1}\frac{\pa}{\pa x_{k-1}}\right), v_k\frac{\pa}{\pa y_k}, \frac{\pa}{\pa z}.
\end{gathered}
\label{tqfb.11c}\end{equation}
En particulier, sur $H_i$, les champs vectoriels horizontaux correspondent à des champs vectoriels edges, alors que les champs vectoriels verticaux sont des champs vectoriels $\nQFB$.  Près de $H_i$, cela suggère de décrire les métriques $\nQFB$ en termes des hypersurfaces de niveaux d'une fonction bordante totale.  Pour ce faire, considérons comme dans \cite[(1.15)]{KR1} l'ensemble 
\begin{equation}
 \cU_i= \left\{ (p,c)\in H_i\times [0,\epsilon) \; | \; \lrp{\prod_{H\in \cM_1(M), H\ne H_i} x_H(p)^{\frac{1}{1-\nu_H}}}> \frac{c}{\epsilon^{\frac{1}{1-\nu_{H_i}}}} \right\} \subset H_i\times [0,\epsilon)
\label{tqfb.11d}\end{equation} 
et le difféomorphisme 
\begin{equation}
\begin{array}{lccc}
  \Psi_i: & (H_i\setminus \pa H_i)\times [0,\epsilon) & \to & \cU_i\\
              & (p,t) & \mapsto & (p, t^{\frac{1}{1-\nu_{H_i}}}\prod_{H\in \cM_1(M), H\ne H_i}x_H(p)^{\frac{1}{1-\nu_H}}).
\end{array}
\label{tqfb.11e}\end{equation}
Sur $\cU_i$ vu comme un ouvert de $H_i\times [0,\epsilon)$, soient $\widetilde{\pr}_1$ et $\widetilde{\pr_2}$ les restrictions des projections $H_i\times [0,\epsilon)\to H_i$ et $H_i\times [0,\epsilon)\to [0,\epsilon)$, de sorte qu'un choix de connexion pour le fibré $\phi_{H_i}$ induit une inclusion de fibrés vectoriels 
\begin{equation}
       T^*(H_i/S_{H_i})|_{\cU_i}\hookrightarrow T^*M|_{\cU_i}.
\label{tqfb.11f}\end{equation}
En utilisant cette inclusion, un exemple de métrique $\nQFB$ sur $\cU_i$ est donné par
\begin{equation}
  \frac{1}{v^2}\left(\frac{du_i^2}{u_i^2}+ \widetilde{\pr}_1^*\phi_{H_i}^*g_{S_{H_i}}\right)+ \widetilde{\pr}_1^*\kappa,
\label{tqfb.11g}\end{equation}
où  $g_{S_{H_i}}$ est une métrique wedge comme dans \eqref{edge.2} et $\kappa\in \CI(H_i\setminus \pa H_i; S^2(T^*(H_i/S_{H_i})))$ induit une métrique $\nQFB$ sur chaque fibre de $\phi_{H_i}$ par rapport à la fonction de pondération canoniquement induite par $\gp$.  Une telle métrique $\nQFB$ \eqref{tqfb.11g} est dite de \textbf{type produit} près de $H_i$.  Plus généralement, une métrique $\nQFB$ est dite \textbf{exacte} si pour chaque hypersurface bordante $H\in\cM_1(M)$, elle est de type produit près de $H$ à un terme de $x_H\CI(M;S^2({}^{\nQFB}T^*M))$ près.  Lorsque la fonction de pondération $\gp$ est triviale, cela coïncide avec la notion de métrique $\QFB$ exacte de \cite{KR1}.  Si de plus $M$ est une variété à bord et le fibré sur $\pa M$ est l'application identité, une telle métrique correspond à un exemple de métrique \textbf{asymptotiquement conique} ($\AC$).

\begin{definition}\emph{(\cite[Definition~2.9]{CR2023})}
Soient $(M,\phi)$ une variété à coins fibrés, $\gp$ une fonction de pondération et $\rho$ un choix de distance $\gp$-tordue.  Dans ce cas, une \textbf{métrique $\gp$-tordue quasi-fibrée au bord} est une métrique riemannienne $g_w$ sur $M\setminus \pa M$ de la forme 
$$
       g_\cT= \rho^{2\gp}g_{\nQFB}, \quad \mbox{où} \quad \rho^{2\gp}:= \prod_{H\in \cM_1(M)} x_H^{\frac{-2\nu_H}{1-\nu_H}}
$$
et $g_{\nQFB}$ est une métrique $\gp$-pondérée quasi-fibrée au bord.
\label{tqfb.12}\end{definition}
\begin{remarque}
Comme les métriques $\gp$-pondérées sont à géométrie bornée, on déduit de \cite[Theorem~1.159]{Besse} que le tenseur de Riemann $\Rm_\cT$ d'une métrique $\gp$-tordue quasi-fibrée au bord $g_\cT$ est tel que
$$
    |\Rm_\cT|_{g_\cT}=\cO(\rho^{-2\gp}).
$$
De plus, si $M$ est de type $\QAC$, alors par la Remarque~\ref{tqfb.12a}, $g_\cT=\frac{\rho^{2\gp}}{x_{\max}^2}g_{\nQb}$ pour une métrique $\gp$-pondérée $\Qb$ et \cite[Theorem~1.159]{Besse} montre dans ce cas que  
$$
|\Rm_\cT|_{g_\cT}=\cO(\rho^{-2\gp}x_{\max}^2),
$$
puisque par \cite{ALN04,Bui}, les métriques $\gp$-pondérées $\Qb$ sont à géométrie bornée.
\label{tqfb.12b}\end{remarque}

Dans les coordonnées \eqref{tqfb.4} et la base locale de sections \eqref{tqfb.9b}, un exemple de métrique $\gp$-tordue quasi-fibrée au bord est une métrique riemannienne $g_\cT$ sur $M\setminus \pa M$ de la forme 
$$
         g_\cT= \sum_{i=1}^k \rho^{2\gp}\left( \rho_i^{-2\gp}d\rho_i^2+ \sum_{j=1}^{\ell_i}\frac{(dy_i^j)^2}{v_i^2} \right) + \rho^{2\gp} \sum_{j=1}^q d(z^j)^2,
$$
où $ \rho_i^{2\gp}= (\rho_i^{\gp})^2$.
Près d'une hypersurface bordante $H_i$ et en termes de \eqref{tqfb.11g}, un autre exemple de métrique $\gp$-tordue quasi-fibrée au bord est donné par
\begin{equation}
 \frac{du_i^2}{u_i^4}+ \frac{\widetilde{\pr}_1^*\phi_{H_i}^*g_{S_{H_i}}}{u_i^2}+ \rho^{2\gp} \widetilde{\pr}_1^*\kappa =  \frac{du_i^2}{u_i^4}+ \frac{\widetilde{\pr}_1^*\phi_{H_i}^*g_{S_{H_i}}}{u_i^2}+ \frac{\rho^{2\gp}}{\rho^{2\gp}_i} \widetilde{\pr}_1^*\kappa_\cT,
\label{tqfb.12c}\end{equation}
où $\kappa_\cT\in \CI(H_i\setminus \pa H_i; S^2(T^*(H_i/S_{H_i})))$ induit une métrique $\gp$-tordue quasi-fibrée au bord sur chaque fibre de $\phi_{H_i}$ par rapport à la fonction de pondération canoniquement induite par $\gp$.  On dit qu'une métrique $\gp$-tordue $\QFB$ de la forme \eqref{tqfb.12c} est de \textbf{type produit} près de $H_i$.

 En particulier, cette description locale d'une métrique $\gp$-tordue $\QFB$ permet de donner une définition itérative des métriques $\gp$-tordues $\QFB$ comme dans \cite{DM2018} pour les métriques quasi-asymptotiquement coniques.  Comme les métriques $\gp$-tordues $\QFB$ sont conformément reliées à des métriques $\gp$-pondérées $\QFB$, on peut leur associer un fibré vectoriel ${}^{\cT}TM\to M$ donné par 
$$
    {}^{\cT}TM:= (\rho^{-\gp})({}^{\gp}TM).
$$
On dénotera par ${}^{\cT}T^*M$ le fibré vectoriel dual de ${}^{\cT}TM$.   En termes de ce fibré, on dira qu'une métrique $\gp$-tordue $\QFB$ est \textbf{exacte} si pour chaque hypersurface bordante $H\in \cM_1(M)$, elle est de type produit près de $H$ à un terme de $x_H\CI(M;S^2({}^{\cT}T^*M))$ près.

Lorsque la fonction de pondération est triviale, ${}^{\cT}TM$  correspond au fibré tangent $\QFB$ et possède donc une structure naturelle d'algébroïde de Lie.  Il y a une autre situation où ${}^{\cT}TM$ est naturellement un algébroïde de Lie, à savoir le cas où la fonction de pondération est à valeurs dans $\{0,\frac12\}$.  En effet, dans ce cas, 
$$
     \rho^{-\gp}= v_{\frac12}:= \prod_{H, \nu_H=\frac12} x_H
$$ 
est une fonction lisse jusqu'au bord, de sorte que $\iota_{\gp}$ induit une application d'ancrage 
$$
      \begin{array}{llcl}
       \iota_{\cT}: & {}^{\cT}TM & \to & TM \\
       & v_{\frac12}\xi &\mapsto & v_{\frac12}\iota_{\gp}(\xi)
      \end{array}
$$
et une identification
$$
     (\iota_{\cT})_*:\CI(M;{}^{\cT}TM)\tilde{\longrightarrow} v_{\frac12}\cV_{\nQFB}(M) \subset \cV_{\nQFB}(M).
$$
Comme $v_{\frac12}\cV_{\nQFB}(M)$ est clairement une sous-algèbre de Lie de $\cV_{\nQFB}(M)$, cela confère à ${}^{\cT}TM$ une structure d'algébroïde de Lie.  En fait, l'algèbre de Lie de champs vectoriels $v_{\frac12}\cV_{\nQFB}(M)$ induit une structure de Lie à l'infini sur $M\setminus\pa M$ et les métriques $\gp$-tordues $\QFB$ constituent précisément la classe de métriques associées à cette structure de Lie à l'infini.  À travers $\cV_{\nQFB}(M)$, l'algèbre de Lie  $v_{\frac12}\cV_{\nQFB}(M)$ dépend évidemment du choix de distance $\gp$-tordue, mais elle ne dépend pas du choix de la fonction $v_{\frac12}$ comme on peut le vérifier aisément.  Profitons-en pour mentionner le pendant de la Remarque~\ref{tqfb.12b} en termes d'algébroïdes de Lie.
\begin{lemme}
Si la fonction de pondération $\gp$ est à valeurs dans $\{0,\frac12\}$, alors 
$$
    [v_{\frac12}\cV_{\nQFB}(M),v_{\frac12}\cV_{\nQFB}(M)]\subset v_{\frac12}(v_{\frac12}\cV_{\nQFB}(M)).
$$
De plus, si $M$ est de type $\QAC$, alors
$$
 [v_{\frac12}\cV_{\nQFB}(M),v_{\frac12}\cV_{\nQFB}(M)]\subset x_{\max}v_{\frac12}(v_{\frac12}\cV_{\nQFB}(M)).
$$
\label{tqfb.12d}\end{lemme} 
\begin{proof}
Puisque  
$$
 [\cV_{\nQFB}(M),\cV_{\nQFB}(M)]\subset \cV_{\nQFB}(M),
$$
le résultat découle du fait qu'un champ vectoriel $\xi \in \cV_{\nQFB}(M)$ est en particulier un champ vectoriel bordant, donc que
$$
    \xi v_{\frac12}\in v_{\frac12}\CI(M).
$$
Si $M$ est de type $\QAC$, alors en termes de la Remarque~\ref{tqfb.12}, 
$$
   \cV_{\nQFB}(M)=x_{\max}\cV_{\nQb}(M),
$$
où $\cV_{\nQb}(M)$ est l'algèbre de Lie des champs vectoriels $\gp$-pondérés $\Qb$.  Comme un champ vectoriel $\xi\in\cV_{\nQb}(M)$ est en particulier un champ vectoriel bordant, on aura que
$$
     \xi (x_{\max} v_{\frac12})\in (x_{\max}v_{\frac12})\CI(M),
$$ 
donc le résultat découle du fait que 
$$
 [\cV_{\nQb}(M),\cV_{\nQb}(M)]\subset \cV_{\nQb}(M).
$$

\end{proof}

Dans cet article, on ne considérera que des fonctions de pondération à valeurs dans $\{0,\frac12\}$.  Pour une telle fonction de pondération, la classe des métriques $\gp$-tordues $\QFB$ n'est toutefois pas celle dont on aura besoin.  En effet, il faudra en général prendre en compte un feuilletage à l'infini.  Dans nos exemples, les feuilles de ce feuilletage correspondront aux orbites d'une action de $\bbR$ sur $M$.

Pour introduire la classe de métriques centrale à cet article,  soit $(M,\phi)$ une variété à coins fibrés munie d'une fonction de pondération $\gp$ à valeurs dans $\{0,\frac12\}$.
\begin{definition}
 Un \textbf{$(\phi,\gp)$-feuilletage} sur la variété à coins fibrés $(M,\phi)$ est un feuilletage $\cF$ tel que pour toute hypersurface bordante $H$ de $M$, $T\cF|_H\subset TH$ de sorte que:
 \begin{enumerate} 
 \item La base $S_H$ est munie d'un feuilletage $\cF_H$ tel que $T_{\phi_H(p)}\cF_H= (\phi_H)_*(T_p\cF)$ pour tout $p\in H$;
 \item Si $\gp(H)=0$, $\cF_H$ est le feuilletage sur $S_H$ ayant pour feuilles les points de $S_H$, c'est-à-dire que les feuilles de $\cF|_H$ sont tangentes aux fibres de $\phi_H:H\to S_H$.
 \end{enumerate}
\label{feuille.5}\end{definition} 

 Supposons donc que la variété à coins fibrés $(M,\phi)$ soit munie d'un $(\phi,\gp)$-feuilletage. Pour $p\in M$, soit $\cU\subset M$ un voisinage ouvert de $p$ admettant une carte locale du feuilletage $\cF$.  Dans ce cas, il existe un fibré $\psi: \cU\to \cV$  tel que  $T\cF|_{\cU}=T(\cU/\cV$), c'est-à-dire que dans $\cU$, les feuilles locales du feuilletage $\cF$ correspondent aux fibres du fibré $\psi$.  Le voisinage $\cU$ est une variété à coins non compacte.  Sans perte de généralité, en prenant $\cU$ plus petit au besoin, on peut supposer que les seules hypersurfaces bordantes $H$ de $M$ telles que $H\cap \cU\ne \emptyset$ sont celles contenant $p$.  Dans ce cas, si $\cM_1(\cU)$ dénote le sous-ensemble des hypersurfaces bordantes de $M$ contenant $p$, alors le fibré $\psi$ induit une bijection 
\begin{equation}
    \psi_{\#}: \cM_1(\cU)\to \cM_1(\cV).
\label{feuille.1a}\end{equation}
Pour $H\in\cM_1(\cU)$, remarquons que la restriction de $\psi$ à $H\cap \cU$ induit un fibré $\psi: H\cap \cU\to B_{H}$ avec $B_H$ une hypersurface bordante de $\cV$.  Nos hypothèses sur le feuilletage $\cF$ nous assurent qu'il existe une variété à coins $C_H$ et des fibrés $\psi_H: S_{H,\cU} \to C_H$ sur $S_{H,\cU}:=\phi_H(H\cap\cU)$ et $\phi_{\cV,H}: B_{H}\to C_H$ induisant un diagramme commutatif
\begin{equation}
\xymatrix{
  H\cap \cU \ar[d]^{\psi}\ar[r]^{\phi_H} & S_{H,\cU} \ar[d]^{\psi_H} \\
   B_H \ar[r]^{\phi_{\cV,H}} & C_H.
}
\label{feuille.1}\end{equation}  
\begin{lemme}
Les fibrés $\phi_{\cV,H}: B_H\to C_H$  engendrent une structure de fibrés itérés $\phi_{\cV}$ pour $\cV$.
\label{feuille.2}\end{lemme}
\begin{proof}
Les hypersurfaces bordantes de $\cV$ sont de la forme $B_H$ pour $H$ une hypersurface bordante de $M$ contenant $p$.  On peut donc prendre comme ordre partiel sur $\cM_1(\cV)$ celui induit par l'ordre partiel de $\cM_1(M)$.  Si $G,H\in \cM_1(M)$ sont telles que $p\in G\cap H$ et $H<G$, alors posons 
$$
C_{GH}:= \psi_G(S_{GH,\cU}), \quad \mbox{où} \quad S_{GH,\cU}:=S_{GH}\cap \phi_G(G\cap \cU).
$$
Par nos hypothèses sur le feuilletage $\cF$, $C_{GH}$ est l'une des hypersurfaces bordantes de $C_G$ et toutes les hypersurfaces bordantes de $C_G$ sont de cette forme.  Clairement, il existe un unique fibré $\phi_{GH,\cF}: C_{GH}\to C_H$ induisant le diagramme commutatif
\begin{equation}
\xymatrix{
  H\cap G\cap \cU \ar[rrr]^{\psi} \ar[rd]^{\phi_G}\ar[dd]^{\phi_H} &&& B_H\cap B_G \ar[ld]^{\phi_{G,\cF}}\ar[dd]^{\phi_{H,\cF}} \\
  & S_{GH,\cU}\ar[ld]^{\phi_{GH}}\ar[r]^{\psi_G} & C_{GH} \ar[rd]^{\phi_{GH,\cF}} & \\
  S_{H,\cU}\ar[rrr]^{\psi_H} &&& C_H.  
}
\label{feuille.3}\end{equation}
Cela confirme que les fibrés $\phi_{\cV,H}$ induisent bien une structure de fibrés itérés sur $\cV$.
\end{proof}

Par le Lemme~\ref{feuille.2}, $\gp\circ \psi_{\#}^{-1}$ est une fonction de pondération pour $(\cV,\phi_{\cV})$.  Supposons maintenant qu'il existe une distance $\gp$-tordue $\rho$ qui soit constante le long des feuilles de $\cF$.  En particulier, sur $\cU$, $\rho$ descend pour définir sur $\cV$ une distance $\gp\circ\psi_{\#}^{-1}$-tordue $\rho_{\cV}$.  Par la discussion précédente, cela induit un algébroïde de Lie ${}^{\cT}T\cV$ sur $\cV$.  Soit ${}^{\cT}T^*\cV$ le fibré dual et posons
$$
     {}^{\cT}T^*(\cU/\cF):= \psi^* {}^{\cT}T^*\cV. 
$$    
Comme la distance $\gp$-tordue est définie globalement et est constante le long des feuilles de $\cF$, on voit plus généralement qu'on peut définir un fibré vectoriel ${}^{\cT}T^*(M/\cF)\to M$ correspondant à ${}^{\cT}T^*(\cU/\cF)$ sur  un ouvert $\cU$ admettant une carte locale du feuilletage $\cF$ comme considéré ci-haut.  En utilisant l'application d'ancrage de ${}^{\cT}T\cV$, remarquons que sur $M\setminus \pa M$, la restriction du fibré vectoriel ${}^{\cT}T^*(M/\cF)$ s'identifie canoniquement avec la restriction du fibré $T^*(M/\cF)$, lui-même canoniquement un sous-fibré vectoriel du fibré cotangent $T^*M$.

\begin{definition}
Pour $(M,\phi)$ une variété à coins fibrés, $\gp$ une fonction de pondération à valeurs dans $\{0,\frac12\}$, $\cF$ un $(\phi,\gp)$-feuilletage et $\rho$ une distance $\gp$-tordue constante le long des feuilles de $\cF$, \textbf{l'espace des champs vectoriels $(\cF,\gp)$-tordus quasi-feuilletés au bord}, dénoté $\cV_{\cF,\cT}(M)$, est constitué des champs vectoriels bordants $\xi$ tels que
\begin{enumerate}
\item $\xi|_H\in\CI(H;T\cF|_H)$ pour toute hypersurface bordante $H\in\cM_1(M)$ telle que $\gp(H)=\frac12$;
\item $\eta(\xi)\in \CI(M)$ pour tout $\eta\in \CI(M;{}^{\cT}T^*(M/\cF))$;
\end{enumerate} 
où la deuxième condition signifie que $\eta(\xi)$, initialement défini sur $M\setminus \pa M$, se prolonge en une fonction lisse sur $M$.
\label{tqfb.13b}\end{definition}
Vérifions que ce sous-espace est bien  une sous-algèbre de Lie de $\cV_b(M)$.
\begin{lemme}
Le sous-espace $\cV_{\cF,\cT}(M)$ est une sous-algèbre de Lie de $\cV_b(M)$.
\label{tqfb.13c}\end{lemme}
\begin{proof}
Clairement, la condition (1) de la Définition~\ref{tqfb.13b} est fermée par rapport au crochet de Lie.  Pour la deuxième condition, remarquons que si $\xi_1,\xi_2\in \cV_{\cF,\cT}(M)$ et $\eta\in \CI(M;{}^{\cT}T^*(M/\cF))$, alors par la formule de Cartan,
$$
    \eta([\xi_1,\xi_2])= -d\eta(\xi_1,\xi_2)+ \cL_{\xi_1}(\eta(\xi_2))-\cL_{\xi_2}(\eta(\xi_1)).
$$
Par hypothèse, $\eta(\xi_1),\eta(\xi_2)\in\CI(M)$, donc puisque $\xi_1$ et $\xi_2$ sont des champs vectoriels lisses, on a automatiquement que $\cL_{\xi_1}(\eta(\xi_2)),\cL_{\xi_2}(\eta(\xi_1))\in\CI(M)$.  Pour établir que $\eta([\xi_1,\xi_2])\in\CI(M)$, il suffit donc de montrer que $d\eta(\xi_1,\xi_2)\in\CI(M)$. Pour ce faire, il suffit d'établir ce résultat dans un ouvert $\cU$ de $M$ dans lequel les feuilles locales de $\cF$ correspondent aux fibres d'un fibré $\psi: \cU\to \cV$ comme ci-haut.  Or, dans les coordonnées \eqref{tqfb.4} sur $\cV$, une base locale de sections de ${}^{\cT}T^*\cV$ est donnée par
\begin{equation}
\frac{ \rho_1^{-\gp}d\rho_1}{v_{\frac12}}, \frac{dy_1^1}{v_{\frac12}v_1}, \ldots, \frac{dy_1^{\ell_1}}{v_{\frac12}v_1}, \ldots, \frac{\rho_k^{-\gp}d\rho_k}{v_{\frac12}}, \frac{dy_k^1}{v_{\frac12}v_k},\ldots, \frac{dy_k^{\ell_k}}{v_{\frac12}v_k}, \frac{dz^1}{v_{\frac12}},\ldots, \frac{dz^q}{v_{\frac12}}.
\label{tqfb.13d}\end{equation}  
En tirant en arrière sur $\cU\subset M$, cela montre que pour $\eta\in\CI(\cU;\psi^*({}^{\cT}T^*\cV))$, on a que
$$
   d\eta\in \CI(\cU;{}^{b}T^*M\wedge \psi^*({}^{\cT}T^*\cV)),
$$
où ${}^{b}T^*M$ est le fibré cotangent bordant, c'est-à-dire le dual du fibré tangent bordant ${}^{b}TM$ associé à l'algèbre de Lie $\cV_b(M)$ des champs vectoriels bordants de $M$.  Comme $\xi_1$ et $\xi_2$ sont en particulier des champs vectoriels bordants, on en déduit que $d\eta(\xi_1,\xi_2)\in\CI(\cU)$, d'où le résultat.
\end{proof}

Dans un ouvert $\cU$ de $M$ sur lequel les feuilles du feuilletage $\cF$ correspondent aux fibres d'un fibré $\psi:\cU\to \cV$, considérons des coordonnées \eqref{tqfb.4} sur $\cV$.  Si $w=(w^1,\ldots,w^p)$ sont des fonctions telles que 
\begin{equation}
  (x_1,y_1,\ldots,x_k,y_k,z,w)
\label{tqfb.13g}\end{equation}
est un système de coordonnées sur $M$, alors en tant que $\CI(M)$-module, $\cV_{\cF,\cT}(M)$ est localement engendré par les champs vectoriels 
\begin{multline}
v_{\frac12}v_1x_1\frac{\pa}{\pa x_1}, v_{\frac12}v_1\frac{\pa}{\pa y_1}, v_{\frac12}v_2\left((1-\nu_{H_2})x_2\frac{\pa}{\pa x_2}-(1-\nu_{H_1})x_1\frac{\pa}{\pa x_1}\right), v_{\frac12}v_2\frac{\pa}{\pa y_2}, \ldots, \\
v_{\frac12}v_k\left((1-\nu_{H_k})x_k\frac{\pa}{\pa x_k}-(1-\nu_{H_{k-1}})x_{k-1}\frac{\pa}{\pa x_{k-1}}\right), v_{\frac12}v_k\frac{\pa}{\pa y_k}, v_{\frac12}\frac{\pa}{\pa z}, \frac{\pa}{\pa w},
\label{tqfb.13e}\end{multline}
où $\frac{\pa}{\pa w}$ dénote $\left( \frac{\pa}{\pa w^1},\ldots, \frac{\pa}{\pa w^p} \right)$.  Cela montre que $\cV_{\cF,\cT}(M)$ correspond à l'espace des sections globales d'un faisceau localement libre de rang $m=\dim M$.  Par le théorème de Serre-Swan, il existe donc un fibré vectoriel ${}^{\cF,\cT}TM\to M$ et une application lisse $\iota_{\cF,\cT}: {}^{\cF,\cT}TM\to TM$ linéaire dans chaque fibre induisant une identification
$$
   \cV_{\cF,\cT}(M)= (\iota_{\cF,\cT})_*\CI(M;{}^{\cF,\cT}TM).
$$
En particulier, cette identification induit une structure d'algèbre de Lie sur $\CI(M;{}^{\cF,\cT}TM)$, conférant à ${}^{\cF,\cT}TM$ une structure d'algébroïde de Lie avec application d'ancrage $\iota_{\cF,\cT}$.  
\begin{remarque}
Dans un ouvert $\cW$ de $M$ tel que $\cW\cap v_{\frac12}^{-1}(0)=\emptyset$, c'est-à-dire un ouvert n'intersectant pas les hypersurfaces bordantes $H$ telles que $\gp(H)=\frac12$, la fonction de pondération devient triviale et  
$
    {}^{\cF,\cT}TM|_{\cW}={}^{\cT}TM|_{\cW}
$
par la condition $(2)$ de la Définition~\ref{feuille.5}.   Pour définir ${}^{\cF,\cT}TM$, il suffit donc de supposer en fait que le $(\phi,\gp)$-feuilletage $\cF$ n'est défini que dans un voisinage de $v_{\frac12}^{-1}(0)$.
\label{feuille.4}\end{remarque}

 \begin{remarque}
Pour les déformations de Taub-NUT d'ordre 1 de variétés hyperkählériennes toriques, le $\phi$-feuilletage $\cF$ aura pour feuilles les orbites d'une action de $\bbR$ sur $M$.  Cette action ne sera typiquement pas libre, mais elle le sera dans un voisinage de $v_{\frac12}^{-1}(0)$.  
\label{tqfb.16}\end{remarque}

Cela nous permet de formuler la définition principale de cette section.

\begin{definition}
Soient $(M,\phi)$ une variété à coins fibrés et $\gp: \cM_1(M)\to \{0,\frac12\}$ un choix de fonction de pondération.  Soient $\cF$ un $(\phi,\gp)$-feuilletage défini dans un voisinage de $v_{\frac12}^{-1}(0)$, $\rho$ un choix de distance $\gp$-tordue constante le long des feuilles de $\cF$ et ${}^{\cF,\cT}TM\to M$ le fibré associé.  Alors une \textbf{métrique $(\cF,\gp)$-tordue quasi-feuilletée au bord}  sur $M\setminus \pa M$ est une métrique riemannienne $g$ sur $M\setminus \pa M$ se prolongeant en une métrique euclidienne (lisse jusqu'au bord) pour le fibré ${}^{\cF,\cT}TM$ sur $M$.
\label{tqfb.15}\end{definition}
Par \cite{ALN04,Bui}, les métriques $(\cF,\gp)$-tordues quasi-feuilletées au bord sont automatiquement complètes à géométrie bornée.  Dans un ouvert $\cU$ de $M$ où les feuilles locales de $\cF$ correspondent aux fibres d'un fibré $\psi: \cU\to \cV$, un exemple de métrique $(\cF,\gp)$-tordue quasi-feuilletée au bord est donné par 
\begin{equation}
 g_{\psi,\cT} = \psi^{*}g_\cT + \kappa,
\label{tqfb.17}\end{equation}
où $g_\cT$ est une métrique $\gp$-tordue $\QFB$ sur $\cV$ et $\kappa\in\CI(\cU; S^2(T^*M))$ est un $2$-tenseur symétrique induisant une métrique riemannienne sur chaque fibre de $\psi$.  Les exemples de métriques $(\cF,\cT)$-tordues quasi-feuilletée au bord hyperkählériennes que nous allons considérer seront localement de cette forme.  

Les modèles locaux \eqref{tqfb.17} nous permettent d'estimer la croissance du volume de ces métriques.  Rappelons que la \textbf{croissance du volume} (à l'infini) d'une variété riemannienne complète $(M,g)$ est dite \textbf{d'ordre $k$} si pour tout $p\in M$, il existe une constante $C>0$ telle que 
$$
 \frac{r^{k}}{C} \le \vol_g(B_g(p,r)) \le Cr^{k} \quad \mbox{pour} \quad \forall \; r\ge 1,
$$
où $B_g(p,r)$ est la boule de rayon $r$ (mesuré avec $g$) centrée en $p$ et $\vol_g(B_g(p,r))$ est le volume de cette boule mesuré avec la forme volume de la métrique $g$.  
\begin{lemme}
Soient $(M,\phi)$ une variété à coins fibrés de type $\QAC$, $\gp: \cM_1(M)\to \{0,\frac12\}$ un choix de fonction de pondération et $\rho$ une distance $\gp$-tordue.  Alors les métriques $\gp$-tordues $\QAC$ associées ont une  croissance du volume d'ordre $\dim M$.  Si de plus $\cF$ est un $(\phi,\gp)$-feuilletage défini dans un voisinage de $v_{\frac12}^{-1}(0)$ tel que $\rho$ est constante le long de ses feuilles, alors les métriques $(\cF,\gp)$-tordue quasi-feuilletées au bord associées ont une croissance du volume d'ordre $\dim M-f$, où $f$ est la dimension des feuilles de $\cF$.
\label{vg.1}\end{lemme}
\begin{proof}
Pour les métriques $\gp$-tordues $\QAC$, cela découle de \cite[Corollary~3.8]{CR2023}.  Pour les métriques $(\cF,\gp)$-tordues quasi-feuilletées jusqu'au bord, on sait par \cite[Corollary~3.6]{ALN04} qu'elles sont toutes mutuellement quasi-isométriques, puisqu'elles correspondent à une classe de métriques associée à une structure de Lie à l'infini.  On peut donc utiliser les modèles locaux \eqref{tqfb.17} pour estimer la croissance de leurs volumes.  Or, dans ces modèles, on peut choisir $\kappa$ indépendant de $\rho$, si bien que le volume des fibres est indépendant de $\rho$.  La croissance du volume de ce modèle a donc le même ordre que celle de la métrique $g_\cT$ sur la base, qui par le résultat précédent à une croissance du volume d'ordre $\dim M-f$.    
\end{proof}

\section{Métriques hyperkählériennes toriques quasi-asymptotiquement coniques} \label{qac.0}

Soit $(M,g)$ une variété hyperkählérienne torique complète simplement connexe de type topologique fini.  Supposons qu'elle a une croissance maximale du volume, c'est-à-dire d'ordre $\dim M$ (l'ordre maximal pour une variété Ricci plate par l'inégalité de Bishop-Gromov).
Par la classification de Bielawski \cite[Theorem~1]{Bielawski}, c'est un quotient hyperkählérien de $\bbH^d$ par $T^{d-m}$ avec $m=\dim M\le d$, où $T^{d-m}$ est un sous-tore du tore diagonal maximal dans $\Sp(d)$.   En suivant l'approche de \cite{Melrose_London,DR}, le but de cette section sera de montrer que de telles métriques sont typiquement quasi-asymptotiquement coniques, c'est-à-dire quasi-fibrée au bord pour une variété à coins fibrés de type $\QAC$.  

Rappelons que l'espace vectoriel quaternionique $\bbH^d$ muni de sa métrique euclidienne canonique $g_{\bbH^d}$ est une variété hyperkählérienne plate avec structures complexes $I_1$, $I_2$ et $I_3$ induites par la la multiplication à gauche par $i,j$ et $k$.  En utilisant l'identification $\bbH^d=\bbC^d\times \bbC^d$ induite par 
$$
   \bbC^d\times \bbC^d\ni (z,w)\to z+wj= (z_1+w_1j, \ldots, z_d+w_dj)\in \bbH^d,
$$
les formes symplectiques $\omega_1$, $\omega_2$ et $\omega_3$ associées à $I_1$, $I_2$ et $I_3$ prennent la forme
\begin{equation}
\omega_1= \frac{\sqrt{-1}}2 \sum_{k=1}^d \left( dz_k\wedge d\overline{z}_k+ dw_k\wedge d\overline{w}_k \right) \quad \mbox{et} \quad
\omega_2+\sqrt{-1}\omega_3= \sum_{k=1}^d dz_k\wedge dw_k.
\label{qac.1}\end{equation}
Le sous-groupe de $\GL(4d,\bbR)$ préservant $g_{\bbH^d}$, $\omega_1$, $\omega_2$ et $\omega_3$ est dénoté $\Sp(d)$.  Sous l'identification 
\begin{equation}
 \Sp(d)\cong \{ A\in \mathbf{M}_d(\bbH)\; | \; A \overline{A}^t=\Id\},
\label{qac.2}\end{equation}
l'action de $\Sp(d)$ sur $\bbH^d$ est donnée par 
\begin{equation}
A\cdot q:= q  \overline{A}^t \quad \mbox{pour} \; q\in\bbH^d \; \mbox{et} \; A\in \Sp(d).
\label{qac.3}\end{equation}
Comme les structures complexes $I_1$, $I_2$ et $I_3$ sont induites par la multiplication à gauche par $i,j$ et $k$, celles-ci sont en particulier bien préservées par l'action de $\Sp(d)$ sur $\bbH^d$.  Sous l'identification \eqref{qac.2}, le tore diagonal maximal $T^d$ est donné par
$$
 T^d=\{ \diag(e^{it_1},\ldots, e^{it_d})\in \mathbf{M}_d(\bbC)\subset \mathbf{M}_d(\bbH) \;|\; t_1,\ldots, t_d\in\bbR\},
$$
où $\diag(e^{it_1},\ldots, e^{it_d})$ représente la matrice diagonale
$$
\begin{bmatrix}
    e^{it_1} & & 0\\
    & \ddots & \\
   0 & & e^{it_d}
  \end{bmatrix}.
$$
L'action de $T^d$ sur $\bbH^d=\bbC^d\times \bbC^d$ est donnée explicitement par
\begin{equation}
  \diag(e^{it_1},\ldots,e^{it_d})\cdot (z,w)= (e^{-it_1}z_1,\ldots, e^{-it_d}z_d, e^{it_1}w_1,\ldots, e^{it_d}w_d).
\label{action.1}\end{equation}
Si $\gt^d$ dénote l'algèbre de Lie de $T^d$, alors une \textbf{application moment hyperhamiltonienne} associée à l'action de $T^d$ sur $\bbH^d$ est une application $\mu=(\mu_1,\mu_2,\mu_3)$ avec $\mu_i: \bbH^d\to (\gt^d)^*$ telle que
$$
   d(\mu_i(t))= -\iota_{\xi_t}\omega_i \quad \forall \;t\in \gt^d,
$$
où $\xi_t$ est le champ vectoriel sur $\bbH^d$ correspondant à l'action infinitésimale engendrée par $t$.  Une telle application est bien définie à un élément de $\bbR^3\otimes \gt^d= \gt^d\oplus \gt^d\oplus \gt^d$ près. En demandant que cette application s'annule à l'origine et en utilisant  l'identification canonique $\gt^d\cong \bbR^d$ pour laquelle l'application exponentielle de $T^d$ prend la forme 
$$
 \begin{array}{lccc}
 \exp: & \gt^d & \to & T^d \\
         & (t_1,\ldots,t_d) & \mapsto & \diag(e^{it_1},\ldots, e^{it_d}),
 \end{array}
$$
on a  explicitement que
\begin{equation}
\mu_1= \frac12\sum_{k=1}^d \lrp{-|z_k|^2+ |w_k|^2}e^*_k \quad \mbox{et} \quad \mu_2+\sqrt{-1}\mu_3= \sum_{k=1}^d \sqrt{-1}z_kw_k e_k^*,
\label{am.1}\end{equation}
où $(e_1,\ldots,e_d)$ est la base canonique de $\bbR^d\cong \gt^d$ et $(e_1^*,\ldots, e_d^*)$ est sa base duale.  L'action de $T^d$ sur $\bbH^d$ induit une stratification de $\bbH^d$ en termes des stabilisateurs de cette action \cite{DK2000,AM2011}.  Plus précisément dans ce cas, pour $\cI\subset \{1,\ldots, d\}$, le sous-espace 
\begin{equation}
 V_{\cI}:= \{ (q_1,\ldots, q_d)\in \bbH^d\; | \; q_p=0 \; \mbox{pour} \; p\notin \cI\}
\label{qac.4}\end{equation}
est la strate (fermée) associée au stabilisateur 
\begin{equation}
 T_{\cI}:= \{ \diag(e^{it_1},\ldots, e^{it_d})\in T^d\; | \; t_p=0 \; \mbox{pour} \; p\in \cI\}.
\label{qac.4a}\end{equation}
Soient $\gt_{\cI}$ l'algèbre de Lie de $T_{\cI}$ et $\iota_{\cI}: \gt_{\cI}\hookrightarrow \gt^d$  l'inclusion naturelle.  Alors remarquons que 
\begin{equation}
V_{\cI}= (\iota_{\cI}^*\circ \mu)^{-1}(0), 
\label{qac.4b}\end{equation}
où $\iota_{\cI}^*: \bbR^3\otimes (\gt^d)^*\to \bbR^3\otimes \gt_{\cI}^*$ est l'application duale induite par $\iota_{\cI}$.

Soit maintenant $N$ un sous-tore de $T^d$ avec algèbre de Lie $\gn\subset \gt^d\cong \bbR^d$.  Alors l'application moment hyperhamiltonienne associée à l'action de $N$ sur $\bbH^d$ et s'annulant à l'origine est donnée par 
\begin{equation}
\mu_N=(\mu_{N,1},\mu_{N,2},\mu_{N,3})\quad \mbox{avec} \quad \mu_{N,j}=\iota_\gn^* \circ \mu_j,
\label{am.2}\end{equation}
 où $\iota_\gn^*: (\gt^d)^*\to \gn^*$ est l'application linéaire surjective duale à l'inclusion naturelle $\iota_\gn: \gn\to \gt^d$.

En termes d'algèbres de Lie, la suite exacte 
$$
\xymatrix{
 0\ar[r] & N \ar[r] & T^d\ar[r] & T^d/N \ar[r] & 0
}
$$
correspond à une suite exacte
\begin{equation}
        \xymatrix{
 0\ar[r] & \gn \ar[r]^{\iota_{\gn}} & \gt^d\ar[r]^{\beta} & \bbR^n \ar[r] & 0,
}
\label{qac.5}\end{equation}
où $n=d-\dim_{\bbR} N$ et $\beta$ est une application linéaire envoyant $e_i$ sur $u_i\in\bbZ^n$ avec $\{u_1,\ldots,u_d\}\subset \bbZ^n$ un ensemble primitif de $\bbZ^n$ engendrant $\bbR^n$.  L'algèbre de Lie du groupe quotient $T^d/N$ est donc identifiée avec $\bbR^n$ et 
$$
 T^d/N \cong T^n:= \bbR^n/\bbZ^n.
$$ 
Inversement, un tel ensemble primitif $\{u_1,\ldots, u_d\}$ de $\bbZ^n$ engendrant $\bbR^n$ définit un sous-tore $N$ de $T^d$.  La suite exacte duale de \eqref{qac.5} est donnée par
\begin{equation}
        \xymatrix{
 0\ar[r] & (\bbR^n)^* \ar[r]^{\beta^*} & (\gt^d)^*\ar[r]^{\iota_{\gn}^*} & \gn^* \ar[r] & 0,
}
\label{qac.6}\end{equation}
où $\beta^*(s)= \sum_{k=1}^d \langle s, u_k \rangle e_k^*$.  On fera l'hypothèse suivante sur le sous-ensemble $\{u_1,\ldots, u_d\}$ définissant l'application $\beta$.
\begin{hypothese}
Le sous-tore $N$ est induit par une suite exacte \eqref{qac.5} telle que tout sous-ensemble de $n$ vecteurs linéairement indépendants de $\{u_1,\ldots,u_d\}$ forme une base du réseau $\bbZ^n$.
\label{qac.7}\end{hypothese}

Pour $\zeta\in \bbR^3\otimes \gn^*$, on peut alors considérer le quotient hyperkählérien 
\begin{equation}
  M_{\zeta}:= \mu_N^{-1}(-\zeta)/N.  
\label{qac.8}\end{equation}
Dans \cite{BD}, un critère simple est donné pour déterminer lorsque ce quotient hyperkählérien est lisse.  Pour le formuler, soit $\tau=(\tau_1,\ldots,\tau_d)\in \bbR^3\otimes (\gt^d)^*\cong \bbR^3\otimes \bbR^d$ tel que $\iota^*_{\gn} \tau =\zeta$ et posons 
$$
    H_k=\{s\in\bbR^3\otimes (\bbR^n)^*\; |\; \langle s, u_k\rangle =\tau_k\} \subset \bbR^3\otimes \bbR^n
$$
pour $k\in\{1,\ldots,d\}$.
\begin{theoreme}\emph{(\cite[Theorem~3.2]{BD})}
Supposons que les hyperplans $H_k$ sont distincts.  Alors le quotient hyperkählérien $M_{\zeta}$ est lisse si et seulement si l'intersection de $n+1$ hyperplans distincts $H_k$ est toujours vide et si pour $n$ hyperplans distincts $H_{k_1},\ldots, H_{k_n}$ d'intersection non vide, l'ensemble $\{u_1,\ldots,u_n\}$ est une base du réseau $\bbZ^n$.  En particulier, sous l'Hypothèse~\ref{qac.7}, $M_{\zeta}$ sera lisse pour $\zeta\in \bbR^3\otimes \gn^*$ générique.
\label{qac.9}\end{theoreme}
\begin{remarque}
La définition des hyperplans $H_k$ dépend du choix de $\tau$, mais un autre choix de $\tau$ ne fait que déplacer tous les hyperplans par une même translation dans $\bbR^3\otimes (\bbR^n)^*$.  
\label{qac.10}\end{remarque}
\begin{remarque}
Sous l'Hypothèse~\ref{qac.7}, on sait par la partie (a) de la preuve de \cite[Theorem~3.2]{BD} qu'un sous-ensemble de $\{u_1,\ldots,u_d\}$ constitué de vecteurs linéairement indépendants est contenu dans une base du réseau $\bbZ^n$. 
\label{qac.11}\end{remarque}
\begin{remarque}
Lorsque l'Hypothèse~\ref{qac.7} n'est pas satisfaite, la preuve de \cite[Theorem~3.2]{BD} montre que le quotient hyperkählérien est singulier peu importe le choix de $\zeta$.
\label{qac.10b}\end{remarque}

Pour décrire la géométrie à l'infini du quotient hyperkählérien $M_{\zeta}$, nous allons introduire une compactification de $\bbH^d$ adaptée à l'action du sous-tore $N$.  D'abord, par \cite{DK2000, AM2011}, en tant que $N$-espace, $\bbH^d$ admet une stratification naturelle qui à chaque sous-groupe stabilisateur $G\subset N$ de l'action de $N$ associe la strate ouverte constituée des points de $\bbH^d$ ayant $G$ comme stabilisateur.  Or, par la description \eqref{qac.4a} des stabilisateurs de l'action de $T^d$, les stabilisateurs de l'action de $N$ sur $\bbH^d$ sont de la forme
$N_{\cI}=N\cap T_{\cI}$ avec strate fermée $V_{\cI}$ pour $\cI\subset \{1,\ldots,d\}$ tel que $N\cap T_{\cJ}\subsetneqq N\cap T_{\cI}$ pour tout $\cJ\subset \{1,\ldots,d\}$ tel que $\cI\subsetneqq \cJ$.  Si $\gn_{\cI}$ et $\gt_{\cI}$ dénotent les algèbres de Lie de $N_{\cI}$ et $T_{\cI}$ respectivement, alors l'analogue de la suite exacte \eqref{qac.5} pour le stabilisateur $N_{\cI}$ est la suite exacte
\begin{equation}
       \xymatrix{
 0\ar[r] & \gn_{\cI} \ar[r]^{\iota} & \gt_{\cI}\ar[r]^{\beta_{\cI}} & W_{\cI} \ar[r] & 0,
}
\label{qac.11b}\end{equation}
où $W_{\cI}:= \beta(\gt_{\cI})\subset \bbR^n$ et $\beta_{\cI}:= \beta|_{\gt_\cI}$.
\begin{lemme}
Si $N$ satisfait à l'Hypothèse~\ref{qac.7}, alors $R_\cI:=W_{\cI}\cap \bbZ^n$ est un réseau de $W_\cI$ et tout sous-ensemble de $\dim W_{\cI}$ vecteurs linéairement indépendants de $\{u_i\; | \; i\in \cI^c\}$ forme une base de $R_{\cI}$.  De plus, $N_{\cI}$ est un tore satisfaisant à l'Hypothèse~\ref{qac.7} par rapport à la suite \eqref{qac.11b} et au réseau $R_{\cI}$.  
\label{qac.11c}\end{lemme}
\begin{proof}
Par la Remarque~\ref{qac.11}, un sous-ensemble de $\dim W_{\cI}$ vecteurs indépendants de $\{u_i\; | \; i\in \cI^c\}$ est contenu dans une base de $\bbZ^n$.  En particulier, un tel sous-ensemble engendre $R_{\cI}$, ce qui montre que $R_{\cI}$ est bien un réseau de $W_\cI$.  Cet argument montre aussi que tout sous-ensemble de $\dim W_{\cI}$ vecteurs linéairement indépendants de $\{u_i\; | \; i\in \cI^c\}$ forme une base de $R_{\cI}$.  On en conclut que $N_\cI$ est un sous-groupe connexe, donc que c'est un tore. Le fait qu'il satisfait à l'Hypothèse~\ref{qac.7}  par rapport à la suite \eqref{qac.11b} et au réseau $R_{\cI}$ découle de la première partie de l'énoncé du lemme.     
\end{proof}

Soit $\overline{\bbH^d}$ la compactification radiale de $\bbH^d$ au sens de \cite{MelroseGST}.  Le bord $\pa\overline{\bbH^d}$ est naturellement identifié avec la sphère unité $\bbS(\bbH^d)$ de $\bbH^d$.  En particulier, comme l'action de $N$ sur $\bbH^d$ est unitaire, elle se restreint à une action sur $\bbS(\bbH^d)$ et se prolonge naturellement en une action sur $\overline{\bbH^d}$.  Soient $s_1,\ldots,s_{\ell}$ les strates ouvertes de $\pa\overline{\bbH^d}$ induites par l'action de $N$.  Supposons qu'elles soient énumérées de manière à respecter l'ordre partiel des strates, c'est-à-dire que
$$
     s_i< s_j \; \Longrightarrow \; i<j.
$$
En particulier, $s_{\ell}$ est la strate régulière de $\pa \overline{\bbH^d}$.
\begin{definition}\emph{(\cf \cite[Definition~4.5]{DR})}  La \textbf{compactification $\QAC$} de $\bbH^d$ par rapport à l'action de $N$ est la variété à coins
$$
  \widetilde{\bbH^d}= [\overline{\bbH^d};\overline{s}_1,\ldots, \overline{s}_{\ell-1}].
$$
\label{qac.12}\end{definition}

L'ordre dans lequel les éclatements sont effectués est important.  Puisque la strate $s_1$ est minimale, $s_1=\overline{s}_1$ est une sous-variété fermée, donc son éclatement dans $\bbH^d$ est bien défini, de même que celui de toute autre strate minimale.  Pour les autres éclatements, remarquons qu'avant d'éclater $\overline{s}_i$, il faut avoir éclaté les strates $\overline{s}_j$ telles que $s_j<s_i$, de sorte que par \cite[Proposition~7.4 et Theorem~7.5]{AM2011}, le relèvement de $\overline{s}_i$ est une $p$-sous-variété de $[\overline{\bbH^d};\overline{s}_1,\ldots, \overline{s}_{i-1}]$ et son éclatement est donc bien défini.  La variété à coins $\widetilde{\bbH^d}$ possède donc $\ell$ hypersurfaces bordantes $H_1,\ldots, H_{\ell}$ correspondant à chacune des strates $s_1,\ldots, s_{\ell}$ de $\pa\overline{\bbH^d}$.  Cela induit un ordre partiel sur $\cM_1(\widetilde{\bbH^d})$ conférant à $\widetilde{\bbH^d}$ une structure de fibrés itérés.  Les différents fibrés $\phi_i: H_i\to S_i$ sont induits par les applications de contraction des éclatements.  Plus précisément, $S_i$ correspond à la variété à coins fibrés résolvant la strate fermée $\overline{s}_i$.  Pour l'unique hypersurface bordante maximale 
$$
     H_{\ell}= S_\ell= [\pa \overline{\bbH^d};\overline{s}_1,\ldots, \overline{s}_{\ell-1}],
$$
le fibré $\phi_\ell: H_{\ell}\to S_{\ell}$ correspond à l'application identité.  La variété à coins fibrés $(\widetilde{\bbH^d},\phi)$ avec $\phi=(\phi_1,\ldots,\phi_{\ell})$ est donc de type $\QAC$.  Pour décrire les fibrés $\phi_i$ pour $i<\ell$, considérons d'abord le cas $i=1$.  L'application de contraction $[\overline{\bbH^d};\overline{s}_1]\to \overline{\bbH^d}$ induit alors un fibré
$$
    \overline{\phi}_1: \overline{H}_1\to S_1=\overline{s}_1
$$  
sur l'hypersurface bordante $\overline{H}_1$ de $[\overline{\bbH^d};\overline{s}_1]$ créée par l'éclatement de $\overline{s}_1$.  Si $\overline{s}_1=\pa\overline{V}_{\cI_1}$ pour $\cI_1\subset \{1,\ldots, d\}$, alors les fibres de $\overline{\phi}_1$ sont naturellement identifiées avec la compactification radiale $\overline{V}_{\cI_1^c}$ de $V_{\cI_1^c}$, où $\cI^c_1$ est le complément de $\cI_1$ dans $\{1,\ldots, d\}$.  De plus, on a une décomposition $\overline{H}_1= S_1\times \overline{V}_{\cI_1^c}$  dans laquelle le fibré $\overline{\phi}_1$ correspond au fibré trivial
$$
   \pr_1: S_1\times \overline{V}_{\cI^c_1}\to S_1
$$
donné par la projection sur le premier facteur. Or, en termes de cette décomposition, les éclatements subséquents correspondent sur $\overline{V}_{\cI^c_1}$ aux éclatements des intersections de $\overline{s}_2,\ldots,\overline{s}_{\ell-1}$ avec $\pa\overline{V}_{\cI^c_1}$, de sorte que sur $\widetilde{\bbH^d}$, $\overline{V}_{\cI^c_1}$ se relève en une compactification $\QAC$ $\widetilde{V}_{\cI^c_1}$ de $V_{\cI^c_1}$ par rapport à l'action du stabilisateur $N_{\cI_1}=T_{\cI_1}\cap N$ sur $V_{\cI^c_1}$.  Cela induit une décomposition $H_1= S_1\times \widetilde{V}_{\cI_1^c}$ dans laquelle $\phi_1$ correspond au fibré trivial
$$
        \pr_1: S_1\times \widetilde{V}_{\cI^c_1}\to S_1
$$
donné par la projection sur le premier facteur.  Le même genre de raisonnement montre plus généralement que pour $s_i$ pas forcément minimale avec $\overline{s}_i= \pa\overline{V}_{\cI_i}$, on a une décomposition $H_i=S_i\times \widetilde{V}_{\cI^c_i}$ dans laquelle $\phi_i$ correspond au fibré trivial
\begin{equation}
   \pr_1:S_i\times  \widetilde{V}_{\cI^c_i}\to S_i
\label{fibre.1}\end{equation}
et $\widetilde{V}_{\cI^c_i}$ est la compactification $\QAC$ de $V_{\cI^c_i}$ par rapport à l'action du stabilisateur $N_{\cI_i}=N\cap T_{\cI_i}$.
\begin{remarque}
La variété à coins fibrés $\widetilde{\bbH^d}$ est munie d'une classe canonique de fonctions bordantes totales $\QAC$-équivalentes, à savoir la classe correspondant au relèvement à $\widetilde{\bbH^d}$ d'une fonction bordante quelconque sur $\overline{\bbH^d}$.  Lorsqu'on parlera de métriques $\QAC$ sur $\widetilde{\bbH^d}$, ce sera par rapport à ce choix de classe de fonctions bordantes totales.
\label{qac.12b}\end{remarque}

Pour $\zeta\in \bbR^3\otimes \gn^*$, soit $\widetilde{\mu_N^{-1}(\zeta)}$ la fermeture de $\mu_N^{-1}(\zeta)$ dans $\widetilde{\bbH^d}$.  Dans ce cadre, on a l'analogue suivant du résultat de \cite[Theorem~4.6]{DR} pour les variétés de carquois.
\begin{theoreme}
Soit $N\subset T^d$ un sous-tore satisfaisant à l'Hypothèse~\ref{qac.7}.  Alors pour $\zeta\in \bbR^3\otimes \gn^*$  générique, $\widetilde{\mu_N^{-1}(-\zeta)}$ est une $p$-sous-variété de $\widetilde{\bbH^d}$ et la structure de fibrés itérés sur $\widetilde{\bbH^d}$ en induit une sur $\widetilde{\mu^{-1}(-\zeta)}$, c'est-à-dire que pour chaque $H_i\in \cM_1(\widetilde{\bbH^d})$, l'ensemble $H_i\cap \widetilde{\mu_N^{-1}(-\zeta)}$, lorsque non vide, est une hypersurface bordante de $\widetilde{\mu_N^{-1}(-\zeta)}$ avec fibré 
$$
  \phi_{N,i}: H_i\cap \widetilde{\mu_N^{-1}(-\zeta)}\to \Sigma_i
$$
induit par la restriction du fibré $\phi_i: H_i\to S_i$.  De plus, l'action de $N$ sur $\mu_N^{-1}(-\zeta)$ se prolonge en une action libre sur $\widetilde{\mu_N^{-1}(\zeta)}$ compatible avec la structure de fibrés itérés.
\label{qac.13}\end{theoreme}
\begin{proof}
La preuve est similaire à celle de \cite[Theorem~4.6]{DR}, mais plus simple du fait que $N$ est un groupe abélien.  En procédant par induction sur la profondeur de $\bbH^d$ en tant que $N$-espace, on peut supposer que le résultat est valide pour des cas de profondeur plus petite.  Considérons d'abord le cas où $H_i$ est minimale.  Soit $\cI\subset \{1,\ldots, d\}$ tel que $\overline{s}_i= \pa \overline{V}_{\cI}$.  Alors par la discussion précédente, il y a une décomposition
$$
    H_i= S_i\times \widetilde{V}_{\cI^c},
$$
où $\widetilde{V}_{\cI^c}$ est la compactification $\QAC$ de $V_{\cI^c}$ par rapport à l'action du stabilisateur $N_\cI=T_\cI\cap N$.  Si $\gn_{\cI}$ est l'algèbre de Lie de $N_\cI$, soit $\gn_{\cI}^{\perp}$ son complément orthogonal dans $\gn$.  En termes de la décomposition
\begin{equation}
\bbR^3\otimes \gn^*= (\bbR^3\otimes \gn_{\cI}^*)\oplus (\bbR^3\otimes (\gn^{\perp}_{\cI})^*),
\label{qac.14}\end{equation}
on a alors une décomposition correspondante pour l'application moment $\mu_N$, à savoir
\begin{equation}
 \mu_N= \mu_{N_{\cI}}+ \check{\mu}_{\cI^c}
\label{qac.15}\end{equation}
avec $\mu_{N_{\cI}}: \bbH^d\to \bbR^3\otimes \gn_{\cI}^*$ et $\check{\mu}_{\cI^c}: \bbH^d\to \bbR^3\otimes (\gn_{\cI}^{\perp})^*$.  En termes de  cette décomposition et de la décomposition 
\begin{equation}
 \bbH^d= V_{\cI^c}\oplus V_{\cI},
\label{qac.16}\end{equation}
l'équation $\mu_N(q)=-\zeta$ prend alors la forme
\begin{equation}
 \mu_{N_\cI}(\hat{q})=-\zeta_{\cI} \quad \mbox{et}  \quad \check{\mu}_{\cI^c}(\check{q})= -\zeta_{\cI}^\perp,
\label{qac.17}\end{equation}
où $\zeta=(\zeta_{\cI},\zeta_{\cI}^{\perp})$ dans la décomposition \eqref{qac.14} et $q=(\hat{q},\check{q})$ dans la décomposition \eqref{qac.16}.

Soient $\rho_{\cI}$ la distance par rapport à l'origine sur $V_{\cI}$ et $u_\cI=\frac{1}{\rho_{\cI}}$ la fonction bordante correspondante près de $\pa\overline{V}_\cI$ dans $\overline{V}_{\cI}$.  Soit $\check{\omega}$ des coordonnées sur la sphère unité $\bbS(V_{\cI})$ de $V_{\cI}$, de sorte que $(\rho_{\cI},\check{\omega})$ correspond à $\check{q}=\rho_{\cI}\check{\omega}$.  Dans les coordonnées $(\hat{q}, u_{\cI}, \check{\omega})$ près de l'intérieur de $H_i$, les équations \eqref{qac.17} prennent la forme
\begin{equation}
\mu_{N_\cI}(\hat{q})=-\zeta_{\cI} \quad \mbox{et}  \quad \check{\mu}_{\cI^c}(\check{\omega})= -u_\cI^2\zeta_{\cI}^\perp.
\label{qac.18}\end{equation}
Par définition d'une application moment hyperhamiltonienne, notons que la différentielle de la restriction de $\check{\mu}_{\cI^c}$ à $\bbS(V_{\cI})$ est de rang maximal, ce qui montre que l'équation
\begin{equation}
    \check{\mu}_{\cI^c}(\check{\omega})= -u_\cI^2\zeta_{\cI}^\perp
\label{qac.18b}\end{equation}
définit une $p$-sous-variété lisse sur $\overline{V}_{\cI}$ près de $\pa \overline{V}_{\cI}$ dont on dénotera le bord $\Sigma_i$.  En particulier, $\Sigma_i$ est une sous-variété de $S_i=\pa\overline{V}_{\cI}$.    

D'autre part, $\hat{q}$ peut être vu comme une coordonnée euclidienne sur l'intérieur des fibres de $\phi_i: H_i\to S_i$.  Par le Lemme~\ref{qac.11c}, $N_\cI$ satisfait à l'Hypothèse~\ref{qac.7}, donc par notre hypothèse d'induction sur la profondeur, pour $\zeta_{\cI}\in\bbR^3\otimes \gn_\cI^*$  générique, donc pour $\zeta$  générique, la fermeture $\widetilde{\mu_{N_\cI} ^{-1}(-\zeta_{\cI})}$ de $\mu_{N_\cI} ^{-1}(-\zeta_{\cI})$ dans $\widetilde{V}_{\cI^c}$ est une variété à coins fibrés comme décrite dans l'énoncé du théorème.  En tenant compte du fait que $u_{\cI}$ se relève sur $\widetilde{\bbH^d}$ en une fonction bordante totale près de $H_i$, cela montre que sur $\widetilde{\bbH^d}$, les équations \eqref{qac.18} définissent bien une $p$-sous-variété de $\widetilde{\bbH^d}$ près de $H_i$ et que $\phi_i$ induit bien par restriction un fibré
\begin{equation}
   \phi_{N,i}: H_i\cap \widetilde{\mu_N^{-1}(-\zeta)}\to \Sigma_i.
\label{fibre.2}\end{equation}

Le long d'une hypersurface bordante $H_i$ qui n'est pas minimale, on peut supposer par induction que $\widetilde{\mu_N^{-1}(-\zeta)}$ est une $p$-sous-variété près de $H_j$ pour $H_j<H_i$.  Il suffit donc de montrer que $\widetilde{\mu^{-1}_N(-\zeta)}$ est une $p$-sous-variété près de $\phi_i^{-1}(p)$ pour $p\in S_i\setminus \pa S_i$, en quel cas on peut appliquer essentiellement le même argument que dans le cas où $H_i$ est minimale pour obtenir le résultat.  Clairement, l'action de $N$ sur $\mu_N^{-1}(-\zeta)$ se prolonge en une action libre sur $\widetilde{\mu_N^{-1}(-\zeta)}$ compatible avec la structure de fibrés itérés.
\end{proof}

On en déduit le résultat suivant pour le quotient hyperkählérien $M_\zeta$.
\begin{corollaire}
Pour $\zeta\in \bbR^3\otimes \gn^*$  générique, la métrique hyperkählérienne du quotient hyperkählérien $M_{\zeta}$ est une métrique quasi-asymptotiquement conique exacte ayant un développement lisse à l'infini.
\label{qac.19}\end{corollaire}
\begin{proof}
Par \cite{CDR}, la métrique euclidienne $g_{\bbH^d}$ sur $\bbH^d$ peut être vue comme une métrique $\QAC$ exacte sur $\widetilde{\bbH^d}$.  Puisque la structure de fibrés itérés sur $\widetilde{\mu_N^{-1}(-\zeta)}$ est induite par celle de $\widetilde{\bbH^d}$, la restriction de la métrique $g_{\bbH^d}$ à $\widetilde{\mu_N^{-1}(-\zeta)}$ est automatiquement $\QAC$ exacte.  Puisque l'action de $N$ sur $\widetilde{\mu_N^{-1}(-\zeta)}$ est libre, compatible avec la structure de fibrés itérés et préserve la classe canonique d'équivalence $\QAC$ de fonctions bordantes totales, on en déduit que la métrique induite sur le quotient est bien $\QAC$ exacte.
\end{proof}

Soit $\cH^*(M_{\zeta})$ l'espace des formes harmoniques de carré intégrable par rapport à la métrique hyperkählérienne.  Le corollaire précédent nous permet de décrire ces espaces en termes des groupes de cohomologie de $M_{\zeta}$.

\begin{corollaire}
Si $\zeta\in\bbR^3\otimes \gn^*$ est  générique, alors
$$
    \cH^*(M_{\zeta})\cong \Im (H^*_c(M_{\zeta})\to H^*(M_{\zeta})),
$$
où $H^*(M_{\zeta})$ et $H_c^*(M_{\zeta})$ dénotent les groupes de cohomologie de de Rham pour les formes lisses et les formes lisses à support compact respectivement.
\label{qac.20}\end{corollaire}
\begin{proof}
Par \eqref{fibre.1}, on déduit que le fibré \eqref{fibre.2} est trivial et donné par la projection
$$
    \pr_1: \Sigma_i\times \widetilde{\mu^{-1}(-\zeta)}\to \Sigma_i
$$
sur le premier facteur.  En particulier, le fibré euclidien plat $E_{H_i}\to \Sigma_i$ correspondant aux formes harmoniques de carré intégrable dans les fibres de \eqref{fibre.2} est trivial.  Cela implique qu'il est trivialement pleinement adéquatement $\Sp(1)$-équivariant au sens de \cite[Definition~8.1]{DR}.
En s'appuyant sur \cite{KR1,KR2},  \cite[\S~7]{DR} et en utilisant le fait \cite[Theorem~6.7]{BD} que $H^k(M_{\zeta})=\{0\}$ pour $k>\frac{\dim_{\bbR}M_{\zeta}}2$, on peut donc procéder comme dans \cite[\S~8]{DR} pour déduire le résultat du Corollaire~\ref{qac.20}.
\end{proof}

Lorsque $\zeta=0$, remarquons que $\mu^{-1}_N(0)$ est un cône dans $\bbH^d$ et que son quotient $\mu^{-1}_N(0)/N$ est un cône hyperkählérien.  De plus, sur $\overline{\bbH^d}$, on a par homogénéité  de l'application moment hyperhamiltonienne que $\overline{\mu^{-1}_N(-\zeta)}\cap \pa \overline{\bbH^d}=\overline{\mu^{-1}_N(0)}\cap \pa \overline{\bbH^d}$.  Sur l'hypersurface maximale $H_{\ell}$ de $\widetilde{\bbH^d}$, cela correspond au fait que 
$$
\widetilde{\mu^{-1}_N(-\zeta)}\cap H_\ell =\widetilde{\mu^{-1}_N(0)}\cap H_\ell .
$$
En particulier, près de $\widetilde{\mu^{-1}_N(-\zeta)}\cap H_{\ell}$, la métrique induite de $\mu_N^{-1}(-\zeta)$ est asymptotique à celle de $\mu^{-1}_N(0)$, c'est-à-dire qu'en tant que métrique $\QAC$ exacte, la métrique induite de $\mu^{-1}_N(-\zeta)$ est de la forme
\begin{equation}
 d\rho^2+ \rho^2g_{\cS}  + r, \quad r\in x_{H_{\ell}}\CI(\widetilde{\mu^{-1}(-\zeta)}; S^2({}^{\QAC}T^*(\widetilde{\mu^{-1}_N(-\zeta)}))),
\label{cone.1}\end{equation}
où $\rho$ est la distance par rapport à l'origine dans $\bbH^d$ et $g_{\cS}$ est la métrique wedge de la base du cône de $\mu^{-1}_N(0)$.  Lorsqu'on passe au quotient, on a de même que la métrique hyperkählérienne sur $\mu^{-1}(-\zeta)/N$ est asymptotique à celle de $\mu^{-1}(0)/N$ près de $H_{\ell}/N$, donnant lieu au résultat suivant.
\begin{corollaire}
Pour $\zeta\in\bbR^3\otimes \gn^*$ générique, le cône tangent à l'infini de la métrique hyperkählérienne sur $M_{\zeta}$ est le cône hyperkählérien $M_0$.
\label{cone.2}\end{corollaire}

Remarquons que la métrique hyperkählérienne $\QAC$ du Corollaire~\ref{qac.19} sera asymptotiquement conique ($\AC$) si et seulement si $\widetilde{\mu^{-1}_N(-\zeta)/N}$ est une variété à bord, c'est-à-dire si et seulement si la base du cône tangent à l'infini $\mu^{-1}_N(0)/N$ est lisse.  Or, Bielawski et Dancer \cite[Theorem~4.1]{BD} donnent un critère pour que cette section soit lisse, ce qui nous permet de déduire le résultat suivant.
\begin{corollaire}
Si $\zeta\in\bbR^3\otimes\gn^*$ est  générique, la métrique hyperkählérienne $\QAC$ sur $M_{\zeta}$ sera asymptotiquement conique si et seulement si le sous-tore $N$ satisfaisant à l'Hypothèse~\ref{qac.7} est tel que tout sous-ensemble de $n$ vecteurs de $\{u_1,\ldots,u_d\}$ forme une base du réseau $\bbZ^n$.
\label{cone.3}\end{corollaire}
\begin{proof}
Comme on suppose déjà que le sous-tore $N$ satisfait à l'Hypothèse~\ref{qac.7}, les deux conditions de \cite[Theorem~4.1]{BD} correspondent à la condition dans l'énoncé du corollaire.
\end{proof}
\begin{remarque}
Lorsque les conditions du Corollaire~\ref{cone.3} sont vérifiées, l'action de $N$ sur $\mu^{-1}_N(0)\setminus \{0\}$ est libre.
\label{cone.4}\end{remarque}
\begin{exemple}\emph{(\cite[p.737]{BD})}
Si $n=d-1$ et $\{e_1,\ldots,e_n\}$ est la base canonique de $\bbR^n$, on peut prendre $u_i=e_i$ pour $i\le n$ et $u_{n+1}=-(e_1+\ldots+e_n)$, en quel cas $N$ correspond au cercle diagonal dans $T^d$.  Par \cite[Example~12.8.5]{Boyer-Galicki}, on sait alors que pour $\zeta\in\bbR^3\setminus\{0\}$,  la métrique $g_{\zeta}$ correspond à (un multiple de) la métrique de Calabi \cite[Théorème~5.3]{Calabi} sur $T^*\bbC\bbP^{d-1}$.
\label{con.5}\end{exemple}

En fait, le résultat suivant montre qu'outre la métrique euclidienne, les seuls exemples de métriques hyperkählériennes toriques asymptotiquement coniques simplement connexes de dimension $4n\ge 8$ sont précisément ceux de l'Exemple~\ref{con.5}.
\begin{corollaire}
Soit $(M,g)$ une variété hyperkählérienne torique simplement connexe de dimension $4n\ge 8$.  Alors la métrique $g$ est isométrique à la métrique euclidienne ou à un multiple de la métrique de Calabi.
\label{acht.1}\end{corollaire}
\begin{proof}
Une variété asymptotiquement conique étant de type topologique fini, on est dans le cadre de la classification de Bielawski \cite[Theorem~1]{Bielawski} et on peut supposer que la variété $(M,g)$ est un quotient hyperkählérien de la forme \eqref{qac.8}.  Comme $M$ est une variété lisse, on peut supposer par la Remarque~\ref{qac.10b} que l'Hypothèse~\ref{qac.7} est vérifiée et que $(M,g)$ est une variété obtenue dans le cadre du Corollaire~\ref{cone.3}.  

Si $n=d$ dans ce corollaire, alors $N=\{\Id\}$ et la métrique hyperkählérienne torique correspond à la métrique euclidienne.  Si $n<d$, alors par le Corollaire~\ref{cone.3}, $\{u_1,\ldots,u_n\}$ forme une base de $\bbZ^n$ et $u_{n+1}$ est forcément de la forme
$$
     u_{n+1}= \sum_{i=1}^n a_iu_i \quad \mbox{avec} \quad a_i\in\{-1,1\}.
$$  
De plus, on a forcément que $d=n+1$.  En effet, autrement, $u_{n+2}$ serait aussi de la forme
$$
          u_{n+2}= \sum_{i=1}^n b_iu_i \quad \mbox{avec} \quad b_i\in\{-1,1\}.
$$
Puisque $u_{n+1}$ et $u_{n+2}$ doivent être linéairement indépendants, il existerait $i_1,i_2\in\{1,\ldots,n\}$  tels que $b_{i_1}= a_{i_1}$ et $b_{i_2}=-a_{i_2}$.  Or, dans ce cas, $\{u_1,\ldots,u_{n+2}\}\setminus\{u_{i_1},u_{i_2}\}$ ne forme pas une base de $\bbZ^n$, car $u_{i_1}$ et $u_{i_2}$ ne sont pas contenus dans le réseau engendré par les vecteurs de cet ensemble.  Pour éviter une contradiction avec l'énoncé du Corollaire~\ref{cone.3}, il faut donc admettre que $d=n+1$.  

De la description de $u_1,\ldots, u_{n+1}$, on déduit que le sous-tore $N$ correspond au cercle
$$
    \{(e^{-ia_1t},\ldots,e^{-ia_nt},e^{it})\in T^d\; | \; t\in\bbR\} \subset T^d.
$$
Quitte à interchanger les rôles de $z_i$ et $w_i$ dans \eqref{action.1} lorsque $a_i=1$, on peut en fait supposer que $N$ correspond au cercle diagonal de $T^d$.  C'est le cadre de l'Exemple~\ref{con.5} et la métrique hyperkählérienne torique est donc un multiple de la métrique de Calabi.

\end{proof}

Cependant, lorsque $n=1$ dans le Corollaire~\ref{cone.3}, on obtient d'autres exemples de variétés hyperkählériennes toriques asymptotiquement coniques en prenant $d>n+1$, en quel cas $M_{\zeta}$ est difféomorphe à une résolution crépante de $\bbC^2/\bbZ_d$ et la métrique $g_{\zeta}$ est $\ALE$ avec cône tangent à l'infini  $\bbC^2/\bbZ_d$.

\section{Déformations de Taub-NUT d'ordre 1 des variétés hyperkählériennes toriques asymptotiquement coniques} \label{dtn.0}

Dans cette section, nous allons donner une description détaillée de la géométrie à l'infini des déformations de Taub-NUT d'ordre 1 de variétés hyperkählériennes toriques asymptotiquement coniques.  Pour ce faire, procédons dans un premier temps à une révision rapide de la notion de déformation de Taub-NUT.
\begin{definition}\emph{(\cite[Definition~2]{Bielawski})}. Soit $(M^{4n},g)$ une variété hyperkählérienne complète et connexe de type topologique fini munie d'une action hyperhamiltonienne effective de $G=\bbR^p\times \bbT^{n-p}$.  Alors une \textbf{déformation de Taub-NUT d'ordre $m$} de $(M^{4n},g)$ est le quotient hyperkählérien de $M\times \bbH^m$ par $\bbR^m$, où l'action de  $\bbR^m$ sur $M$ est spécifiée par une application linéaire injective $\sigma$ de $\bbR^m$ dans l'algèbre de Lie de $G$, alors que son action sur  $\bbH^m\cong (\bbR^m)^4$ est par translation dans le facteur réel.  On dénote la métrique hyperkählérienne correspondante par $g_{\TN_{\sigma}}$.   
\label{dtn.1}\end{definition}  

Soit $\mu: M\to \bbR^3\otimes(\bbR^m)^*$ un choix d'application moment hyperhamiltonienne de l'action de $\bbR^m$ sur $M$.  Pour $q=(q_0,q_1,q_2,q_3)\in (\bbR^m)^4\cong \bbH^m$, l'action de $\bbR^m$ sur $\bbH^m$ est induite par l'application moment  $ \lambda=(\lambda_1,\lambda_2,\lambda_3)$ définie par
\begin{equation}
   \lambda_\ell(q)= -q_{\ell}.
\label{dtn.2}\end{equation}
La déformation de Taub-NUT de $M^{4n}$ correspond alors au quotient 
$$
     (\mu+\lambda)^{-1}(0)/\bbR^m.
$$
Clairement, chaque orbite de l'action de $\bbR^m$ sur $M\times \bbH^m$ contient un unique point $p$ pour lequel $q_0=0$.  Si $M\times (\Im \bbH)^m$ dénote la sous-variété de $M\times \bbH^m$ sur laquelle $q_0=0$, cela signifie que l'inclusion
$$
      X_M:=(\mu+\lambda)^{-1}(0)\cap (M\times (\Im\bbH)^m ) \subset  (\mu+\lambda)^{-1}(0)
$$
induit un difféomorphisme 
\begin{equation}
  X_M\cong (\mu+\lambda)^{-1}(0)/\bbR^m.
\label{dtn.3}\end{equation}
Comme $X_M$ correspond au graphe de l'application $\mu$, on a donc un difféomorphisme naturel entre $M$ et sa déformation de Taub-NUT.  Soient $\xi_1,\ldots,\xi_m\in\CI(M;TM)$ les champs vectoriels sur $M$ correspondant à l'action infinitésimale des éléments $e_1,\ldots,e_m$ de la base canonique de $\bbR^m$.  
Soit $\cU$ un ouvert de $M$ où l'action de $\bbR^m$ est localement libre.  Alors les champs vectoriels $\xi_1,\ldots, \xi_m$ engendrent un sous-fibré vectoriel $E$ de $T\cU$.  Comme les champs vectoriels $\xi_i$ sont tangents aux hypersurfaces de niveaux de $\mu$, $E$ se relève naturellement en un sous fibré du fibré tangent de  $\cV:=X_M\cap (\cU\times(\Im\bbH)^m)$ qu'on dénotera aussi $E$. Soit $g_{(\Im\bbH)^m}$ la métrique canonique sur $(\Im\bbH)^m$.  Dénotons par $E^{\perp}$ le complément orthogonal de $E$ dans $T\cV$ par rapport à la métrique induite par $g+g_{(\Im\bbH)^m}$, de sorte que
\begin{equation}
 T\cV= E^{\perp}\oplus E.
\label{dtn.4}\end{equation}
Soient $g_E$ et $g_{E^{\perp}}$ les restrictions de la métrique $g+g_{(\Im\bbH)^m}$ à $E$ et $E^{\perp}$.  Pour décrire $g_{E}$, nous ferons l'hypothèse suivante, qui sera satisfaite pour les cas sur lesquels on se penchera dans cet article.
\begin{hypothese}
Il existe une base $\{\widetilde{e}_1,\ldots,\widetilde{e}_m\}$ de $\bbR^m$ telle que les champs vectoriels $\widetilde{\xi}_1,\ldots,\widetilde{\xi}_m\in\CI(M;TM)$ associés aux actions infinitésimales de $\widetilde{e}_1,\ldots,\widetilde{e}_m$ sont mutuellement orthogonaux par rapport à la métrique $(g+g_{(\Im\bbH)^m})|_\cV$ lorsque relevés à $\cV$.
\label{dtn.5}\end{hypothese}
Sous cette hypothèse, soient $\theta_1,\ldots,\theta_m$ les $1$-formes duales à $\widetilde{\xi_1},\ldots,\widetilde{\xi}_m$ par rapport à la métrique $(g+g_{(\Im\bbH)^m})|_\cV$.  Posons 
$$
     V_i:= \frac{1}{(g+g_{(\Im\bbH)^m})|_\cV(\widetilde{\xi}_i,\widetilde{\xi}_i)}= \frac{1}{g(\widetilde{\xi}_i,\widetilde{\xi}_i)} \quad \mbox{et} \quad \eta_i=V_i\theta_i,
$$
de sorte que $\eta_i(\widetilde{\xi}_j)=\delta_{ij}$.  En termes de cette notation, on a alors que 
\begin{equation}
 g_E= \sum_{i=1}^m \frac{\eta_i^2}{V_i}.
\label{dtn.6}\end{equation}
\begin{lemme}
Sous l'Hypothèse~\ref{dtn.5} et en termes de l'identification \eqref{dtn.3} et de la décomposition \eqref{dtn.4}, on a que
$$
    g_{\TN_\sigma}|_{\cV}= g_{E^{\perp}}+ \sum_{i=1}^m \frac{h_{ii}\eta_i^2}{h_{ii}V_i+1}+ \sum_{i\ne j}\frac{h_{ij}\eta_i\otimes \eta_j}{(h_{ii}V_i+1)(h_{jj}V_j+1)},
$$
où $h_{ij}:=\widetilde{e}_i\cdot \widetilde{e}_j$ est le produit scalaire de $\widetilde{e}_i$ et $\widetilde{e}_j$ par rapport à la métrique canonique de $\bbR^m$.
\label{dtn.7}\end{lemme}
\begin{proof}
Comme $E^{\perp}$ est perpendiculaire aux orbites de l'action de $\bbR^m$ sur $\cV$ (et sur $\cV\times \bbR^m\subset M\times \bbH^m$), on a clairement que \eqref{dtn.4} est aussi une décomposition orthogonale par rapport à la métrique $g_{\TN_\sigma}$ et que 
$$
g_{\TN_\sigma}|_{E^{\perp}}= (g+g_{(\Im\bbH)^m})|_{E^{\perp}}= g_{E^{\perp}}.  
$$
Il reste donc à déterminer $g_{\TN_{\sigma}}|_E$.  Par l'Hypothèse~\ref{dtn.5}, cela correspond à calculer $g_{\TN_\sigma}(\widetilde{\xi}_i,\widetilde{\xi}_j)$ pour chaque $i,j\in\{1,\ldots,m\}$.  Pour $p\in \cV$ fixé, cela revient à projeter 
$$
(\widetilde{\xi}_i(p),0)\in E_p\oplus \bbR^m \subset T_{(p,0)}(\cV\times \bbR^m)= T_{(p,0)}(M\times \bbH^m)
$$
sur le sous-espace  de $T_{(p,0)}(M\times \bbH^m)$ orthogonal au plan tangent de  l'orbite de l'action de $\bbR^m$ sur $M\times \bbH^m$ passant par $(p,0)$.  Or, la restriction de $g+g_{\bbH^m}$ à $E_p\oplus \bbR^m\subset T_{(p,0)}(M\times \bbH^m)$ est donnée par
$$
(g+g_{\bbH^m})|_{E_p\oplus\bbR^m}\lrp{ (\sum_i a_i\widetilde{\xi}_i(p), s), (\sum_i b_i\widetilde{\xi}_i(p),t) }= \sum_i \frac{a_ib_i}{V_i(p)}+ s\cdot t,  \quad s,t\in\bbR^m,
$$
alors que le plan tangent de l'orbite passant par $(p,0)$ est engendré par $\widetilde{\xi}_1(p)+\widetilde{e}_1,\ldots, \widetilde{\xi}_m(p)+\widetilde{e}_m$.  Un calcul élémentaire montre alors que la projection de $(\widetilde{\xi}_i(p),0)$ sur le sous-espace orthogonal à ce plan tangent de l'orbite est
$$
        \frac{h_{ii}V_i(p)\widetilde{\xi}_i(p) -e_i}{1+h_{ii}V_i(p)}.
$$ 
On en déduit que
$$
\begin{aligned}
   g_{\TN_\sigma}(\widetilde{\xi}_i(p),\widetilde{\xi}_j(p))&= (g+g_{\bbH^m})\lrp{\frac{h_{ii}V_i(p)\widetilde{\xi}_i(p) -\widetilde{e}_i}{1+h_{ii}V_i(p)},\frac{h_{jj}V_j(p)\widetilde{\xi}_j(p) -\widetilde{e}_j}{1+h_{jj}V_j(p)}} \\
   &= \left\{\begin{array}{ll}\frac{h_{ii}}{1+h_{ii}V_i(p)}, & i=j, \\ \frac{h_{ij}}{(1+h_{ii}V_i(p))(1+h_{jj}V_j(p))}, & i\ne j, \end{array} \right.  
 \end{aligned}  
$$
d'où le résultat.
\end{proof}

Dans les cas qui nous occuperont, les fonctions $V_i$ tendent vers zéro à l'infini dans $\cV$, ce qui montre que  pour $t\in[a,b]$ et $p\in\cV$, la longueur du segment $\gamma(t)=\phi_{\xi_i,t}(p)$  d'une orbite engendrée par $\xi_i$  
tend vers $b-a$ lorsque $p$ tend vers l'infini.  Comprendre la géométrie à l'infini de $g_{\TN_\sigma}$ revient alors essentiellement à comprendre la géométrie à l'infini de $g_{E^{\perp}}$.  

Dans cette section, nous allons nous concentrer sur des déformations de Taub-NUT d'ordre $1$, donc le cas $m=1$.  l'Hypothèse~\ref{dtn.5} est alors triviale, puisqu'il suffit de prendre $\widetilde{e}_1=e_1$. En ce qui concerne la variété hyperkählérienne $(M,g)$, nous allons nous restreindre au cas où celle-ci est une variété hyperkählérienne torique asymptotiquement conique 
$$
      M_{\zeta}=\mu_N^{-1}(-\zeta)/N
$$
donnée par le Corollaire~\ref{cone.3}.  Dénotons par $g_{\zeta}$ la métrique hyperkählérienne asymptotiquement conique associée. 
\begin{exemple}
On peut prendre $N=\{\Id\}$ et $\zeta=0$, en quel cas la métrique $g_{\zeta}$ correspond à la métrique canonique $g_{\bbH^d}$ sur $\bbH^d$.  C'est un exemple trivial du Corollaire~\ref{cone.3}, mais important malgré tout, puisqu'il admet plusieurs déformations de Taub-NUT d'ordre $1$ non triviales.  
\label{dtn.7a}\end{exemple}

Pour spécifier une déformation de Taub-NUT d'ordre $1$ de $M_{\zeta}$, il nous faut aussi spécifier une application linéaire  $\sigma: \bbR\hookrightarrow \gt^d$ telle que le sous-espace $\sigma(\bbR)$ n'est pas inclus dans l'algèbre de Lie $\gn$ de $N$.  En effet, par passage au quotient, celle-ci spécifie une inclusion linéaire dans l'algèbre de Lie du groupe $G:= T^d/N\cong T^n$ agissant de manière hyperhamiltonienne sur $M_{\zeta}$. L'inclusion $\sigma$ induit une action hyperhamiltonienne de $\bbR$ sur $\bbH^d$ avec application moment hyperhamiltonienne
$$
   \mu_{\sigma}:=(\Id\otimes \sigma)^*\circ \mu.
$$ 
On obtient par restriction une action de $\bbR$ sur $\mu_N^{-1}(-\zeta)$ qui descend en une action hyperhamiltonienne de $\bbR$ sur le quotient $M_{\zeta}$.  Notons que $\mu_{\sigma}$ est constante le long des orbites de l'action de $N$, donc $\mu_{\sigma}|_{\mu^{-1}(-\zeta)}$ induit sur $M_{\zeta}$ une application moment hyperhamiltonienne pour l'action de $\bbR$ sur $M_{\zeta}$ qu'on dénotera $\mu_\sigma|_{M_\zeta}$.  Les applications hyperhamiltoniennes des actions de $\bbR$ sur $\bbH^d\times \bbH$ et $M_{\zeta}\times \bbH$ sont donc données par
$$
      \mu_{\sigma}+\lambda \quad \mbox{et} \quad \mu_{\sigma}|_{M_{\zeta}}+\lambda,
$$
où $\lambda:\bbH\to \bbR^3$ est l'application moment hyperhamiltonienne spécifiée par \eqref{dtn.2} dans le cas où $m=1$.  Dénotons par $g_{\zeta,\TN_\sigma}$ la métrique hyperkählérienne associée à la déformation de Taub-NUT de $g_\zeta$ spécifiée par l'inclusion $\sigma: \bbR\hookrightarrow \gt^d$.

Pour décrire la géométrie à l'infini de $g_{\zeta,\TN_\sigma}$, nous allons introduire une compactification de $\bbH^d\times \Im\bbH$ par rapport à laquelle la fermeture de la   sous-variété
$$
    X_{\zeta}:= (\mu_\sigma+\lambda)^{-1}(0)\cap\lrp{\mu_N^{-1}(-\zeta)\times \Im\bbH}
$$
 se comporte bien.  

  Considérons dans un premier temps la compactification radiale $\overline{\bbH^d\times \Im \bbH}$ de $\bbH^d\times \Im\bbH$.  
Soit $\cI\subset \{1,\ldots,d\}$ le plus grand sous-ensemble tel que $\sigma(\bbR)\subset \gt_{\cI}$, où $\gt_{\cI}$ est l'algèbre de Lie du tore $T_{\cI}$ défini en \eqref{qac.4a}.  Avec la convention que $V_{\emptyset}=\{0\}$, on a alors que $V_{\cI}$ est le sous-espace des points fixes de $\bbH^d$ par rapport à l'action de $\bbR$.  En particulier, si $q_{\cI}$ et $q_{\cI^c}$ sont les variables quaternioniques de $V_{\cI}$ et $V_{\cI^c}$, remarquons qu'en termes de la décomposition $\bbH^d=V_{\cI}\oplus V_{\cI^c}$, l'application $\mu_{\sigma}$ ne dépend que de $q_{\cI^c}$. Soit $\overline{V_{\cI}\times \Im\bbH}$ la fermeture de $V_{\cI}\times \Im\bbH$ dans $\overline{\bbH^d\times \Im\bbH}$.   Les équations définissant $X_{\zeta}$ sont alors données par 
\begin{equation}
  \mu_\sigma(m)=q \quad \mbox{et}  \quad \mu_N(m)=-\zeta  \quad \mbox{pour} \quad (m,q)\in \bbH^d\times \Im\bbH.
\label{dtn.9}\end{equation}
Soit $\rho$ la distance par rapport à l'origine dans $\bbH^d\times \Im\bbH$.  Soient $\rho$ et $\omega=(\omega_{\bbH^d},\omega_{\Im \bbH})=(\frac{m}{\rho},\frac{q}{\rho})$ les coordonnées sphériques correspondantes.  Alors $u:=\rho^{-1}$ est une fonction bordante pour $\pa\overline{\bbH^d\times \Im\bbH}$ dans $\bbH^d\times \Im\bbH$.  Dans les coordonnées $(u,\omega)$ près de $\pa\overline{\bbH^d\times \Im\bbH}$, les équations \eqref{dtn.9} prennent la forme
\begin{equation}
  \mu_{\sigma}(\omega_{\bbH^d})= u\omega_{\Im\bbH} \quad \mbox{et} \quad  \mu_N(\omega_{\bbH^d})=-u^2\zeta,
\label{dtn.10}\end{equation}
de sorte que si $\bX_{\zeta}$ dénote la fermeture de $X_{\zeta}$ dans $\overline{\bbH^d\times \Im\bbH}$, alors $\pa\bX_{\zeta}=\bX_{\zeta}\cap\pa\overline{\bbH^d\times \Im\bbH}$ est donné par la restriction de ces équations en $u=0$, c'est-à-dire par les équations
\begin{equation}
    \mu_{\sigma}(\omega_{\bbH^d})=0 \quad \mbox{et}\quad \mu_N(\omega_{\bbH^d})=0.
\label{dtn.10b}\end{equation}

Puisque l'action de $N$ sur $\mu_N^{-1}(0)\setminus\{0\}$ est libre, remarquons que la différentielle de $\mu_N$ est surjective lorsqu'évaluée en un point de $\mu_N^{-1}(0)\setminus\{0\}$.  De même, puisque l'action de $\bbR$ restreinte à $\bbH^d\setminus V_{\cI}$ est localement libre, la différentielle de $\mu_\sigma$ est surjective lorsqu'évaluée en  $\bbH^d\setminus V_{\cI}$.
Comme $\sigma(\bbR)$ n'est pas contenu dans $\gn$, cela montre que le sous-ensemble $\bX_\zeta$  est une $p$-sous-variété, sauf le long de $\pa\overline{V_{\cI}\times \Im\bbH}$ où les équations \eqref{dtn.10} dégénèrent.  En procédant comme dans \cite{CR,CR2023}, ces dégénérescences peuvent être résolues grâce à une suite d'éclatements.  Soit $\overline{V_{\cI}\times \{0\}}$ la fermeture de $V_{\cI}\times \{0\}$ dans $\overline{\bbH^d\times \Im\bbH}$ et considérons l'espace éclaté
\begin{equation}
   W:= [\overline{\bbH^d\times \Im\bbH}; \pa\overline{V_{\cI}\times \{0\}}]
\label{dtn.11}\end{equation}
avec application de contraction $\beta_{W}: W\to \overline{\bbH^d\times \Im\bbH}$,
où $\pa \overline{V_{\cI}\times \{0\}}= \overline{V_{\cI}\times \{0\}}\cap \pa \overline{\bbH^d\times \Im\bbH}$.  Si $\cI=\emptyset$, $V_{\cI}=\{0\}$ et $\pa\overline{V_{\cI}\times \{0\}}=\emptyset$, de sorte que $W=\overline{\bbH^d\times\Im\bbH}$.  Autrement, $W$  est une variété à coins ayant deux hypersurfaces bordantes  
\begin{equation}
   H_1:= \beta_{W}^{-1}(\pa\overline{V_{\cI}\times \{0\}}) \quad \mbox{et} \quad H_4:= \overline{\beta_{W}^{-1}(\pa \overline{\bbH^d\times \Im\bbH}\setminus \pa\overline{V_{\cI}\times \{0\}} )}.
\label{dtn.12}\end{equation}
Lorsque $V_{\cI}=\{0\}$, on utilisera la convention $H_1=\emptyset$ et $H_4=\pa\overline{\bbH^d\times \Im\bbH}$ pour que \eqref{dtn.12} reste valide.    

Si $x_{H_i}$ dénote un choix de fonction bordante pour $H_i$, soit maintenant $\tW$ la variété à coins qui, en tant qu'espace topologique, coïncide avec $W$, mais a pour espace de fonctions lisses les fonctions lisses sur $W\setminus H_4$ ayant un développement lisse en puissances entières de $x_{H_4}^{\frac12}$ (plutôt qu'en puissances entières de $x_{H_4}$).  Dénotons par $\tH_i$ l'hypersurface bordante $H_i$ de $W$ vue comme une hypersurface bordante de $\widetilde{W}$.  Soient $\widetilde{V_{\cI}\times \Im \bbH}$ et $\widetilde{\{0\}\times \Im \bbH}$ les fermetures de $V_{\cI}\times \Im \bbH$ et $\{0\}\times \Im \bbH$ dans $\widetilde{W}$ et considérons l'espace éclaté
\begin{equation}
   \hW:= [\tW; (\widetilde{\{0\}\times \Im \bbH})\cap \tH_4, (\widetilde{V_{\cI}\times \Im \bbH})\cap \tH_4]
\label{dtn.13}\end{equation}  
avec application de contraction $\beta_{\hW}: \hW\to \tW$.
\begin{remarque}
Lorsque $V_{\cI}=\{0\}$, les deux $p$-sous-variétés éclatées coïncident de sorte que le deuxième éclatement est trivial.  D'autre part, si $V_{\cI}\ne \{0\}$, mais $N=\{\Id\}$ et $\zeta=0$, l'éclatement de $(\widetilde{\{0\}\times \Im \bbH})\cap \tH_4$ n'est pas nécessaire pour décrire le comportement à l'infini de $X_{\zeta}$.
\label{dtn.14}\end{remarque}
\begin{remarque}
Comme les $p$-sous-variétés $(\widetilde{\{0\}\times \Im \bbH})\cap \tH_4$ et $(\widetilde{V_{\cI}\times \Im \bbH})\cap \tH_4$ sont emboîtées, les éclatements correspondant commutent par \cite[Lemma~2.1]{hmm}.  
\label{dtn.14c}\end{remarque}

Soient $\hH_1$ et $\hH_4$ les hypersurfaces bordantes correspondant aux relèvements à $\hW$ des hypersurfaces bordantes $\tH_1$ et $\tH_4$ de $\tW$, et soient 
$$
    \hH_2:= \overline{\beta_{\hW}^{-1}((\widetilde{\{0\}\times \Im \bbH})\cap \tH_4)} \quad \mbox{et} \quad
    \hH_3:= \overline{\beta_{\hW}^{-1}( ((\widetilde{V_\cI\times \Im\bbH})\setminus (\widetilde{\{0\}\times \Im \bbH}))\cap \tH_4)}
$$
les hypersurfaces bordantes créées par les deux éclatements définissant $\hW$.  Lorsque $V_{\cI}=\{0\}$, remarquons que $\hH_3=\hH_1=\emptyset$, c'est-à-dire que $\hW$ ne possède que deux hypersurfaces bordantes.  

Chaque application de contraction associée à un  éclatement entrant dans la définition de  $\hW$ à partir de $\overline{\bbH^d\times \Im\bbH}$ induit un fibré $\hphi_i: \hH_i\to \hS_i$ sur l'hypersurface bordante correspondante.  Les bases de ces fibrés sont données par
$$
\hS_1=\pa\overline{V_{\cI}\times\{0\}}, \quad \hS_2= \pa \overline{\{0\}\times \Im\bbH} \quad \mbox{et} \quad \hS_3=[\pa \overline{V_\cI\times \Im\bbH}; \pa\overline{V_{\cI}\times \{0\}}, \pa \overline{\{0\}\times \Im\bbH}].
$$
Dans ces trois cas, on a en fait une décomposition canonique $\hH_i= \widehat{F}_i\times \hS_i$ avec $\hphi_i$ correspondant à la projection sur le deuxième facteur.  Pour $i\in\{2,3\}$, les fibres de ces fibrés sont données explicitement par
\begin{equation}
     \widehat{F}_2= [\overline{\bbH^d}; \pa\overline{V_{\cI}}] \quad\mbox{et} \quad \widehat{F}_3=\overline{V_{\cI^c}}.
\label{fib.1}\end{equation}
Pour $i=1$, la description de $\widehat{F}_1$ est légèrement plus compliquée.  Soit $\widetilde{V_{\cI^c}\times\Im\bbH}$ la variété à bord topologiquement identifiée avec $\overline{V_{\cI^c}\times \Im\bbH}$, mais ayant pour fonctions lisses les fonctions lisses sur $V_{\cI^c}\times \Im\bbH$ admettant une développement lisse en $\pa\overline{V_{\cI^c}\times \Im\bbH}$ en puissances entières de $x^{\frac12}$ pour $x$ une fonction bordante de $\pa\overline{V_{\cI^c}\times \Im\bbH}$ dans $\overline{V_{\cI^c}\times \Im\bbH}$. Si $\widetilde{\{0\}\times \Im\bbH}$ dénote la fermeture de $\{0\}\times \Im\bbH$ dans $\widetilde{V_{\cI^c}\times\Im\bbH}$, alors les fibres de $\hphi_1$ sont données par
\begin{equation}
 \widehat{F}_1= [\widetilde{V_{\cI^c}\times\Im\bbH}; \pa\widetilde{\{0\}\times \Im\bbH}]. 
\label{fib.2}\end{equation}

D'autre part, sur $\hH_4$, qui correspond au relèvement de $\pa(\overline{\bbH^d\times \Im\bbH})$, on peut poser $\hS_4=\hH_4$ et prendre pour fibré naturel $\hphi_4:\hH_4\to \hS_4$ l'application identité.  Ces fibrés $\hphi= (\hphi_1,\hphi_2,\hphi_3,\hphi_4)$ induisent une structure de fibrés itérés sur $\hW$ avec ordre partiel sur $\cM_1(\hW)$ spécifié par
$$
    \hH_1<\hH_3<\hH_4, \quad \hH_1<\hH_4, \quad \hH_2<\hH_3<\hH_4, \quad \mbox{et} \quad \hH_2<\hH_4.
$$

En prenant en compte le changement de structure lisse lors du passage de $W$ à $\tW$, remarquons aussi que la fonction $\rho$ de distance par rapport à l'origine sur $\bbH^d\times \Im\bbH$ induit sur $\hW$ une distance $\gp$-tordue avec fonction de pondération donnée par
\begin{equation}
  \gp(\hH_1)=0, \quad \gp(\hH_2)= \gp(\hH_3)=\gp(\hH_4)=\frac12.
\label{dtn.14b}\end{equation}
La structure de fibrés itérés et ce choix de distance $\gp$-tordue induit sur $\hW$ une algèbre de Lie de champs vectoriels $\CI(\hW; {}^{\cT}T\hW)$ $\gp$-tordus $\QAC$.  
\begin{lemme}
La métrique euclidienne $g_{\bbH^d}+g_{\Im\bbH}$ correspond à une métrique $\gp$-tordue $\QAC$ exacte sur $\hW$.
\label{dtn.15}\end{lemme}
\begin{proof}
Comme dans \cite[\S~3.2]{CDR}, la métrique $g_{\bbH^d}+g_{\Im\bbH}$ correspond à une métrique $\QAC$ exacte sur $W$.  Sur $\tW$, on vérifie que ça devient une métrique $\mathfrak{q}$-tordue $\QAC$ exacte pour la fonction de pondération définie par $\mathfrak{q}(\tH_1)=0$ et $\mathfrak{q}(\tH_4)=\frac12$.  Enfin, on montre en utilisant des coordonnées locales qu'elle se relève bien en une métrique $\gp$-tordue $\QAC$ exacte sur $\hW$. 
\end{proof}

Soit $\hX_{\zeta}$ la fermeture de  $X_{\zeta}$ dans $\hW$.  
\begin{proposition}
Le sous-ensemble $\hX_{\zeta}$ est une  $p$-sous-variété de $\hW$.  De plus, la structure de fibrés itérés sur $\hW$ induit par restriction une structure de fibrés itérés sur $\hX_\zeta$.  En particulier, si $\gp$ dénote aussi la fonction de pondération induite par $\gp$ sur $\hX_\zeta$, alors $(g_{\bbH^d}+g_{\Im\bbH})$ induit par restriction une métrique $\gp$-tordue $\QAC$ exacte sur $\hX_{\zeta}$.  
\label{dtn.16}\end{proposition}
\begin{proof}
Soient $q_{\cI}$ $q_{\cI^c}$ et $q_{\Im\bbH}$ les variables quaternioniques sur $V_{\cI}$, $V_{\cI^c}$ et $\Im\bbH$ respectivement.  Soient $\rho$ et $\omega=(\omega_{\cI},\omega_{\cI^c},\omega_{\Im\bbH})\in \pa\overline{\bbH^d\times \Im\bbH}$ des coordonnées sphériques sur $\bbH^d\times \Im\bbH$ en termes de la décomposition $\bbH^d\times \Im\bbH= V_{\cI}\times V_{\cI^c}\times\Im\bbH$, c'est-à-dire que
$$
    \omega_{\cI}=\frac{q_{\cI}}{\rho}, \quad \omega_{\cI^c}=\frac{q_{\cI^c}}{\rho} \quad \mbox{et} \quad
    \omega_{\Im \bbH}= \frac{q_{\Im\bbH}}{\rho}.
$$
Comme on l'a vu précédemment, sur $\overline{\bbH^d\times \Im\bbH}$, les équations \eqref{dtn.10} dégénèrent le long de $\overline{V_{\cI}\times \Im \bbH}$.  Voyons si les éclatements introduits pour définir $\hW$ résolvent ces singularités.

Si $u=\rho^{-1}$, alors en termes des coordonnées $(u,\omega)$ près de $\pa\overline{\bbH^d\times \Im\bbH}$ sur $\overline{\bbH^d\times \Im\bbH}$,  
l'éclatement de $\pa\overline{ V_{\cI}\times \{0\}}$  correspond à remplacer $\omega_{\cI^c}$ et $\omega_{\Im\bbH}$ par les coordonnées $q_{\cI^c}$ et $q_{\Im\bbH}$, en quel cas la première équation de \eqref{dtn.10} devient
\begin{equation}
  \mu_\sigma(q_{\cI^c})= q_{\Im\bbH}.
\label{dtn.17}\end{equation}
Cela correspond au graphe de $\mu_\sigma|_{V_{\cI^c}}$ à l'intérieur de chaque fibre de $\hphi_{1}:\hH_1\to \hS_1$.  D'autre part, par \eqref{am.1} et \eqref{am.2}, on voit que
$$
     \mu_N(q_{\cI},q_{\cI^c})= \mu_{N,\cI}(q_{\cI})+ \mu_{N,\cI^c}(q_{\cI^c}),
$$
où $\mu_{N,\cI}:= \mu_N|_{V_{\cI}}$ est une application moment hyperhamiltonienne pour l'action de $N$ restreinte à $V_{\cI}$.  En termes de cette décomposition, la deuxième équation de \eqref{dtn.10} devient donc
\begin{equation}
    \mu_{N,\cI}(\omega_{\cI})+ u^2\mu_{N,\cI^c}(q_{\cI^c})=-u^2\zeta.
\label{dtn.17b}\end{equation}
Or, sur $\overline{V_{\cI^c}\times\Im \bbH}$, l'équation \eqref{dtn.17} définit une $p$-sous-variété, sauf sur $\pa\overline{\{0\}\times \Im\bbH}\subset \overline{V_{\cI^c}\times \Im\bbH}$.  En effet, près de $\pa \overline{V_{\cI^c}\times \Im\bbH}$ dans les coordonnées sphériques $(r,\omega_{1,\cI^c},\omega_{1,\Im\bbH})$ sur $V_{\cI^c}\times \Im\bbH$, l'équation \eqref{dtn.17} prend la forme
\begin{equation}
  \mu_{\sigma}(\omega_{1,\cI^c})=\frac{\omega_{1,\Im\bbH}}{r}.
\label{dtn.18}\end{equation}   
D'autre part, restreinte à $\hH_1$, l'équation \eqref{dtn.17b} devient
\begin{equation}
 \mu_{N,\cI}(\omega_{\cI})=0.
\label{dtn.18b}\end{equation}
Cela correspond à une équation dans la base du fibré $\hphi_1: \hH_1\to \hS_1$.  Comme $N$ agit librement sur $\mu_N^{-1}(0)\setminus \{0\}$, la restriction de cette action à $(\mu_N^{-1}(0)\cap V_{\cI})\setminus \{0\}=(\mu_{N,\cI}^{-1}(0)\cap V_{\cI})\setminus \{0\}$ est aussi libre, ce qui montre que la différentielle de $\mu_{N,\cI}$ y est en tout point surjective et que l'équation \eqref{dtn.18b} définit bien une sous-variété de $\hS_1$.  D'autre part, l'équation \eqref{dtn.18} induit une équation sur l'intérieur des fibres $\widehat{F}_1$ du fibré $\hphi_1$ décrites en \eqref{fib.2}.
  En termes de cette description, la différentielle de $\mu_{\sigma}$ ne s'annule qu'en $\{0\}\times \Im\bbH$, ce qui montre que le long de $\pa ( \widetilde{V_{\cI^c}\times \Im\bbH})$ dans $\widetilde{V_{\cI^c}\times \Im\bbH}$, l'équation \eqref{dtn.18} ne dégénère qu'en $\pa\widetilde{\{0\}\times \Im\bbH}$.  Dans la description \eqref{fib.2} de $\widehat{F}_1$, l'éclatement de $\pa\widetilde{\{0\}\times \Im\bbH}$ correspond à introduire les coordonnées 
$$
      \xi:= r^{-\frac12}, \quad Q_{\cI^c}:= r^{\frac12}\omega_{1,\cI^c}, \quad \omega_{1,\Im\bbH},
$$ 
en quel cas \eqref{dtn.18} et \eqref{dtn.17b} prennent la forme 
\begin{equation}
   \mu_{\sigma}(Q_{\cI^c})= \omega_{1,\Im\bbH}\quad \mbox{et} \quad \mu_{N,\cI}(\omega_{\cI})+ \frac{u^2}{\xi^2}\mu_{N,\cI^c}(Q_{\cI^c})=-u^2\zeta,
\label{dtn.19}\end{equation}
où $\frac{u}{\xi^2}$ est une fonction bordante pour $\hH_1$.    
En particulier, ces équations définissent une $p$-sous-variété près de l'intérieur de $\hH_1\cap \hH_3$, car l'action de $\bbR$ est localement libre sur $V_{\cI^c}\setminus \{0\}$ et $\omega_{1,\Im\bbH}\ne 0$ près de $\hH_1\cap \hH_3$. En remplaçant $Q_{\cI^c}$ par des coordonnées sphériques sur $V_{\cI^c}$, on vérifie aussi que $\hX_\zeta$ est bien une $p$-sous-variété dans un voisinage de $\hH_1\cap\hH_3$, c'est-à-dire prés de $\hH_4$. En fait, les équations \eqref{dtn.19} montrent que $\hX_{\zeta}$ est une $p$-sous-variété près de $\hH_3\setminus (\hH_2\cap\hH_3)$.

Toutefois, l'équation \eqref{dtn.18b} dans la base $\hS_3$ dégénère en $\omega_{\cI}=0$ si on omet l'éclatement engendrant $\hH_2$.  En fait, sur l'intérieur de $\hH_2$, l'éclatement de $(\widetilde{\{0\}\times \Im\bbH})\cap \tH_4$ correspond à introduire les coordonnées 
\begin{equation}
 Q_{\cI}= \frac{\omega_{\cI}}{u^{\frac12}}, \quad Q_{2,\cI^c}= \frac{\omega_{\cI^c}}{u^{\frac12}}, \quad \omega_{\Im\bbH}, \quad u^{\frac12},
\label{dtn.20}\end{equation}
en quel cas les équations \eqref{dtn.10}  prennent la forme
\begin{equation}
 \mu_{\sigma}(Q_{2,\cI^c})=\omega_{\Im\bbH} \quad \mbox{et} \quad \mu_N(Q_{\cI}, Q_{2,\cI^c})=-u\zeta.
\label{dtn.21}\end{equation}
Comme $\omega_{\Im\bbH}\in \bbS(\Im\bbH)\cong \pa\overline{\{0\}\times \bbH}=\hS_2$ dans ces équations, on a en particulier que $\omega_{\Im\bbH}\ne 0$.  Comme $\mu_{\sigma}(0)=0$, on en déduit que $Q_{2,\cI^c}\ne 0$ pour les solutions de \eqref{dtn.21}.  Or, l'action de $N$ sur $\mu^{-1}_N(0)\setminus\{0\}$ étant libre, la différentielle de $\mu_N$ est surjective lorsqu'évaluée en un point de $\mu_N^{-1}(0)\setminus \{0\}$.  De même, l'action de $\bbR$ étant localement libre sur $\bbH^d\setminus V_\cI$, la différentielle de $\mu_\sigma$ est surjective lorsqu'évaluée en un point où $Q_{2,\cI^c}\ne 0$.  Ajouté au fait que $\sigma(\bbR)$ n'est pas contenu dans $\gn$,  cela montre que les équations \eqref{dtn.21} définissent bien une sous-variété sur l'intérieur de $\hH_2$ et sur l'intérieur de chaque fibre de $\hphi_2$.  En utilisant les coordonnées $Q_{\cI},Q_{2,\cI^c}$, l'intérieur de ces fibres s'identifie avec $\bbH^d$, et les fibres elles-mêmes avec $[\overline{\bbH^d};\pa\overline{V}_{\cI}]$ par \eqref{fib.1}. 

Pour $\omega_{\Im\bbH}\in \hS_2$ fixé, les équations \eqref{dtn.21} dégénèrent précisément en $\pa\overline{V_{\cI}}$ le long de $\pa\overline{\bbH^d}$.  Pour le voir, on peut utiliser les coordonnées sphériques
$$
     \rho_2:= \sqrt{|Q_{\cI}|^2+|Q_{2,\cI^c}|^2}, \quad \omega_{2,\cI}:= \frac{Q_{\cI}}{\rho_2}, \quad \omega_{2,\cI^c}:= \frac{Q_{2,\cI^c}}{\rho_2},
$$
en quel cas les équations \eqref{dtn.21} prennent la forme
\begin{equation}
\mu_{\sigma}(\omega_{2,\cI^c})=u_2^2\omega_{\Im\bbH}, \quad \mu_N(\omega_{2,\cI}, \omega_{2,\cI^c})=-u_2^2u\zeta,
\label{dtn.22}\end{equation}
où $u_2:= \rho_2^{-1}$ est une fonction bordante pour $\pa \overline{\bbH^d}$ dans $\overline{\bbH^d}$.  Or l'éclatement de $\pa \overline{V}_{\cI}$ correspond à remplacer $\omega_{2,\cI^c}$ par $Q_{2,\cI^c}=\frac{\omega_{2,\cI^c}}{u_2}$, en quel cas les équations \eqref{dtn.22} deviennent
\begin{equation}
\mu_{\sigma}(Q_{2,\cI^c})=\omega_{\Im\bbH} \quad \mbox{et} \quad   \mu_{N,\cI}(\omega_{2,\cI})+ u_2^2\mu_{N,\cI^c}(Q_{2,\cI^c})=-u_2^2u\zeta.
\label{dtn.23}\end{equation}

Comme la restriction de l'action de $N$ à $V_{\cI}$ est aussi libre sur $V_{\cI}\setminus \{0\}$ et comme $\omega_{2,\cI}\ne 0$ puisque $\omega_{2,\cI}\in \pa \overline{V_{\cI}}$, la différentielle de $\mu_{N,\cI}$ est bien surjective lorsque restreinte aux solutions de \eqref{dtn.23} près de $u_2=0$.  Ces équations définissent donc bien une $p$-sous-variété près de l'intérieur de $\hH_2\cap \hH_3$. En remplaçant $Q_{2,\cI^c}$ par des coordonnées sphériques, on vérifie aisément que ça reste le cas près de $\hH_4$.  

Cela montre que $\hX_\zeta$ est bien une $p$-sous-variété de $\hW$.  Comme les équations définissant ces $p$-sous-variétés se divisent sur chaque hypersurface bordante $\hH_i$ en des équations sur la base $\hS_i$ et des familles d'équations sur les fibres $\widehat{F}_i$ de $\hphi_i:\hH_i\to\hS_i$ paramétrées par la base, la structure de fibrés itérés de $\hW$ induit clairement par restriction une structure de fibrés itérés sur $\hX_{\zeta}$.

\end{proof}

Comme $\mu_{\sigma}=(\Id\otimes\sigma)^*\circ \mu$, remarquons que cette application est constante le long des orbites de l'action de $T^d$ sur $\bbH^d$.  En particulier, $\mu_\sigma$ est constante le long des orbites de $N$, ce qui montre que l'action de $N$ sur $\bbH^d\times \Im\bbH$ se restreint à une action sur $X_{\zeta}$.  La description de la compactification $\hX_\zeta$ donnée dans la preuve de la Proposition~\ref{dtn.16} montre que cette action se prolonge en une action de $N$ sur $\hX_{\zeta}$.  Pour $i\in\{1,3\}$, l'action de $N$ sur $\hH_i=\hF_i\times \hS_i$ est induite par une action de $N$ sur $\hF_i$ (elle-même induite par l'action de $N$ sur $V_{\cI^c}$) et une action de $N$ sur $\hS_i$ induite par celle de $N$ sur $V_\cI$.   Pour $i=2$, l'action de $N$ sur $\hH_2=\hF_2\times \hS_2$ est triviale sur $\hS_2$, tandis que sur $\hF_2$, elle est induite par celle sur $\bbH^d$.  Enfin, pour $i=4$, puisque $\hH_4=\hS_4$, l'action de $N$ sur $\hH_4$ correspond à une action sur la base $\hS_4$.    En particulier, pour chaque $i$, l'application $\hphi_i:\hH_i\to\hS_i$ est une application $N$-équivariante par rapport aux actions de $N$ sur $\hH_i$ et $\hS_i$.    

Par hypothèse, on sait que $N$ agit librement sur $\mu_N^{-1}(-\zeta)$, donc librement sur $\mu^{-1}_N(-\zeta)\times \Im\bbH$.  Puisque la sous-variété $X_{\zeta}$ est contenue dans $\mu^{-1}_N(-\zeta)\times \Im\bbH$, on en déduit que l'action de $N$ sur $X_{\zeta}$ est libre.  Par la preuve de la Proposition~\ref{dtn.16}, on sait que l'action de $N$ sur $\hH_i\cap \hX_{\zeta}$ est libre pour $i\in{1,3}$, puisqu'elle est libre sur   $\hphi_{i}(\hH_i\cap \hX_{\zeta})$.  On a vu aussi que l'action de $N$ sur $\hphi^{-1}_{2}(p)\cap \hX_{\zeta}$ est libre pour tout $p\in \hS_2$, donc libre sur $\hH_2\cap \hX_{\zeta}$.  Par la Remarque~\ref{cone.4}, elle est libre aussi sur $\hH_4\cap \hX_{\zeta}$.  On en déduit que $N$ agit librement sur $\hX_{\zeta}$.

Son action est aussi compatible avec la structure de fibrés itérés de $\widehat{X}_\zeta$, de sorte que le quotient
\begin{equation}
 \hM_\zeta:= \hX_\zeta/N
\label{dtn.23b}\end{equation}     
est naturellement une variété à coins fibrés avec structure de fibrés itérés induite par celle de $\hX_{\zeta}$.  Comme $N$ agit par isométries sur $\bbH^d\times \Im\bbH$, la métrique $(g_{\bbH^d}+g_{\Im\bbH})|X_{\zeta}$ induit par passage au quotient une métrique $\gp$-tordue $\QAC$ $g_\cT$ sur $\hM_{\zeta}$.  L'application moment hyperhamiltonienne $\mu_{\sigma}$ étant constante le long des orbites de l'action de $N$, elle définit par restriction sur $\hX_{\zeta}$ une application moment hyperhamiltonienne sur $\hM_{\zeta}$.

Soit $\xi_1\in \CI(\bbH^d\times\Im\bbH;T(\bbH^d\times \Im\bbH))$ le générateur infinitésimal de l'action de $\bbR$ sur $\bbH^d\times \Im\bbH$.  Comme cette action est induite par $\sigma:\bbR\hookrightarrow \gt^d$ et que $T^d\subset \Sp(d)\subset \SU(2d)$, le champ vectoriel $\xi_1$ est homogène de degré $0$ et est induit par un champ vectoriel sur la sphère unité $\bbS(\bbH^d\times \Im\bbH)$.  Cela montre que $\xi_1$ se prolonge en un champ vectoriel bordant sur $\overline{\bbH^d\times \Im\bbH}$ qu'on dénotera aussi par $\xi_1$.  Comme $V_{\cI}\times \Im\bbH$ est le sous-espace des points fixes de l'action de $\bbR$, on aura que $\overline{V_{\cI}\times \Im\bbH}$ est la sous-variété où le champ vectoriel $\xi_1$ s'annule.  Puisque $\pa\overline{V_{\cI}\times \{0\}}$ est une $p$-sous-variété de $\overline{V_{\cI}\times \Im\bbH}$, on en déduit que $\xi_1$ se relève en un champ vectoriel bordant sur $W$ tangent aux fibres du fibré induit par l'application de contraction sur l'hypersurface bordante $H_1$.  Il se relève aussi automatiquement en un champ vectoriel bordant $\txi_1$ sur $\tW$.   

Comme le champ vectoriel $\txi_1$ s'annule sur les $p$-sous-variétés qu'on éclate pour obtenir $\hW$, il se relève en un champ vectoriel bordant $\hxi_1$ sur $\hW$ tangent aux fibres de $\hphi_i: \hH_i\to \hS_i$ pour $i<4$.  Puisque $\xi_1$ est tangent à la sous-variété $X_\zeta$, le champ vectoriel $\hxi_1$ est tangent à la  $p$-sous-variété $\hX_\zeta$.  Clairement, $\hxi_1$ ne s'annule que sur la fermeture $\widehat{V_\cI\times \Im\bbH}$ de $V_\cI\times \Im\bbH$ dans $\hW$.  Or, par \eqref{dtn.19} et \eqref{dtn.21}, on voit que $\hX_\zeta\cap\hH_i$ est disjoint de $\widehat{V_{\cI}\times\Im\bbH}$ pour $i>1$.  Ainsi, près de $(\hH_2\cup\hH_3\cup\hH_4)\cap \hX_{\zeta}$, le champ vectoriel $\hxi_1$ engendre un $(\hphi,\gp)$-feuilletage $\cF$ sur $\hX_{\zeta}$.  Bien sûr, les feuilles de $\cF$ correspondent aux orbites de l'action de $\bbR$ naturellement prolongée à $\hX_{\zeta}$.  

Comme l'action de $\bbR$ commute avec celle de $N$, on a une action induite de $\bbR$ sur le quotient $\hM_{\zeta}$ et $N$ agit sur le feuilletage $\cF$.  D'autre part, comme $\sigma(\bbR)$ n'est pas contenu dans $\gn$, les feuilles de $\cF$ sont transverses aux orbites de l'action de $N$.  Le feuilletage $\cF$ induit donc sur $\hM_{\zeta}$ un feuilletage $\cF/N$.  Si on dénote aussi par $\hphi$ et $\gp$ la structure de fibrés itérés et la fonction de pondération de $\hM_{\zeta}$ induites par celles de $\hX_{\zeta}$, alors $\cF/N$ est aussi un $(\hphi,\gp)$-feuilletage sur $\hM_{\zeta}$.

Comme $\bbR$ et $N$ agissent par isométries sur $\bbH^d\times \Im\bbH$, cela signifie que la distance $\rho$ par rapport à l'origine dans $\bbH^d\times \Im\bbH$ est constante le long des feuilles de $\cF$, ainsi que le long des orbites de $N$.  La fonction $\rho$ induit donc une fonction sur le quotient $\hM_{\zeta}$. On en déduite que le feuilletage $\cF/N$ et la fonction $\rho$ vue comme une fonction de distance $\gp$-tordue sur $\hM_{\zeta}$  définissent une classe de métriques $(\cF/N,\gp)$-tordues quasi-feuilletées au bord sur $\hX_{\zeta}$.   Cela nous permet de formuler le résultat principal de cette section.

\begin{theoreme}
Soient $(M_\zeta,g_{\zeta})= (\mu^{-1}(-\zeta)/N,g_\zeta)$ une métrique hyperkählérienne torique asymptotiquement conique du Corollaire~\ref{cone.3} et $\sigma:\bbR\hookrightarrow \gt^d$ une application linéaire dont l'image n'est pas contenue dans l'algèbre de Lie $\gn$ de $N$.  Alors la déformation de Taub-NUT  $g_{\zeta,\TN_\sigma}$ de  $g_{\zeta}$ est une métrique $(\cF/N,\gp)$-tordue quasi-feuilletée au bord sur la variété à coins fibrés $\hM_\zeta=\hX_{\zeta}/N$.  En particulier, ce résultat montre que $g_{\zeta,\TN_\sigma}$ est complète à géométrie bornée.    
\label{dtn.25}\end{theoreme}
\begin{proof}
Par la Proposition~\ref{dtn.16}, $(g_{\bbH^d}+g_{\Im\bbH})|_{\hX_{\zeta}}$ est une métrique $\gp$-tordue $\QAC$ exacte sur $\hX_{\zeta}$, ainsi que la métrique induite $g_\cT$ sur $\hM_{\zeta}$.  Par le Lemme~\ref{dtn.7}, la déformation de Taub-NUT transforme cette métrique en une métrique $(\cF/N,\gp)$-tordue quasi-feuilletée au bord pour le $(\hphi,\gp)$-feuilletage $\cF/N$ induit par les orbites de l'action de $\bbR$ sur $\hM_{\zeta}$.  Par \cite{ALN04,Bui}, cette métrique est complète à géométrie bornée.
\end{proof}

Le Théorème~\ref{dtn.25} montre entre autres que la courbure de $g_{\zeta,\TN_\sigma}$ est bornée, ainsi que toutes ses dérivées.  En utilisant la formule d'O'Neill, \cf  \cite{Bielawski2008} et \cite[Proposition~3.3]{Min2025}, on peut toutefois obtenir une estimation plus fine de la courbure sectionnelle comme suit.
\begin{corollaire}
La courbure sectionnelle $K_{g_{\zeta,\TN_\sigma}}$ de la déformation de Taub-NUT $g_{\zeta,\TN_\sigma}$ est telle que
$$
   |K_{g_{\zeta,\TN_\sigma}}|=\mathcal{O}(x_{4}^2\rho^{-2\gp})=\mathcal{O}(x_{4}^2v_{\frac12}^2).
$$
En particulier, si $V_{\cI}=\{0\}$, alors $\hH_1=\hH_3=\emptyset$ et 
$$
 |K_{g_{\zeta,\TN_\sigma}}|=\mathcal{O}(x_4^2\rho^{-1}).
$$ 
Cependant, si $V_{\cI}\ne\{0\}$, la courbure sectionnelle ne décroît pas partout vers zéro lorsqu'on s'approche de $(\hH_1\cap \hX_{\zeta})/N$.  
\label{dtn.26}\end{corollaire}
\begin{proof}
Soit $g_\cT$ la métrique $\gp$-tordue $\QAC$ sur $\hM_{\zeta}$ induite par restriction de $g_{\bbH^d}+g_{\Im\bbH}$ sur $\hX_{\zeta}$ et passage au quotient par l'action de $N$.  Alors le quotient de l'action de $\bbR$ sur $\hM_{\zeta}\times \bbR$ avec $\bbR$ agissant par translation dans le deuxième facteur induit une submersion riemannienne de $(\hM_{\zeta}\times \bbR, g_\cT+g_{\bbR})$ vers $((\hM_{\zeta}\times \bbR)/\bbR,g_{\zeta,\TN_{\sigma}})$. Par la formule de O'Neill \cite[9.29c]{Besse}, on a l'estimation 
\begin{equation}
    |K_{g_{\zeta,\TN_{\sigma}}}(X,Y)- K_{g_\cT+g_\bbR}(\widetilde{X},\widetilde{Y})|= \frac{3}{4} |[\widetilde{X},\widetilde{Y}]^v|^2_{g_\cT+g_\bbR}  \le \frac{3}{4} |[\widetilde{X},\widetilde{Y}]|^2_{g_\cT+g_\bbR},
\label{dtn.27}\end{equation}
où $X,Y$ sont des champs vectoriels sur $(\hM_\zeta\times\bbR)/\bbR$,  $\widetilde{X}$ et $\widetilde{Y}$ sont des relèvements horizontaux sur $\hM_{\zeta}\times \bbR$ et  $[\widetilde{X},\widetilde{Y}]^v$ est la partie verticale du crochet de Lie $[\widetilde{X},\widetilde{Y}]$.   

Dans la notation du Lemme~\ref{dtn.7}, soit $\hE^{\perp}:=T(\cF/N)^{\perp}$ le prolongement naturel de $E^{\perp}$ à $\hM_{\zeta}$ en tant que sous-fibré vectoriel de ${}^{\cF/N,\cT}T\hM_{\zeta}$, de sorte que
\begin{equation}
   {}^{\cF/N,\cT}T\hM_{\zeta}= \hE^{\perp}\oplus T(\cF/N).
\label{dtn.27b}\end{equation}
  Alors en prenant $X\in \CI(\hM_{\zeta};\hE^\perp)$ dans \eqref{dtn.27}, on a que $\widetilde{X}=(X,0)$ sur $\hM_{\zeta}\times \bbR$.  Par la preuve du Lemme~\ref{dtn.7}, on sait aussi que près de $((\hH_2\cup\hH_3\cup\hH_4)\cap\hX_\zeta)/N$, le champ vectoriel
$$
   \widetilde{Z}:= \frac{\hV_1 \hxi_1-e_1}{1+\hV_1}
$$
est horizontal et unitaire sur $(\hM_{\zeta}\times \bbR, g_\cT+g_{\bbR})$ avec projection $Z$ transverse à $\hE^\perp$ correspondant à $\hxi_1$ sur $\hM_{\zeta}\cong (\hM_{\zeta}\times \bbR)/\bbR$, où $\hV_1:=\frac{1}{g_\cT(\hxi_1,\hxi_1)}$ et $\{e_1\}$ est la base canonique de $\bbR$.  

Clairement, pour obtenir le résultat, il suffit donc de montrer que 
\begin{equation}
|K_{g_{\zeta,\TN_{\sigma}}}(X,Y)|= \cO(x_4^2v_{\frac12}^2) \quad \forall X,Y\in \CI(\hM_\zeta;\hE^{\perp})\cup \{Z\}
\label{dtn.28}\end{equation}
près de $((\hH_2\cup\hH_3\cup\hH_4)\cap\hX_\zeta)/N$.  Or, puisque $\hV_1\in x_4^2v_{\frac12}^2\CI(\hM_{\zeta})$, on a que 
$$
     \frac{\hV_1\hxi_1}{1+\hV_1}\in \CI(\hM_{\zeta};{}^{\cT}T\hM_{\zeta})=v_{\frac12}\cV_{\nQFB}(\hM_\zeta).
$$
En utilisant \eqref{dtn.27}, on voit donc que \eqref{dtn.28} découle de la Remarque~\ref{tqfb.12b} et du Lemme~\ref{tqfb.12d}.

Finalement, pour voir que la courbure sectionnelle ne décroît pas partout vers zéro en $(\hH_1\cap\hX_\zeta)/N$ lorsque $\hH_1\ne \emptyset$, remarquons que par \eqref{dtn.17}, le modèle dans les fibres de $(\hH_1\cap\hX_\zeta)/N$ correspond à la déformation de Taub-NUT $g_{V_{\cI^c},\TN_\sigma}$ (spécifiée par $\sigma$) de la métrique euclidienne $g_{V_{\cI^c}}$ sur $V_{\cI^c}$ induite par $g_{\bbH^d}$.  En effet, comme l'action de $N$ est libre sur $\mu_N^{-1}(0)\setminus\{0\}$, elle sera libre sur la base de $\hH_1\cap\hX_{\zeta}$ spécifiée par l'équation \eqref{dtn.18b}, donc les fibres de $\hH_1\cap\hX_{\zeta}/N$ sont bien données par \eqref{dtn.17} sans avoir à prendre un quotient par l'action d'un sous-groupe de $N$.  Or, par le Théorème~\ref{dtn.25} et le Lemme~\ref{vg.1}, la croissance du volume de $g_{V_{\cI^c},TN_\sigma}$ n'est pas maximale.  Cette métrique ne peut donc pas être plate, car autrement elle coïnciderait avec une métrique euclidienne et aurait une croissance du volume maximale.  

Maintenant, par le modèle \eqref{tqfb.12c} d'une métrique tordue $\QAC$ exacte et en prenant compte la déformation de Taub-NUT, on a que près de $(\hH_1\cap\hX_{\zeta})/N$, la métrique $g_{\zeta,\TN_\sigma}$ est asymptotique à une métrique de la forme
\begin{equation}
\frac{du_1^2}{u_1^4}+ \frac{\widetilde{\pr}_1^*\widehat{\varphi}_1^*g_{\hB_1}}{u_1^2}+ \widetilde{\pr}_1\kappa,
\label{cb.1}\end{equation}
où $\widehat{\varphi}_1: (\hH_1\cap\hX_\zeta)/N\to \hB_1$ est le fibré induit par la structure de fibrés itérés de $\hM_\zeta$, $g_{\hB_1}$ est une métrique sur la base $\hB_1$ de ce fibré et  $\kappa$ est une famille de métriques induisant sur chaque fibre la métrique $g_{V_{\cI^c},\TN_\sigma}$.  Comme cette métrique n'est pas plate, la courbure sectionnelle du modèle \eqref{cb.1}, et donc celle de $g_{\zeta,\TN_\sigma}$, ne peut pas partout tendre vers zéro lorsque $u_i\searrow 0$, d'où le résultat.

\end{proof}

Le Théorème~\ref{dtn.25} nous permet aussi de déterminer le cône tangent à l'infini de la déformation de Taub-NUT $g_{\zeta,\TN_\sigma}$.  Pour ce faire, remarquons d'abord que l'homomorphisme de groupes 
$$
    \exp\circ \sigma: \bbR\to T_{\cI}\subset T^d
$$
est une immersion.  Soit $T_{\sigma}$ le sous-tore de $T_{\cI}$ correspondant à la fermeture de l'image de cette immersion, c'est-à-dire que 
$$
    T_{\sigma}:= \overline{\exp\circ\sigma(\bbR)}.
$$
En tant que sous-groupes de $T^d$, les sous-tores $T_{\sigma}$ et $N$ agissent par isométries sur le cône
$$
   \mu_N^{-1}(0)\cap\mu_{\sigma}^{-1}(0)\subset \bbH^d.
$$ 
L'action de $N$ sur $\mu_N^{-1}(0)\cap\mu_\sigma^{-1}(0)\setminus \{0\}$ est en fait libre et le quotient 
$$
   M_{0,\sigma}:= (\mu_N^{-1}(0)\cap \mu_\sigma^{-1}(0))/N
$$
est un sous-cône du cône hyperkählérien $(M_0,g_0)$, où $g_0$ est la métrique hyperkählérienne du quotient hyperkählérien $M_0=\mu^{-1}_N(0)/N$.  Puisque la différentielle de $\mu_\sigma$ s'annule sur $V_\cI$, remarquons que $M_{0,\sigma}$ est singulier en $(\mu_N^{-1}(0)\cap V_{\cI})/N$.  L'action de $T_\sigma$ sur $\mu_N^{-1}(0)\cap\mu^{-1}_\sigma(0)$ induit une action par isométries sur le quotient $M_{0,\sigma}$.  Cette action est effective si et seulement si $T_\sigma\cap N=\{\Id\}$.  En particulier, le quotient $M_{0,\sigma}/T_\sigma$ est aussi un cône métrique.  
\begin{corollaire}
La déformation de Taub-NUT $g_{\zeta,\TN_\sigma}$ a pour unique cône tangent à l'infini le cône métrique
$$
    (M_{0,\sigma}/T_{\sigma})\times \Im\bbH.
$$
\label{dtn.29}\end{corollaire}
\begin{proof}
Rappelons qu'un cône tangent à l'infini d'une variété riemannienne complète $(M,g)$ est la limite Gromov-Hausdorff pointée d'une suite $(M, \frac{g}{\lambda_i^2},p)$ pour $p\in M$ un point fixé et $\{\lambda_i\}$ une suite strictement croissante de nombres réels strictement positifs telle que $\lambda_i\nearrow \infty$.  Dans le cas d'une métrique $\gp$-tordue $\QAC$ exacte ayant une unique hypersurface bordante maximale $H_{\max}$, on vérifie directement que le cône tangent à l'infini est unique et donné par le modèle de type produit \eqref{tqfb.12c} de cette métrique en $H_{\max}$.  Dans le cas de la métrique $\gp$-tordue $\QAC$ $g_{\cT}$ sur $\hM_{\zeta}$ induite par la métrique de la Proposition~\ref{dtn.16} par passage au quotient, on en déduit que son unique cône tangent à l'infini est le cône
$$
     M_{0,\sigma}\times \Im\bbH
$$ 
avec métrique 
\begin{equation}
 g_\cC= d\rho^2+\rho^2 g_{\hB_4}
\label{dtn.29b}\end{equation}
sous l'identification $M_{0,\sigma}\times \Im\bbH= C\cB_4$, où $\hB_4=(\hH_4\cap \hX_{\zeta})/N$ est l'hypersurface bordante maximale de $\hM_\zeta$, $\cB_4$ est la variété stratifiée associée à $\hB_4$, $g_{\hB_4}$ est une métrique wedge exacte sur la variété à coins fibrés $\hB_4$ et 
$$
   C\cB_4:= ([0,\infty)\times \cB_4)/(\{0\}\times \cB_4)
$$
est le cône sur $\cB_4$.

Sur $\hB_4$, l'action de $\bbR$ (induite par $\sigma$) est localement libre et les orbites de cette action correspondent aux feuilles de la restriction du $(\hphi,\gp)$-feuilletage $\cF/N$ à $\hB_4$.  Maintenant, la décomposition \eqref{dtn.27b} donne lieu aussi à une décomposition orthogonale par rapport à la métrique $g_{\cT}$, à savoir la décomposition
$$
    {}^{\cT}T\hM_\zeta= \hE^\perp\oplus (x_4v_{\frac12})T(\cF/N).
$$
En restreignant à $\hB_4$, on obtient une décomposition plus fine
\begin{equation}
  {}^{\cT}T\hM_\zeta|_{\hB_4}= \langle \frac{\pa}{\pa \rho} \rangle\oplus \hE^\perp_{\hB_4}\oplus (x_4v_{\frac12})T(\cF/N),
\label{dtn.30}\end{equation}
où $\hE^\perp_{\hB_4}$ est le sous-fibré de $\hE^\perp|_{\hB_4}$ induisant la décomposition orthogonale
$$
   \hE^\perp|_{\hB_4}= \langle\frac{\pa}{\pa \rho}\rangle\oplus \hE^\perp_{\hB_4}.
$$
On a aussi une décomposition correspondante du fibré wedge sur $\hB_4$ 
\begin{equation}
 {}^{w}T\hB_4= \rho\hE^\perp_{\hB_4}\oplus \rho(x_4v_{\frac12})T(\cF/N)= \rho\hE^\perp_{\hB_4}\oplus (x_1x_2x_3)^{-1}T(\cF/N).
\label{dtn.31}\end{equation}
En termes de cette décomposition, la métrique \eqref{dtn.29b} prend la forme
\begin{equation}
  g_{\cC}= d\rho^2+\rho^2(g_{\hE^{\perp}_{\hB_4}}+ g_{\cF/N}),
\label{dtn.31b}\end{equation}
où $g_{\hE^{\perp}_{\hB_4}}$ et $g_{\cF/N}$ sont des métriques wedges sur $\rho\hE^{\perp}_{\hB_4}$ et $(x_1x_2x_3)^{-1}T(\cF/N)$.

Par le Lemme~\ref{dtn.7}, le modèle correspondant pour $g_{\zeta,\TN_\sigma}$ est 
$$
  g_{\cC,\TN_\sigma}= d\rho^2+ \rho^2(g_{\hE^{\perp}_{\hB_4}}+ V_1g_{\cF/N})= (d\rho^2+\rho^2g_{\hE^{\perp}_{\hB_4}})+ \eta_1^2
$$
avec $V_1\in (x_4v_{\frac12})^2\CI(\hM_\zeta)$.

Puisque $(M_{\zeta},g_\cT)$ a pour unique cône tangent à l'infini la métrique \eqref{dtn.29b}, on en déduit que lorsque $\lambda\to \infty$, $(M_\zeta, \frac{g_{\zeta,\TN_\sigma}}{\lambda^2},p)$ devient de plus en plus proche de 
$$
   (C\hB_4, (d\rho^2+\rho^2g_{E^\perp_{\hB_4}})+ \frac{\eta_1^2}{\lambda^2},0)
$$
dans la topologie de Gromov-Hausdorff pointée, de sorte que le cône tangent à l'infini de $(M_\zeta, g_{\zeta,\TN_\sigma})$ est unique et coïncide avec la limite
$$
  \lim_{\lambda\to \infty} (C\hB_4, (d\rho^2+\rho^2g_{E^\perp_{\hB_4})}+ \frac{\eta_1^2}{\lambda^2},0)
$$
dans la topologie de Gromov-Hausdorff pointée.  Or, pour cette famille de métriques, la distance entre deux points dans la fermeture d'une feuille du feuillage $\cF/N|_{\hB_4}$  relevé à $[0,\infty)\times \hB_4$ tend vers zéro lorsque $\lambda\to \infty$, si bien qu'à la limite $\lambda\to \infty$, on a plus généralement que la distance entre deux points $p$ et $q$ de $C\hB_4$ tend vers la distance entre leurs images sur le quotient $(M_{0,\sigma}\times \Im\bbH)/T_\sigma= (M_{0,\sigma}/T_\sigma)\times \Im\bbH$ avec la métrique induite $g_{M_{0,\sigma}/T_\sigma}+g_{\Im\bbH}$.  On en déduit donc que 
$$
 \lim_{\lambda\to \infty} (C\hB_4, (d\rho^2+\rho^2g_{E^\perp_{\hB_4})}+ \frac{\eta_1^2}{\lambda^2},0)= ((M_{0,\sigma}/T_\sigma)\times \Im\bbH, g_{M_{0,\sigma}/T_\sigma}+g_{\Im\bbH},0), 
$$
d'où le résultat.

\end{proof}

Prenons le temps de bien décrire le cône $(M_{0,\sigma}/T_\sigma)\times \Im\bbH$ apparaissant dans le Corollaire~\ref{dtn.29}.  D'abord, remarquons que les orbites de l'action de $T_\sigma$ sur $\hB_4$ correspondent à la fermeture des feuilles du feuilletage $\cF/N|_{\hB_4}$.  En particulier, sur chaque hypersurface bordante de $\hB_4$, les orbites de cette action sont tangentes aux fibres du fibré provenant de la structure de fibrés itérés.  

Or, par \cite{AM2011}, on sait qu'il existe une suite de strates de $\hB_4$ (associées aux différents stabilisateurs de l'action de $T_\sigma$ sur $\hB_4$) qui lorsqu'éclatées donnent une variété à coins $\hW_4$ munie d'une application de contraction
$$
   \beta_4: \hW_4\to \hB_4
$$ 
telle que l'action de $T_{\sigma}$ sur $\hB_4$ se relève en une action libre de $T_\sigma/(T_\sigma\cap N)$ sur $\hW_4$.  Sur chaque hypersurface bordante $\hH\in \cM_1(\hB_4)$, les strates à éclater sont transverses aux fibres du fibré $\hphi_{\hH}: \hH\to \widetilde{S}_{\hH}$ (de la structure de fibrés itérés de $\hB_4$) et sont envoyées surjectivement sur $\hS_{\hH}$ par $\hphi_{\hH}$.   

Par la construction de \cite{AM2011}, cela implique que $\hW_4$ est naturellement une variété à coins fibrés avec une hypersurface bordante $\hH\in\cM_1(\hB_4)$ se relevant en une hypersurface bordante $\hH'$ de $\hW_4$ ayant la même base $\hS_{\hH}$.  Par rapport à l'ordre partiel sur $\cM_1(\hW_4)$, une telle hypersurface bordante est toujours inférieure aux hypersurfaces bordantes la coupant et provenant d'une des $p$-sous-variétés éclatées pour obtenir $\hW_4$.  Les  fibrés de la structure de fibrés itérés de $\hW_4$ étant naturellement $T_\sigma$-équivariants par rapport à l'action de $T_\sigma$ sur $\hW_4$ et sur les bases des fibrés itérés, le quotient $\hW_4/T_\sigma$, qui est par \cite{AM2011} une variété à coins, est aussi automatiquement une variété à coins fibrés.

La métrique wedge $g_{\hB_4}$ sur $\hB_4$ se relève en une métrique wedge $g_{\hW_4}$ sur $\hW_4$.  Comme $T_\sigma$ agit par isométries par rapport à la métrique $g_{\hW_4}$ et comme les orbites de cette action sont tangentes aux fibres des fibrés de la structure de fibrés itérés de $\hW_4$, la métrique $g_{\hW_4}$ induit par passage au quotient une métrique wedge $g_{\hW_4/T_\sigma}$ sur la variété à coins fibrés $\hW_4/T_{\sigma}$.  Si $\widehat{\cW}_4$ et $\widehat{\cW}_4/T_\sigma$ dénotent les espaces stratifiés associés à $\hW_4$ et $\hW_4/T_\sigma$, alors 
\begin{equation}
   M_{0,\sigma}\times \Im\bbH= C(\widehat{\cW}_4/T_\sigma) 
\label{dtn.32}\end{equation}
et la métrique sur $M_{0,\sigma}/T_\sigma\times \Im\bbH$ est 
\begin{equation}
  d\rho^2+\rho^2g_{\hW_4/T_\sigma}.
\label{dtn.33}\end{equation}
En particulier, on en déduit qu'en tant qu'espace stratifié, 
\begin{equation}
\dim (M_{0,\sigma}/T_\sigma\times \Im\bbH)= 4(d-\dim N)-\dim T_\sigma+ \dim T_\sigma\cap N.
\label{dtn.34}\end{equation}

Illustrons les résultats de cette section par quelques exemples.
\begin{exemple}\emph{(Métrique de Taubian-Calabi)} Lorsque $N=\{\Id\}$, $\zeta=0$ et 
$$
   \begin{array}{llcl}
    \sigma: & \bbR &\to & \gt^d=\bbR^d \\
     & t & \mapsto & (t,\ldots,t)
   \end{array}
$$
est l'inclusion diagonale, la métrique $g_{\zeta,\TN_\sigma}$ correspond à la métrique de Taubian-Calabi.  Dans ce cas, $V_\cI=\{0\}$ et $T_\sigma=\exp\circ \sigma(\bbR)$ est le cercle diagonal dans $T^d$.  Le Théorème~\ref{dtn.25} et ses corollaires permettent alors de retrouver les résultats de \cite{Min2025}, à savoir que la courbure sectionnelle est $\cO(\rho^{-1})$ (en fait améliore cette décroissance à l'infini à $\cO(x_4\rho^{-1})$ par  le Corollaire~\ref{dtn.26} ), que le cône tangent à l'infini est $\mu^{-1}_\sigma(0)/T_\sigma\times \Im \bbH$ et que la croissance du volume est d'ordre $4d-1$ par le Lemme~\ref{vg.1}.  Ils permettent aussi d'obtenir une description plus fine du comportement asymptotique de la métrique à l'infini et de montrer que la variété riemannienne correspondante est à géométrie bornée.  Plus généralement, comme dans \cite[Exemple 7.4]{Min2025} on obtient le même genre de résultats en prenant 
$$
   \begin{array}{llcl}
    \sigma: & \bbR &\to & \gt^d=\bbR^d \\
     & t & \mapsto & (a_1t,\ldots,a_dt)
   \end{array}
$$
avec $a_1,\ldots, a_d$ des entiers non nuls. 
\label{dtn.35}\end{exemple}

\begin{exemple}
Toujours dans le cas où $N=\{\Id\}$ et $\zeta=0$, on peut aussi se mettre dans le cas générique où l'inclusion $\sigma: \bbR\hookrightarrow \gt^d$ est telle que 
$$
    T_\sigma= \overline{\exp\circ \sigma(\bbR)}= T^d.  
$$
Dans ce cas, $V_\cI=\{0\}$.  À nouveau, le Théorème~\ref{dtn.25}  et ses corollaires montrent que la métrique $g_{\zeta,\TN_\sigma}$ a une  courbure sectionnelle qui est $\cO(x_4\rho^{-1})$ et que sa croissance du volume est d'ordre $4d-1$.  Contrairement à l'exemple précédent toutefois, lorsque $d>1$, le cône tangent à l'infini, $\mu^{-1}_\sigma(0)/T^d\times \Im\bbH$, est de dimension $3d$, donc strictement plus petite  que l'ordre $4d-1$ de la croissance du volume.  
\label{dtn.36}\end{exemple}
En faisant varier $\sigma$ dans cet exemple, on peut obtenir plusieurs métriques hyperkählériennes distinctes comme le montre le résultat suivant.
\begin{lemme}
Pour $d\ge 2$, les différents choix possibles de $\sigma$ dans l'Exemple~\ref{dtn.36} donnent lieu à une infinité non dénombrable de métriques hyperkählériennes mutuellement non isométriques sur $\bbC^{2d}$, même à une homothétie près.
\label{dtn.36b}\end{lemme}
\begin{proof}
Avec le modèle asymptotique \eqref{dtn.31b},  on voit que le $(\hphi,\gp)$-feuilletage $\cF$ restreint à $\hH_4\cap \hX_\zeta$ est géométriquement déterminé par $g_{\zeta,\TN_\sigma}$.  Si $d=2$, en regardant la fermeture des feuilles de ce feuilletage pour des feuilles ne coupant par $\widehat{V_{\cI}\times \Im\bbH}$ pour $\cI\subset \{1,\ldots,d\}$, on voit qu'à ce feuilletage est associé canoniquement  un feuilletage de Kronecker $K_{\sigma}$ sur $T^2$ de pente 
$$
     \alpha=  \frac{a_2}{a_1} \quad \mbox{si} \quad \sigma(t)= (a_1t,a_2t).
$$ 
Or, par \cite{Connes1981,Rieffel1981}, voir aussi \cite{Donato-Iglesias} pour une preuve élémentaire,  des feuilletages de Kronecker de pentes $\alpha$ et $\beta$ seront difféomorphes si et seulement si 
$$
      \beta= \frac{a\alpha+b}{c\alpha+d} \quad \mbox{pour un certain} \quad \lrp{\begin{array}{cc} a& b \\ c& d  \end{array}} \in \GL(2,\bbZ).
$$
On a donc une infinité non-dénombrable de choix de $\sigma$ donnant lieu à des métriques hyperkählériennes mutuellement non isométriques, même à une homothétie près.  

Si $d>2$, alors pour chaque sous-ensemble $\cI\subset \{1,\ldots,d\}$ de cardinalité $d-2$, $T_{\cI}$ est un tore de dimension $2$ et les fermetures  des feuilles du $(\hphi,\gp)$-feuilletage sur $(\widehat{V_{\cI}\times \Im\bbH})\cap \hH_4\cap \hX_{\zeta}$ engendrent un feuilletage de Kronecker $K_{\sigma,\cI}$ de pente
$$
     \alpha= \frac{a_i}{a_j}
$$ 
si $\sigma(t)=(a_1t,\ldots,a_dt)$ et $\cI=\{1,\ldots,d\}\setminus \{i,j\}$.  En faisant varier $\cI$, on peut donc utiliser ces feuilletages de Kronecker pour déduire à nouveau qu'il y a une infinité non dénombrable de choix de $\sigma$ donnant lieu à des métriques hyperkählériennes mutuellement non isométriques, même à une homothétie près. 

\end{proof}

Bien sûr, lorsque $d=1$, les Exemples~\ref{dtn.35} et \ref{dtn.36}  coïncident et correspondent à la métrique de Taub-NUT sur $\bbC^2$.  Les exemples précédents sont des métriques hyperkählériennes sur $\bbC^{2d}$ correspondant à des déformations de Taub-NUT d'ordre $1$ de $g_{\bbH^d}$.  On peut aussi considérer des déformations de Taub-NUT de la métrique de Calabi.
 \begin{exemple}
 Prenons $N$ le cercle diagonal dans $T^d$ et $\zeta\in (\gn^*\otimes \bbR^3)\setminus \{0\}$, de sorte que $g_{\zeta}$ est la métrique de Calabi sur $T^*\bbC\bbP^{d-1}$ comme décrite dans l'Exemple~\ref{con.5}.  En prenant une inclusion $\sigma:\bbR\hookrightarrow \gt^d$ transverse à $\gn$, on obtient une déformation de Taub-NUT d'ordre $1$ $g_{\zeta,\TN_\sigma}$ pour laquelle le Théorème~\ref{dtn.25} et ses corollaires s'appliquent.
 \label{dtn.38}\end{exemple}
 Regardons quelques cas plus en détail.
 \begin{exemple}
 Si on prend $\sigma$ tel que $T_\sigma=T^d$ dans l'Exemple~\ref{dtn.38}, alors la déformation de Taub-NUT a une courbure sectionnelle $\cO(x_4\rho^{-1})$ et son cône tangent à l'infini $M_{0,\zeta}\times \Im\bbH$ est de dimension $3(d-1)$ strictement inférieure à l'ordre $4(d-1)-1$ de la croissance de son volume lorsque $d>2$.  
 \label{dtn.40}\end{exemple}
 
 On peut aussi obtenir des déformations de Taub-NUT de la métrique de Calabi pour lesquelles la courbure sectionnelle ne décroît pas vers zéro à l'infini dans certaines directions.
 
 \begin{exemple}
Dans l'Exemple~\ref{dtn.38}, on peut choisir une inclusion $\sigma: \bbR\hookrightarrow \gt^d$ telle que
$$
    T_\sigma= \overline{\exp\circ \sigma(\bbR)}= T_\cI  
$$
pour un sous-ensemble $\cI\subset\{1,\ldots,d\}$ non vide et contenant strictement moins de $d-1$ éléments, c'est-à-dire tel que $1<\dim T_\cI<d$. En particulier, on a que $T_\sigma\cap N=\{\Id\}$.  Dans ce cas, $V_\cI\ne \{0\}$, donc par le Corollaire~\ref{dtn.26} la courbure sectionnelle de $g_{\zeta,\TN_\sigma}$ est plutôt seulement $\cO(x_4^2\rho^{-2\gp})=\cO(x_4^2v_{\frac12}^2)$ et ne décroit pas vers zéro dans certaines directions à l'infini.  Le cône tangent à l'infini, à savoir $(M_{0,\sigma}/T_\sigma) \times \Im\bbH$, est de dimension $4(d-1) -\dim T_\sigma=3d-4+|\cI |$, ce qui est strictement inférieur à l'ordre $4(d-1)-1$ de la croissance du volume de la métrique.   
\label{dtn.37}\end{exemple}

\begin{exemple}
Toujours dans le cadre de l'Exemple~\ref{dtn.38}, prenons plutôt une inclusion $\sigma$ de la forme $\sigma(t)=(a_1t,\ldots,a_dt)$ avec $a_i$ des entiers dont certains sont nuls.  Posons 
$$
   \cI= \{i\in\{1,\ldots,d\}\;|\; a_i=0\},
$$
de sorte que $\cI$ est le plus grand sous-ensemble de $\{1,\ldots,d\}$ tel que $\sigma(\bbR)\subset \gt_\cI$.  Dans ce cas $T_\sigma= \overline{\exp\circ\sigma(\bbR)}= \exp\circ\sigma(\bbR)$ est un cercle contenu dans $T_{\cI}$.  On  déduit du Théorème~\ref{dtn.25} et de ses corollaires que  la déformation de Taub-NUT $g_{\zeta,\TN_\sigma}$ a une courbure sectionnelle 
$\cO(x_4^2\rho^{-2\gp})=\cO(x_4^2v_{\frac12}^2)$ qui ne décroît pas dans toutes les directions à l'infini, une croissance du volume d'ordre $4(d-1)-1$ et un cône tangent à l'infini qui est de dimension $4(d-1)-1$.  

\label{dtn.37b}\end{exemple}

\begin{remarque}
On peut aussi prendre les inclusions $\sigma$ des Exemples~\ref{dtn.37} et \ref{dtn.37b} lorsque $N=\{\Id\}$ et $\zeta=0$, mais on obtient dans ce cas un produit cartésien de $\bbH^{|\cI|}$ avec une métrique des Exemples~\ref{dtn.36} et \ref{dtn.35} respectivement.
\label{dtn.39}\end{remarque}

\section{Déformations de Taub-NUT d'ordre maximal de la métrique euclidienne} \label{QFB.0}

Dans cette dernière section de l'article, nous allons étudier la géométrie à l'infini des déformations de Taub-NUT d'ordre maximal de la métrique euclidienne $g_{\bbH^d}$ de $\bbH^d$.  Plus précisément, nous allons établir que de telles déformations sont quasi-fibrées au bord au sens de \cite{CDR}.

Soit donc $\sigma: \bbR^d\to \mathfrak{t}^d$ une application linéaire bijective.  Par l'entremise de l'action canonique du tore diagonal $T^d$ sur $\bbH^d$, celle-ci induit une action hyperhamiltonienne de $\bbR^d$ sur $\bbH^d$ avec application moment hyperhamiltonienne
$$
   \mu_{\sigma}:= (\Id\otimes \sigma^*)\circ \mu,
$$
où $\mu: \bbH^d\to \bbR^3\otimes (\mathfrak{t}^d)^*$ dénote l'application moment hyperhamiltonienne de l'action du tore diagonal $T^d$ sur $\bbH^d$.  Si $\lambda: \bbH^d\to \bbR^3\otimes (\bbR^d)^*$ dénote l'application moment hyperhamiltonienne \eqref{dtn.2} avec $m=d$, alors la déformation de Taub-NUT d'ordre $d$ associée est alors donnée par 
$$
   (\mu_\sigma+\lambda)^{-1}(0)/\bbR^d,
$$ 
où $\mu_\sigma+\lambda: \bbH^d\times \bbH^d\to \bbR^3\otimes (\bbR^d)^*$ est l'application moment hyperhamiltonienne de l'action de $\bbR^d$ sur $\bbH^d\times \bbH^d$ spécifiée par $\mu_\sigma$ sur le premier facteur et par $\lambda$ sur le deuxième facteur.  Par \eqref{dtn.3}, l'inclusion 
$$
   X_{\bbH^d}:= (\mu_\sigma+\lambda)^{-1}(0)\cap (\bbH^d\times (\Im\bbH)^d)\subset (\mu_\sigma+\lambda)^{-1}(0)
$$
induit un difféomorphisme
$$
   X_{\bbH^d}\cong (\mu_\sigma+\lambda)^{-1}(0)/\bbR^d.
$$
Rappelons aussi que la sous-variété $X_{\bbH^d}\subset \bbH^d\times (\Im \bbH)^d$ correspond au graphe de $\mu_{\sigma}:\bbH^d\to \bbR^3\otimes (\bbR^d)^*\cong (\Im\bbH)^d$. 

Pour vérifier que l'Hypothèse~\ref{dtn.5} est satisfaite, il suffit d'utiliser l'application $\sigma$.  Plus précisément, via l'identification canonique $\mathfrak{t}^d=\bbR^d$, celle-ci peut être vue comme une application $\sigma:\bbR^d\to \bbR^d$.  Il suffit alors de prendre la base $\{\widetilde{e}_1,\ldots,\widetilde{e}_d\}$ de $\bbR^d$ donnée par 
\begin{equation}
   \widetilde{e}_i:= \sigma^{-1}(e_i),
\label{base.1}\end{equation}
où $\{e_1,\ldots, e_d\}$ est la base canonique de $\bbR^d$.  En effet, dans ce cas, le champ vectoriel $\widetilde{\xi}_i\in\CI(\bbH^d;T\bbH^d)$ correspondant à l'action infinitésimale de $\widetilde{e}_i$ spécifiée par $\mu_\sigma$ correspond aussi à l'action infinitésimale de $e_i\in \bbR^d\cong \mathfrak{t}^d$ pour l'action du tore diagonale $T^d$ spécifiée par $\mu$.  Cela signifie que $\widetilde{\xi}_i$ est tangent aux fibres de la projection
\begin{equation}
\begin{array}{lccl}
 \widehat{\pr}_i: & \bbH^d &\to & \bbH^{d-1}  \\
         & (p_1,\ldots,p_d) &\mapsto & (p_1,\ldots,p_{i-1},p_{i+1},\ldots,p_d)
\end{array}
\label{QFB.1}\end{equation}
omettant le $i$ ième facteur $\bbH$ de $\bbH^d$.  Les champs vectoriels $\widetilde{\xi}_1,\ldots,\widetilde{\xi}_d$ sont donc clairement mutuellement orthogonaux, les fibres des différentes projections \eqref{QFB.1} étant mutuellement orthogonales.

Ce choix de base suggère de réécrire les équations définissant $X_{\bbH^d}$, à savoir
$$
  \mu_\sigma(m)=q \quad \mbox{pour} \; (m,q)\in \bbH^d\times (\Im\bbH)^d,
$$
par
\begin{equation}
 \mu(m)= (\Id\otimes (\sigma^{-1})^*)(q) \quad \mbox{pour} \; (m,q)\in \bbH^d\times (\Im\bbH)^d.
\label{QFB.2}\end{equation}
Autrement dit, $(\Id\otimes (\sigma^{-1})^*): (\Im\bbH^d)\to (\Im\bbH)^d$ spécifie un changement de coordonnées linéaire sur $(\Im\bbH)^d$, à savoir
$$
   q_{\sigma}:= (\Id\otimes (\sigma^{-1})^*)(q).
$$
Dans les coordonnées $q_\sigma= (q_{\sigma,1},\ldots,q_{\sigma,d})\in (\Im\bbH)^d$, la sous-variété $X_{\bbH^d}$ est décrite par les équations 
$$
    \mu_{(T^d)_{\{i\}^c}}(m)= q_{\sigma,i}
$$
dans la notation de \eqref{qac.15}, où $ \mu_{(T^d)_{\{i\}^c}}: \bbH^d\to \Im\bbH$ correspond à l'application moment hyperhamiltonienne du $i$ ième cercle de $T^d=(\bbS^1)^d$.  En particulier, dans les coordonnées canoniques $p=(p_1,\ldots,p_d)\in\bbH^d$ avec $p_i\in\bbH$, on a que
$$
      \mu_{(T^d)_{\{i\}^c}}(p)= \mu_{\bbH}(p_i),
$$ 
où $\mu_{\bbH}:\bbH\to \Im \bbH$ dénote l'application moment hyperhamiltonienne de l'action du cercle $\bbS^1\subset \Sp(1)$ sur $\bbH$, c'est-à-dire le cas $d=1$ de l'action \eqref{action.1}.  Dans les coordonnées $(p,q_\sigma)$, les équations définissant $X_{\bbH^d}$ prennent donc la forme
\begin{equation}
\mu_{\bbH}(p_i)=q_{\sigma,i} \quad \mbox{pour} \; i\in\{1,\ldots,d\}.
\label{QFB.3}\end{equation}
Le prix à payer pour cette description simple de $X_{\bbH^d}$ est que la métrique euclidienne sur le facteur $(\Im\bbH^d)$ aura des termes croisés lorsque décrite dans les coordonnées $q_{\sigma}$.  Plus précisément, si $q_{\sigma,i}=(q^1_{\sigma,i},q^2_{\sigma,i},q^3_{\sigma,i})\in \bbR^3\cong\Im\bbH$, alors la métrique euclidienne prendra la forme
$$
   \sum_{k=1}^3 \sum_{i,j=1}^d g_{ij} dq^k_{\sigma,i}\otimes dq^k_{\sigma,j}
$$
avec les $g_{ij}$ correspondant aux coefficients d'une matrice symétrique $d\times d$ définie positive qui n'est pas forcément diagonale.  

Soit $\overline{\bbH^d\times (\Im\bbH)^d}$ la compactification radiale de $\bbH^d\times (\Im\bbH)^d$.  Si $r$ dénote la fonction de distance par rapport à l'origine sur $\bbH^d\times (\Im\bbH)^d$ en termes de la métrique euclidienne, alors près du bord $\pa\overline{\bbH^d\times (\Im\bbH)^d}$, on peut utiliser les coordonnées sphériques 
$$
      u=\frac{1}{r}, \quad \omega_i= \frac{p_i}{r} \quad \mbox{et} \quad \varpi_i=\frac{q_{\sigma,i}}{r}
$$
dans lesquelles les équations définissant $X_{\bbH^d}$ prennent la forme
\begin{equation}
  \mu_{\bbH}(\omega_i)=u\varpi_i \quad \mbox{pour} \quad i\in\{1,\ldots,d\}.
\label{QFB.4}\end{equation}
Soit $\overline{X}_{\bbH^d}$ la fermeture de $X_{\bbH^d}$ dans $\overline{\bbH^d\times (\Im\bbH)^d}$.  Alors les équations définissant $\pa \overline{X}_{\bbH^d}:= \overline{X}_{\bbH^d}\cap \pa \overline{\bbH^d\times (\Im\bbH)^d}$ sont obtenues en posant $u=0$ dans \eqref{QFB.4}, ce qui donne
\begin{equation}
\mu_{\bbH}(\omega_i)=0 \quad \mbox{pour} \quad i\in\{1,\ldots,d\}.
\label{QFB.5}\end{equation}
Comme $\mu_{\bbH}$ ne s'annule qu'à l'origine (le seul point fixe de l'action du cercle sur $\bbH$), les équations \eqref{QFB.5} correspondent simplement à 
\begin{equation}
 \omega_i=0  \quad \mbox{pour} \quad \quad i\in\{1,\ldots,d\}.
\label{QFB.6}\end{equation} 
En particulier, $\pa \overline{X}_{\bbH^d}= \pa \overline{(\Im\bbH)^d}$, où $\overline{(\Im \bbH)^d}$ est la fermeture de $\{0\}\times (\Im\bbH)^d$ dans $\overline{\bbH^d\times (\Im\bbH)^d}$ et $\pa \overline{(\Im \bbH)^d}$ est le bord de $\overline{(\Im \bbH)^d}$.

Comme la différentielle de $\mu_{\bbH}$ s'annule à l'origine, cela montre que $\overline{X}_{\bbH^d}$ est singulier sur $\pa\overline{\bbH^d\times (\Im\bbH)^d}$.  Pour résoudre ces singularités, nous allons effectuer des éclatements sur le bord de $\bbH^d\times (\Im\bbH)^d$.  Pour $\cI\subset \{1,\ldots,d\}$ un sous-ensemble non-vide, considérons le sous-espace
$$
 V_{\cI}= \{ (q_{\sigma,1},\ldots,q_{\sigma,d})\in (\Im\bbH)^d \; | \; q_{\sigma,i}=0 \; \mbox{pour} \; i\notin \cI\}\subset (\Im\bbH)^d
$$
 vu comme un sous espace de $\bbH^d\times (\Im\bbH)^d$ via l'identification canonique $\{0\}\times (\Im\bbH)^d=(\Im\bbH)^d$.  Dénotons par $\overline{V}_{\cI}$ sa fermeture dans $\overline{\bbH^d\times (\Im\bbH)^d}$.  Soit $\overline{s}_{\cI}:=\pa \overline{V}_{\cI}$ le bord de $\overline{V}_{\cI}$.  La relation d'inclusion induit un ordre partiel sur les sous-ensembles $\overline{s}_{\cI}$ en convenant que 
\begin{equation}
   \overline{s}_{\cI}\le \overline{s}_{\cJ} \; \Longleftrightarrow \; \overline{s}_{\cI}\subset \overline{s}_{\cJ} \; \Longleftrightarrow \; \cI\subset \cJ.
\label{QFB.7}\end{equation} 

La résolution de $\overline{\bbH^d\times (\Im\bbH)^d}$ dont on aura besoin consistera précisément à éclater ces sous-variétés dans une ordre compatible avec l'ordre partiel \eqref{QFB.7}.  Cependant, ces éclatements ne seront pas conventionnels.  Ils seront pondérés au sens de \cite[\S~4]{CR}.  En tant que sous-variété de $\overline{\bbH^d\times (\Im\bbH)^d}$, $\overline{s}_{\cI}$ est définie par les équations
\begin{equation}
u=0, \quad \omega_1=\cdots=\omega_d=0 \quad \mbox{et} \quad \varpi_i=0 \quad \mbox{pour} \; i\in\cI^c.
\label{QFB.8}\end{equation}
Pour définir l'éclatement pondéré de $\overline{s}_{\cI}$, il faut assigner un poids à chacune des variables apparaissant dans \eqref{QFB.8}.  Le multi-poids qu'on prendra sera celui qui pour $i\in\cI^c$ assigne le poids $1$ à $u$ et aux variables $\omega_i$ et $\varpi_i$ et autrement assigne le poids $\frac12$ aux variables $\omega_i$ pour $i\in\cI$.  

En termes de ces choix de multi-poids, la variété éclatée qu'il faut considérer est la variété à coins
\begin{equation}
 \widehat{\bbH^d\times (\Im\bbH)^d}:= [\overline{\bbH^d\times (\Im\bbH)^d}; (\overline{s}_\cI,w_{\cI}) \; \mbox{pour} \; \emptyset\ne\cI\subset \{1,\ldots,d\}]
\label{QFB.9}\end{equation}
obtenue en éclatant les sous-variétés $\overline{s}_{\cI}$ de façon pondérée par rapport au multi-poids $w_{\cI}$ dans un ordre compatible avec l'ordre partiel \eqref{QFB.7}.  Par rapport à \cite[\S~4]{CR}, une différence importante est que les poids assignés aux variables varient d'un éclatement à un autre.  Par exemple, si $\cI\subset \cJ$ est une inclusion stricte, alors pour $i\in \cJ\setminus \cI$, le poids de $\omega_i$ est de 1 pour l'éclatement de $\overline{s}_{\cI}$, mais devient $\frac12$ pour l'éclatement de $\overline{s}_\cJ$.  Comme dans \cite[\S~4]{CR}, il faut vérifier que la décomposition \eqref{QFB.8} induisant une trivialisation du fibré normal se relève au relèvement de $\overline{s}_{\cI}$ suite aux éclatements de $\overline{s}_{\cJ}$ pour $\overline{s}_\cJ<\overline{s}_\cI$, ce qui est garanti par \cite[Lemma~4.5]{CR}.  La suite d'éclatements pondérés de \eqref{QFB.9} est donc bien définie.

De plus, la définition de la variété $\widehat{\bbH^d\times (\Im\bbH)^d}$ ne dépend pas du choix de l'ordre des éclatements pourvu que celui-ci soit compatible avec l'ordre partiel \eqref{QFB.7}.  En effet, si $\overline{s}_{\cI}$ et $\overline{s}_\cJ$ sont deux sous-variétés qui ne sont pas reliées par la relation d'ordre, alors soit $\overline{s}_{\cI}\cap \overline{s}_{\cJ}=\emptyset$, soit  $\overline{s}_{\cI}\cap \overline{s}_{\cJ}= \overline{s}_{\cK}$ avec $\cK:= \cI\cap \cJ$ un sous-ensemble strictement contenu dans $\cI$ et $\cJ$.  Les éclatement pondérés de $\overline{s}_{\cI}$ et $\overline{s}_{\cJ}$ commutent donc, dans le premier cas parce que $\overline{s}_{\cI}$ et $\overline{s}_{\cJ}$ sont disjoints, dans le deuxième cas parce que leurs relèvements le deviennent suite à l'éclatement pondéré de $\overline{s}_{\cK}$.   
 
Loin de $\overline{s}_\cJ$ pour $\overline{s}_\cJ<\overline{s}_{\cI}$, remarquons que l'éclatement pondéré de $\overline{s}_{\cI}$ correspond à remplacer les coordonnées apparaissant dans \eqref{QFB.8} par
\begin{equation}
u, \quad p_i= \frac{\omega_i}{u}, \quad q_{\sigma,i}=\frac{\varpi_i}{u} \quad \mbox{pour} \; i\in\cI^c\quad \mbox{et} \quad \widehat{\omega}_i=\frac{\omega_i}{u^{\frac12}}=u^{\frac12}p_i, \quad \varpi_i \quad\mbox{pour} \; i\in\cI. 
\label{QFB.10}\end{equation}
Dans ces coordonnées, les équations \eqref{QFB.4} définissant $X_{\bbH^d}$ prennent la forme 
\begin{gather}
\label{QFB.11a} \mu_{\bbH}(p_i)=q_{\sigma,i}, \quad \mbox{pour} \quad i\in \cI^c, \\
\label{QFB.11b} \mu_{\bbH}(\widehat{\omega}_i)=\varpi_i, \quad \mbox{pour} \quad i\in \cI.
\end{gather}

Soit $\widehat{H}_\cI$ l'hypersurface bordante de $\widehat{\bbH^d\times (\Im\bbH)^d}$ engendrée par l'éclatement pondéré de $\overline{s}_{\cI}$.  Dénotons aussi par $\widehat{H}_{\max}$ l'hypersurface bordante correspondant au relèvement de $\pa \overline{\bbH^d\times (\Im\bbH)^d}$ à $\widehat{\bbH^d\times (\Im\bbH)^d}$.  Alors les applications de contraction des divers éclatements apparaissant dans la définition de $\widehat{\bbH^d\times (\Im\bbH)^d}$ induisent une structure de fibrés itérés $\widehat{\phi}$ sur $\widehat{\bbH^d\times (\Im\bbH)^d}$ avec fibré 
\begin{equation}
   \widehat{\phi}_\cI: \widehat{H}_{\cI}\to \widehat{S}_{\cI}
\label{QFB.12}\end{equation}
sur $\widehat{H}_\cI$ ayant pour base la variété à coins
\begin{equation}
  \widehat{S}_{\cI}:= [\overline{s}_{\cI}; \overline{s}_{\cJ} \; \mbox{pour} \; \emptyset\ne \cJ\subsetneq \cI]
\label{QFB.13}\end{equation}
obtenue de $\overline{s}_{\cI}$ en éclatant (au sens usuel, donc sans pondération) les sous-variétés $\overline{s}_{\cJ}$ pour $ \emptyset\ne \cJ\subsetneq \cI$ dans un ordre compatible avec l'ordre partiel.  En particulier, l'intérieur de $\widehat{S}_\cI$ est paramétré par $\varpi_i$ pour $i\in\cI$.  Sur $\widehat{H}_{\max}$, le fibré a pour base $\widehat{S}_{\max}=\widehat{H}_{\max}$ et correspond à l'application identité.  Remarquons qu'à l'exception de $\widehat{H}_{\max}$, la fermeture $\widehat{X}_{\bbH^d}$ de $X_{\bbH^d}$ dans $\widehat{\bbH^d\times (\Im\bbH)^d}$ coupe toutes les hypersurfaces bordantes de $\widehat{\bbH^d\times (\Im\bbH)^d}$.  Pour cette raison, l'hypersurface bordante $\widehat{H}_{\max}$ ne jouera pas de rôle important dans notre discussion.

\begin{lemme}
Le sous-ensemble $\widehat{X}_{\bbH^d}$ est une $p$-sous-variété de $\widehat{\bbH^d\times (\Im\bbH)^d}$ et la structure de fibrés itérés $\widehat{\phi}$ sur $\widehat{\bbH^d\times (\Im\bbH)^d}$ induit par restriction une structure de fibrés itérés $\widehat{\phi}|_{\pa\hat{X}_{\bbH^d}}$ sur $\widehat{X}_{\bbH^d}$.
\label{QFB.14}\end{lemme}
\begin{proof}
Les équations \eqref{QFB.11a} et \eqref{QFB.11b} se décomposent bien par rapport au fibré $\widehat{\phi}_{\cI}: \widehat{H}_{\cI}\to \widehat{S}_{\cI}$, les fonctions $\varpi_i$ pour $i\in\cI$ paramétrant (l'intérieur de) la base $\widehat{S}_\cI$, le reste des variables paramétrant les fibres.  Le fait que $\widehat{X}_{\bbH^d}$ ne coupe pas $\widehat{H}_{\max}$ vient du fait que pour $\varpi_i\ne 0$ fixé, la solution  de $\mu_{\bbH}(\widehat{\omega}_i)=\varpi_i$ dans \eqref{QFB.11b}  correspond à un cercle, plus précisément une orbite de l'action de $\bbS^1$ sur $\bbH$.  D'autre part, les équations \eqref{QFB.11a} définissent l'analogue de la variété $X_{\bbH^d}$, mais en dimension $4|\cI^c|$, où $|\cI^c|$ est la cardinalité de $\cI^c$.  

Cela suggère de procéder par induction sur $d$ pour démontrer le lemme, le cas $d=1$ étant trivial puisqu'il n'implique que l'éclatement pondéré de $\overline{s}_{\{1\}}$ et que les équations près de $\widehat{H}_{\{1\}}$ sont alors données par
$$
       \mu_{\bbH}(\widehat{\omega}_1)=\varpi_1
$$   
et définissent clairement une $p$-sous-variété.   Pour $\widehat{H}_{\cI}$ minimale, les équations \eqref{QFB.11a} et \eqref{QFB.11b} et notre hypothèse d'induction montrent que $\widehat{X}_{\bbH^d}$ est une $p$-sous-variété près de $\widehat{H}_{\cI}$.  Pour $\widehat{H}_{\cI}$ pas nécessairement minimale, on peut supposer, en procédant par induction sur la suite d'éclatements pondérés, que le résultat est déjà établi près de $\widehat{H}_{\cJ}$ pour $\widehat{H}_{\cJ}<\widehat{H}_\cI$.  Or, loin de $\widehat{H}_{\cJ}$ pour $\widehat{H}_{\cJ}<\widehat{H}_\cI$, on peut procéder comme dans le cas où $\hH_{\cI}$ est minimale pour montrer que $\widehat{X}_{\bbH^d}$ est une $p$-sous-variété près de $\hH_{\cI}$.  Ceci termine l'induction sur la suite d'éclatements pondérés, ainsi que l'induction sur $d$, montrant que $\widehat{X}_{\bbH^d}$ est bien une $p$-sous-variété de $\widehat{\bbH^d\times (\Im\bbH)^d}$.

Fort de ce résultat, il découle clairement des équations \eqref{QFB.11a} et \eqref{QFB.11b} que la structure de fibrés itérés de $\widehat{\bbH^d\times (\Im\bbH)^d}$ induit par restriction une structure de fibrés itérés sur $\widehat{X}_{\bbH^d}$ avec fibré 
\begin{equation}
  \widehat{\phi}_\cI|_{\hH_{\cI}\cap\widehat{X}_{\bbH^d}}: \hH_{\cI}\cap\widehat{X}_{\bbH^d}\to \widehat{S}_\cI
\label{fibrei.1}\end{equation}
ayant la même base $\widehat{S}_\cI$ que $\widehat{\phi}_\cI$.
\end{proof}

La fonction $u$ sur $\overline{\bbH^d\times (\Im\bbH)^d}$ se relève en une fonction bordante totale sur $\widehat{\bbH^d\times (\Im\bbH)^d}$.  Par restriction, elle induit aussi une fonction bordante totale sur $\widehat{X}_{\bbH^d}$.  
\begin{theoreme}
La déformation de Taub-NUT $g_{\TN_\sigma}$ sur $X_{\bbH^d}$ spécifiée par $\sigma$ est une métrique $\QFB$ polyhomogène sur $(\widehat{X}_{\bbH^d},\widehat{\phi}|_{\pa\widehat{X}_{\bbH^d}})$ pour l'algèbre Lie des champs vectoriels $\QFB$ associée au choix de fonction bordante totale induit par $u$.   
\label{QFB.15}\end{theoreme}
\begin{proof}
Remarquons d'abord que la restriction de la métrique euclidienne sur $\hX_{\bbH^d}$ n'est pas une métrique $\QFB$.  Cela vient du fait que dans les coordonnées \eqref{QFB.10}, elle prend la forme
$$
    \sum_{i,j}^{4d} \Psi_{ij} \nu^i\otimes\nu^j
$$
avec les coefficients $\Psi_{ij}$ polyhomogènes et bornés définissant une matrice uniformément définie positive, où
$\{\nu^1,\ldots,\nu^{4d}\}$ est une base locale de $1$-formes correspondant aux $1$-formes 
$$
  \frac{du}{u^2}, dp_i, dq_{\sigma,i} \; \mbox{pour}\; i\in\cI^c \quad \mbox{et} \quad \frac{d\widehat{\omega}_i}{u^{\frac12}}, \; \frac{d\varpi_i}{u} \; \mbox{pour} \; i\in\cI.
$$
Or, une métrique $\QFB$ polyhomogène prend plutôt la forme
$$
    \sum_{i,j}^{4d} \Psi_{ij} \widehat{\nu}^i\otimes\widehat{\nu}^j
$$
pour $\{\widehat{\nu}^1,\ldots,\widehat{\nu}^{4d}\}$ correspondant à la base locale de $1$-formes $\{\nu^1,\ldots,\nu^{4d}\}$, mais avec les $1$-formes $\frac{d\widehat{\omega}_i}{u^{\frac12}}$ pour $i\in\cI$ remplacées par $d\widehat{\omega}_i$.  

Cependant, par le Lemme~\ref{dtn.7}, la déformation de Taub-NUT a justement pour effet principal, pour $i\in \cI$,  de \og remplacer \fg $\frac{d\widehat{\omega}_i}{u^{\frac12}}$ par $d\widehat{\omega}_i$ le long des cercles définis par l'équation $\mu_{\bbH}(\widehat{\omega}_i)=\varpi_i$ lorsque $\varpi_i$ varie, du moins en des endroits où l'action de $T^d$ est libre, c'est-à-dire pour $p_i\ne 0$ lorsque $i\in\cI^c$.  La métrique $g_{\TN_\sigma}$ est donc $\QFB$ polyhomogène dans cette région.  Près de $p_i=q_{\sigma,i}=0$, on peut utiliser les coordonnées sphériques autour d'un tel point et utiliser le  Lemme~\ref{dtn.7} pour vérifier que la métrique reste uniformément $\QFB$ lorsqu'on s'approche de $p_i$.  Puisque $\hX_{\bbH^d}$ est lisse près d'un tel point et que la métrique euclidienne est polyhomogène, on en déduit que $g_{\TN_\sigma}$ est aussi une métrique $\QFB$ polyhomogène près de $p_i=q_{\sigma,i}=0$.  On peut faire des calculs similaires dans d'autres systèmes de coordonnées pour montrer que $g_{\TN_\sigma}$ est bien une métrique $\QFB$ polyhomogène sur $\hX_{\bbH^d}$.
\end{proof}

Ce théorème montre entre autres que les déformations de Taub-NUT d'ordre maximal de la métrique euclidienne sont à géométrie bornée.  Il permet aussi d'identifier le cône tangent à l'infini, le même peu importe le choix de $\sigma$.
\begin{corollaire}
La variété riemannienne $(X_{\bbH^d},g_{\TN_\sigma})$ a pour unique cône tangent à l'infini $(\Im\bbH)^d$ muni de la métrique euclidienne.
\label{QFB.16}\end{corollaire} 
\begin{proof}
Par le Théorème~\ref{QFB.15} et sa preuve, le cône tangent est unique et correspond au cône ayant pour base $\widehat{S}_{\{1,\ldots,d\}}$ avec sa métrique induite, c'est-à-dire $V_{\{1,\ldots,d\}}= (\Im\bbH)^d$ muni de la métrique euclidienne. 
\end{proof}

L'exemple le plus simple de déformation de Taub-NUT d'ordre maximal de la métrique euclidienne est sans doute obtenu lorsque $\sigma: \bbR^d\to \bbR^d\cong \mathfrak{t}^d$ correspond à l'application identité.  La métrique $g_{\TN_\sigma}$ correspond alors au produit cartésien de $d$ copies de la métrique de Taub-NUT sur $\bbH$, en quel cas le fait que $g_{\TN_\sigma}$ est une métrique $\QFB$ découle aussi du fait que la classe des métriques $\QFB$ est fermée par rapport au produit cartésien \cite[Theorem~6.6]{KR3}.  En physique mathématique, un exemple important de déformation de Taub-NUT est donné par la métrique $L^2$ de l'espace de modules des monopôles centrés de type $(1,1,\ldots,1)$, qui par \cite{LWY1996,Murray1997} et \cite[Appendix]{Kraan} correspond à une déformation de Taub-NUT d'ordre maximal de la métrique euclidienne avec l'application $\sigma:\bbR^d\to \bbR^d$ choisie de sorte que les coordonnées dans \eqref{QFB.3} sont données par 
$$
   q_{\sigma,i}= q_{i+1}-q_i
$$ 
avec $(q_1,\ldots,q_{d+1})\in (\Im\bbH)^{d+1}$ les coordonnées canoniques sur $(\Im\bbH)^{d+1}$ et avec $(q_2-q_1,\ldots, q_{d+1}-q_d)$ vues comme des coordonnées sur le quotient diagonal $(\Im\bbH)^{d+1}/(\Im\bbH)$ muni de sa métrique euclidienne induite.  

Pour déterminer si deux déformations de Taub-NUT du Théorème~\ref{QFB.15} sont isométriquement distinctes, on peut regarder leur comportement asymptotique dans les fibres du fibré \eqref{fibrei.1} avec $\cI=\{1,\ldots,d\}$, ce qui donne le résultat suivant.
\begin{corollaire}
Deux déformations de Taub-NUT $g_{\TN_{\sigma_1}}$ et $g_{\TN_{\sigma_2}}$ du Théorème~\ref{QFB.15} sont isométriquement distinctes pourvu que les applications $\sigma_1$ et $\sigma_2$, vues comme des éléments de $\GL(d,\bbR)$, définissent des classes distinctes dans le quotient 
\begin{equation}
      \GL(d,\bbZ)\setminus \GL(d,\bbR)/\operatorname{O}(d).
\label{QFB.17a}\end{equation}
\label{QFB.17}\end{corollaire} 
\begin{proof}
Si $g_{\TN_{\sigma_1}}$ et $g_{\TN_{\sigma_2}}$ sont isométriques, alors elles induisent les mêmes métriques plates dans les fibres de \eqref{fibrei.1} pour $\cI=\{1,\ldots,d\}$.  Par \eqref{base.1}, le Lemme~\ref{dtn.7} et la description des métriques plates sur le tore donnée par Wolf \cite{Wolf1973}, cela signifie que $\sigma_1$ et $\sigma_2$ engendrent la même classe dans le quotient \eqref{QFB.17a}, d'où le résultat.

\end{proof}

Ce critère confirme en particulier que la métrique $L^2$ de l'espace de modules de monopôles centrés de type $(1,1,\ldots,1)$ n'est pas isométrique à un produit cartésien de métriques de Taub-NUT.  Cela dit, comme les déformations de Taub-NUT du Théorème~\ref{QFB.15} sont toutes $\QFB$ par rapport à la même algèbre de Lie de champs vectoriels $\QFB$, elles sont toutes mutuellement quasi-isométriques, ce qui nous permet de déterminer la dimension de l'espace $\cH^k(X_{\bbH^d},g_{\TN_\sigma})$ des $k$-formes harmoniques de carré intégrable sur $(X_{\bbH^d},g_{\TN_\sigma})$.  
\begin{corollaire}
Si $(X_{\bbH^d},g_{\TN_\sigma})$ est une déformation de Taub-NUT du Théorème~\ref{QFB.15}, alors
$$
   \dim\cH^k(X_{\bbH^d},g_{\TN_\sigma})= \left\{ \begin{array}{ll} 1, & k=2d, \\ 0, & \mbox{autrement}.  \end{array}   \right.
$$
\label{QFB.18}\end{corollaire}
\begin{proof}
Comme la dimension de $\cH^k(X_{\bbH^d},g_{\TN_\sigma})$ ne dépend que de la classe de quasi-isométrie de $g_{\TN_\sigma}$, on peut supposer par la discussion précédente que $g_{\TN_\sigma}$ est isométrique à un produit cartésien de métriques de Taub-NUT,
$$
    (X_{\bbH^d},g_{\TN_\sigma})\cong (\bbH^d, \bigoplus_{i=1}^d \pr_i^*g_{\TN}),
$$
où $g_{\TN}$ est la métrique de Taub-NUT sur $\bbH$ et $\pr_i: \bbH^d\to\bbH$ est la projection sur le $i$ ième facteur.  Par \cite{Hitchin}, voir aussi \cite{HHM2004}, le résultat est vrai pour $d=1$.  
Pour $d>1$, on voit donc par séparation des variables qu'une forme harmonique $\alpha$ de carré intégrable est forcément de la forme
$$
    \alpha= (\pr_1^*\alpha_1)\wedge\cdots \wedge \pr_d^*(\alpha_d)
$$
pour $\alpha_i\in \cH^*(\bbH,g_{\TN})=\cH^2(\bbH, g_{\TN})$, d'où le résultat.
\end{proof}
Le Corollaire~\ref{QFB.18} donne en particulier une démonstration de la conjecture de $S$-dualité de Sen \cite{Sen,Gibbons1996} pour les monopôles centrés de type $(1,1,\ldots,1)$.

\bibliography{GIVHT}
\bibliographystyle{amsplain}

\small{Département de Mathématiques, Université du Québec à Montréal} 

 \small{\textit{Courriel}\! \!\!\!: rochon.frederic@uqam.ca}

\end{document}